\documentclass[12pt,reqno]{amsart}
\usepackage[margin=1in]{geometry}
\usepackage{amsmath,amssymb,amsthm,graphicx,amsxtra, setspace}
\usepackage[utf8]{inputenc}
\usepackage{mathrsfs}
\usepackage{hyperref}
\usepackage{upgreek}
\usepackage{mathtools}
\usepackage{hyperref}
\usepackage{alltt}
\usepackage{fouridx}
\usepackage{dsfont}
\usepackage{dsfont}
\usepackage{cancel}
\usepackage[dvipsnames]{xcolor}
\usepackage[mathcal]{euscript}
\usepackage{shadethm}
\usepackage{float}
\allowdisplaybreaks

\usepackage[pagewise]{lineno}

\DeclareMathAlphabet{\mathpzc}{OT1}{pzc}{m}{it}

\usepackage[cyr]{aeguill}

\colorlet{darkblue}{blue!50!black}

\hypersetup{
	colorlinks,%
	citecolor=blue,%
	filecolor=red,%
	linkcolor=red,%
	urlcolor=blue,%
	pdfnewwindow=true,%
	pdfstartview={FitH}
}

\newtheorem{theorem}{Theorem}[section]
\newtheorem{lemma}[theorem]{Lemma}

\newtheorem{definition}[theorem]{Definition}

\newtheorem{example}[theorem]{Example}
\newtheorem{remark}[theorem]{Remark}

\allowdisplaybreaks

\let\originalleft\left
\let\originalright\right
\renewcommand{\left}{\mathopen{}\mathclose\bgroup\originalleft}
\renewcommand{\right}{\aftergroup\egroup\originalright}

\let\emptyset\varnothing

\renewcommand{\d}{\/\mathrm{d}\/}

\def\I{\mathrm{I}}

\def\B{\mathcal{B}}
\def\D{\mathrm{D}}

\def\W{\mathrm{W}}

\def\H{\mathbb{H}}

\newcommand{\dive}{\operatorname{div}}
\newcommand{\supp}{\operatorname{supp}}

\newcommand{\R}{\mathbb{R}}

\renewcommand{\d}{\/\mathrm{d}\/}

\newcommand{\Addresses}{{
		\footnote{
			
			\noindent \textsuperscript{1,2,3}Department of Mathematics, Indian Institute of Technology Roorkee-IIT Roorkee,
			Haridwar Highway, Roorkee, Uttarakhand 247667, INDIA.\par\nopagebreak
			\noindent  \textit{e-mail:}    \texttt{Akram Khan: akram\_k@ma.iitr.ac.in.}
			
				\textit{e-mail:} \texttt{Sagar Gautam: sagar\_g@ma.iitr.ac.in.}
			
			\textit{e-mail:}  	\texttt{Manil T. Mohan: maniltmohan@ma.iitr.ac.in, maniltmohan@gmail.com.}

			\noindent \textsuperscript{*}Corresponding author.
			
			\textit{Key words:} Convective Brinkman-Forchheimer  equations, Lebesgue-Sobolev approximation, energy equality, Lions-Magenes lemma.

			Mathematics Subject Classification (2020): 65N30, 35Q35, 47J20, 65J15, 65N15.
}}}

\begin{document}
	
	\title[Simultaneous Sobolev-Lebesgue Approximation on Certain Unbounded Domains]{Divergence-Free Approximation in Sobolev and Lebesgue Spaces on general unbounded domains with Applications to Energy Equality in Fluid Dynamics
		\Addresses}
	\author[A. Khan, S. Gautam  and M. T. Mohan]
	{Akram Khan\textsuperscript{1}, Sagar Gautam\textsuperscript{2},   and Manil T. Mohan\textsuperscript{3*}}
	
	\maketitle

	\begin{abstract}
	We construct divergence-free, vector-valued approximation functions on the whole space $\mathbb{R}^d$, $d\geq 2$, as well as on general unbounded domains of uniform $\mathrm{C}^{1,1}$-type. These approximations converge simultaneously in both Sobolev and Lebesgue spaces. In the whole-space setting, we employ Bogovski\u{\i}\ operators to construct such approximations, thereby extending the approximation theory developed for smooth bounded domains and the simultaneous approximation framework introduced  by \emph{Fefferman, Hajduk, and Robinson,  {Proc. Lond. Math. Soc.} (3) \ {125} (2022), no.~4, 759-777}. As an application, we establish energy equality for Leray-Hopf weak solutions of the incompressible convective Brinkman-Forchheimer (CBF) equations on $\mathbb{R}^d$, $d\in\{2,3\}$ covering both the critical and supercritical regimes. For general unbounded domains, we employ the resolvent operator associated with the Stokes operator, developed by \emph{Farwig, Kozono and Sohr,  {Acta Math.}, {195} (2005), 21-53}, to obtain simultaneous approximation results. We also establish a generalized version of the classical Lions-Magenes lemma, which is of independent interest. Finally, by combining this result with the simultaneous approximation framework for unbounded domains, we establish energy equality for weak solutions of the CBF equations on general unbounded domains.
	\end{abstract}

	\section{Introduction} \setcounter{equation}{0}\label{sec-1}
	
	\subsection{The Navier-Stokes equations}
	Let $\Omega\subseteq\mathbb{R}^{d}$, $d\in\{2,3\}$, be a smooth bounded or unbounded domain. Let $0< T\leq \infty$. We consider the incompressible Navier-Stokes equations (NSE)
	\begin{equation}\label{NSE}
	\left\{
	\begin{aligned}
		\frac{\partial\boldsymbol{y}}{\partial t}-\nu \Delta\boldsymbol{y}+(\boldsymbol{y}\cdot\nabla)\boldsymbol{y}+\nabla\pi &=\boldsymbol{f}, \ \text{ in } \ (0,T)\times\Omega, \\ \nabla\cdot\boldsymbol{y}&=0, \ \text{ in } \ [0,T)\times\Omega, \\
		\boldsymbol{y}&=\mathbf{0}\ \text{ on } \ (0,T)\times\partial\Omega, \\
		\boldsymbol{y}(0)&=\boldsymbol{y}_0 \ \text{ in } \ \Omega,
	\end{aligned}
	\right.
\end{equation}
	where  $\boldsymbol{y}=\boldsymbol{y}(t,x)$ denotes the fluid velocity field, $\pi=\pi(t,x)$ is the associated pressure, and $\nu>0$ is the kinematic viscosity.
	Here the no-slip boundary condition
	$\boldsymbol{y}=0\ \text{ on }\ \partial\Omega$
	is imposed whenever $\partial\Omega\neq\emptyset$. If $\Omega$ is unbounded (e.g., $\Omega=\mathbb{R}^{d}$ or an exterior domain), we additionally require
	\begin{align*}
	\boldsymbol{y}(t,x)\to0\ \text{ as }|x|\to\infty,
	\end{align*}
	for each $t\in(0,T)$, which is the usual far-field condition. For a comprehensive treatment of the theory of the NSE, we refer the interested reader to \cite{Foias-2001,Galdi-2000,Ladyzhenskaya-1969,Rieusset-2024,Robinson-2016,SohrNSE,Temam-1984}.
	
	A cornerstone of the mathematical theory of the  three-dimensional incompressible NSE is the celebrated result of Leray and Hopf, who independently established the global-in-time existence of finite-energy weak solutions, that is, weak solutions satisfying
	\[
	\sup_{0\le t\le T}\|\boldsymbol{y}(t)\|_{\mathbb{H}}^2
	+\int_0^T\|\nabla\boldsymbol{y}(t)\|_{\mathbb{H}}^2 \d t<\infty. 
	\]   
	The question of whether every Leray-Hopf weak solution of the 3D NSE  satisfies the energy equality is one of the fundamental open problems in the mathematical theory of fluid mechanics (\cite{Robinson-2020}). 
	However, whether such Leray-Hopf weak solutions are unique remains one of the most challenging open problems in the analysis of nonlinear partial differential equations. Despite significant progress over the past decades, the problem remains open for both bounded and unbounded domains. It is well known that a sufficient condition for the validity of the energy equality is the additional regularity
	\[
	\boldsymbol{y}\in\mathrm{L}^{4}(0,T;\mathbb{L}^{4}(\Omega)).
	\]
	In two space dimensions, this condition is automatically satisfied by every Leray-Hopf weak solution as a consequence of the classical \emph{Ladyzhenskaya inequality} (\cite[Lemma 1, Chpater I]{Ladyzhenskaya-1969}), and therefore every weak solution satisfies the energy equality. 
	This criterion was established independently by Prodi (\cite{Prodi-1959}) and Lions (\cite{Lions-1959}) and is discussed in detail in the monograph of Galdi \cite{Galdi-2000} (see Serrin \cite{Serrin-1963} also).  
	
	A fundamental contribution to the theory of energy equality was made by Shinbrot~(\cite{Shinbrot-1974}), who proved that every weak Leray-Hopf solution of the NSE satisfying
	\begin{align*}
	\boldsymbol{y}\in\mathrm{L}^{p}(0,T;\mathbb{L}^{s}(\Omega)), \ 
	\frac{2}{p}+\frac{2}{s}=1,\   \ s\geq4,
	\end{align*}
	obeys the energy equality, which extends the earlier results of Prodi~\cite{Prodi-1959} and Lions~\cite{Lions-1959} (see \cite[Theorem 1.1]{Veiga-2020} for a simple proof).  It is worth noting that, in his proof of the energy equality, Shinbrot \cite{Shinbrot-1974} implicitly considers a sequence of smooth approximating functions converging simultaneously in both Sobolev and Lebesgue spaces. More precisely, he assumes the existence of a sequence
	$
	\{\boldsymbol{y}_n\}_{n\geq1}\subset \mathrm{C}_0^\infty([0,\infty);\mathbb{C}_{0,\sigma}^{\infty}(\Omega))
	$
	converging to the weak solution in $ \boldsymbol{y}\in\mathrm{L}^{\infty}(0,T;\mathbb{H})
	\cap\mathrm{L}^2(0,T;\mathbb{V})
	\cap\mathrm{L}^p(0,T;\mathbb{L}^s(\Omega))$ (see Subsection \ref{fun-set} for the function spaces).
	As indicated in the proof, the existence of such a sequence is justified by a density argument. However, no explicit construction of these approximating functions is provided. In particular, on unbounded domains, the existence of a sequence converging simultaneously in both Sobolev and Lebesgue spaces is far from immediate and requires additional analysis. The approximation results established in the present work provide precisely such a construction on the whole space and general unbounded domains.

	 Galdi \cite{Galdi-2019} subsequently showed that the Leray-Hopf regularity need not be assumed a priori; under the additional $\mathrm{L}^4(0,T;\mathbb{L}^4(\Omega))$ regularity, a very weak solution can be shown to belong to the Leray-Hopf class. 
		Further developments in the study of energy equality have established	more refined sufficient conditions in terms of the space-time	integrability of the velocity and its gradient. Berselli and Chiodaroli \cite{Berselli-2020}	obtained several energy-equality criteria involving 	$\nabla \boldsymbol{y}\in \mathrm{L}^p(0,T;\mathbb{L}^q(\Omega))$, including the condition
$$	\nabla \boldsymbol{y}\in	\mathrm{L}^{\frac{5q}{5q-6}}(0,T;\mathbb{L}^q(\Omega)),	\ 	\frac{9}{5}\le q<3.$$	
Subsequently, Beir\~ao da Veiga and Yang \cite{Veiga-2019} extended this criterion, in the Newtonian case, to the full range $q>9/5$. Their results were obtained	within a more general framework that includes non-Newtonian fluids, with the Newtonian case corresponding to the choice $r=2$. They also	introduced a unified viewpoint based on the Shinbrot number, which 	allows different integrability assumptions on $\boldsymbol{y}$ and $\nabla \boldsymbol{y}$ to be compared on a common scale. In particular, for a suitable combination	of velocity and gradient integrability, the corresponding Shinbrot	number is independent of the particular spatial exponent $q$. These	developments provide a unified perspective on the various	space-time regularity conditions that guarantee the energy equality	for Leray-Hopf weak solutions. For further details on energy equality for the 3D NSE, we refer the interested reader to \cite{Veiga-2025,Eguchi-2025,Taniuchi-1997} and the references therein.

	In the approach of~\cite[Corollary 6.3]{Veiga-2025}, it is assumed that the common test space	$\mathbb{C}_{0,\sigma}^{\infty}(\Omega)$ is dense in $\mathbb{X}\cap \mathbb{Y}$, where $\mathbb{X}$ and $\mathbb{Y}$ are Banach spaces. This density assumption, together with a	standard time-mollification argument, yields the existence of a single sequence	in $\mathrm{C}^\infty(0,t_0;\mathbb{C}_{0,\sigma}^{\infty}(\Omega))$ that approximates a given function	simultaneously in the relevant Bochner spaces	$\mathrm{L}^q(0,t_0;\mathbb{X})$ and $\mathrm{L}^p(0,t_0;\mathbb{Y})$. In the present work, rather than merely	invoking such an approximation result, we explicitly construct the required	simultaneous approximating sequence.

	Concerning the uniqueness of Leray-Hopf weak solutions to the 3D NSE, the celebrated Ladyzhenskaya-Prodi-Serrin (LPS) (\cite{Ladyzhenskaya-1967,Prodi-1959,Serrin-1962}) criterion asserts that uniqueness holds whenever
	\begin{align*}
	\boldsymbol{y}\in\mathrm{L}^{p}(0,T;\mathbb{L}^{s}(\Omega)), \ 
	\frac{2}{p}+\frac{3}{s}\le1,\  s\geq 3 \ \text{ and } p\geq2,
	\end{align*}
	and as under this condition $\boldsymbol{y}$ is smooth. It is worth mentioning that the endpoint case $\boldsymbol{y}\in\mathrm{L}^{\infty}(0,T;\mathbb{L}^{3}(\Omega)),$ which remained open for many years, was finally resolved by Escauriaza, Seregin, and \v{S}ver\'ak \cite{Iskauriaza-2003}, who proved that every Leray-Hopf weak solution satisfying this condition is smooth. 
	The authors in \cite{Farwig-2010} established a sufficient condition for energy equality of Leray-Hopf weak solutions to the 3D NSE in general unbounded  domains.

	\subsection{The convective Brinkman-Forchheimer  equations} The convective Brinkman-Forchheimer (CBF) system provides a mathematical
	framework for describing incompressible fluid motion through porous
	media. In a domain $\Omega\subseteq\mathbb{R}^d$, it takes the form (\cite{Antontsev-2010,Cai-2008,Gautam-2025,Markowich-2016})
	\begin{equation}\label{CBF}
		\left\{
		\begin{aligned}
			\frac{\partial \boldsymbol{y}}{\partial t}-\mu \Delta\boldsymbol{y}+(\boldsymbol{y}\cdot\nabla)\boldsymbol{y}+\alpha\boldsymbol{y}+\beta|\boldsymbol{y}|^{r-1}\boldsymbol{y}+\nabla\pi&=\boldsymbol{f}, \ \text{ in } \ (0,T)\times\Omega, \\ \nabla\cdot\boldsymbol{y}&=0, \ \text{ in } \ [0,T)\times\Omega, \\
			\boldsymbol{y}&=\mathbf{0}\ \text{ on } \ (0,T)\times\partial\Omega, \\
			\boldsymbol{y}(0)&=\boldsymbol{y}_0 \ \text{ in } \ \Omega.
		\end{aligned}
		\right.
	\end{equation}
	The coefficient $\mu>0$ represents the viscous diffusion and is
	referred to as the Brinkman coefficient or effective viscosity. The
	parameters $\alpha>0$ and $\beta>0$ correspond, respectively, to the
	Darcy and Forchheimer effects, describing the linear and nonlinear
	resistance generated by the porous medium. The nonlinear damping term
	is characterized by the exponent $r\in[1,\infty)$, which plays a
	central role in determining the analytical nature of the system.
		In particular, $r=3$ is identified as the critical exponent. When $\alpha=0$ and $r=3$, the CBF system exhibits the same scaling as the incompressible NSE (see \cite[Proposition 1.1]{Hajduk-2017}). Accordingly, the ranges $1\leq r<3$ and $r>3$ are commonly termed the subcritical and supercritical regimes, respectively. For the analysis of CBF equations with faster-growing Forchheimer nonlinearities, we refer to \cite{Kalantarov-2012}. Finally, in the absence of both damping terms, that is, $\alpha=\beta=0$, system \eqref{CBF} reduces to the classical incompressible NSE given in \eqref{NSE}.
	
	The existence of Leray-Hopf weak solutions to \eqref{CBF} in bounded domains for $r\geq1$, together with the energy equality for such solutions when $r\geq1$ in two dimensions and $r\geq3$ in three dimensions, has been established in \cite{Antontsev-2010,Fefferman-2022,Gautam-2025,Hajduk-2017,Kalantarov-2012,Li-2024}. The corresponding results in unbounded domains can be found in \cite{Cai-2008,Kinra-2024,Li-2024,Mohan-2021}. Uniqueness, on the other hand, holds for $r\geq1$ in two dimensions and for $r>3$ in three dimensions. In the critical three-dimensional case $d=r=3$, uniqueness continues to hold provided the additional condition $4\beta\mu\geq1$ is satisfied (see \cite{Kim-2021,Zhang-2011,Zhou-2012}). We further refer the interested reader to \cite{Kim-2017,Mohan-2025} and the references therein for results on time-periodic solutions of the 2D and 3D CBF equations in bounded domains.
	
	Regarding energy equality for Leray-Hopf weak solutions, the 2D  CBF equations on bounded domains exhibit a behavior analogous to that of the 2D NSE. In particular, the additional damping does not affect the validity of the energy equality for $r\in[1,3]$, which follows from the same regularity mechanism as in the 2D Navier-Stokes setting, since  $\mathrm{L}^{\infty}(0,T;\mathbb{H})
	\cap\mathrm{L}^2(0,T;\mathbb{V})
	\cap\mathrm{L}^{r+1}(0,T;\mathbb{L}^{r+1}_{\sigma})\hookrightarrow \mathrm{L}^4(0,T;\mathbb{L}^4_{\sigma})$, for $r\in[1,3]$. In three dimensions, however, the situation is substantially different. While the subcritical regime $1\leq r<3$ remains as challenging as for the classical 3D NSE, the critical case $r=3$ and the supercritical case $r>3$ benefit from the stronger dissipative effect of the damping, leading to improved regularity of weak solutions. Exploiting this enhanced regularity, it was established in \cite{Fefferman-2022,Gautam-2025,Hajduk-2017} that every weak solution satisfies the energy equality on bounded or periodic domains. The arguments in \cite{Fefferman-2022,Gautam-2025,Hajduk-2017} rely essentially on the spectral decomposition of the Stokes operator and the construction of divergence-free approximations using its eigenspaces. Such an approach is inherently tied to bounded or periodic domains, where the Stokes operator has a discrete spectrum. Consequently, the corresponding energy equality problem for weak solutions on general unbounded domains remains open.

	For the 3D CBF equations with supercritical damping exponent \(r\ge4\), this criterion is automatically satisfied since every weak solution belongs to $	\mathrm{L}^{r+1}(0,T;\mathbb{L}^{r+1}(\Omega)).$ Indeed, $	\frac{2}{r+1}+\frac{3}{r+1}=\frac{5}{r+1}\le1 $	for all \(r\ge4\). More remarkably, uniqueness continues to hold in the range \(3<r<4\), despite the fact that the corresponding regularity lies outside the LPS regime. Furthermore, in the critical case \(r=3\), uniqueness remains valid under the condition \(4\beta\mu\ge1\) (\cite{Gautam-2025}). These results demonstrate that the nonlinear Brinkman-Forchheimer damping yields a uniqueness theory that extends beyond the classical LPS framework. Additional uniqueness results for the NSE with damping are available in \cite{Zhang-2011,Zhou-2012}.

	\subsection{Novelties, difficulties and approaches}
	The primary objective of the present paper is to develop a simultaneous divergence-free approximation theory on the whole space and general unbounded domains. To this end, we construct divergence-free vector-valued approximation functions on  $\mathbb{R}^d, \ d\geq 2,$ and general unbounded domains of uniform $\mathrm{C}^{1,1}$-type. On $\mathbb{R}^d$, our construction combines Bogovski\u{\i} operators on bounded domains (\cite{Bogovskii-1979,Ciarlet-2025,Galdi-2011}) with smooth cutoff functions and standard mollification techniques. The resulting structure-preserving approximations converge simultaneously in Sobolev and Lebesgue spaces, thereby extending the corresponding approximation theory from bounded domains to $\mathbb{R}^d$ (Theorem \ref{main-1}). As a first application of our approximation theory, we prove that every weak solution of the incompressible CBF equations \eqref{CBF} in the critical and supercritical regimes satisfies the energy equality on $\mathbb{R}^d$ (Theorem \ref{main}). 
	

	On general unbounded domains $\Omega\subset\mathbb{R}^d$, $d\geq 2$, $\partial\Omega\neq\varnothing,$ of uniform $\mathrm{C}^{1,1}$-type, we use the resolvent operator corresponding to the Stokes operator constructed in \cite{Farwig-2005,Farwig-2009,Giga-1991,Kunstmann-2008} for simultaneous approximation (Theorem \ref{main-2}). The main ingredient in the proof is the fact that, by
	\cite[Theorem 1.4]{Farwig-2009} (see also
	\cite[Theorem 2.2]{Kunstmann-2008}), the operator
	\(\varepsilon \mathrm{I}+\widetilde{\mathcal{A}}_p\) is sectorial of type \(0\) for every
	\(\varepsilon>0\), where \(\widetilde{\mathcal{A}}_p\) denotes the Stokes operator
	on \(\widetilde{\mathbb{L}}^p_\sigma(\Omega)\).

	Another essential tool in the proof of energy equality is the Lions-Magenes lemma (\cite[Theorem 3.1, Chapter 1]{Lions-1972}, \cite[Lemma 1.2, Chapter 3]{Temam-1984}). Together with Strauss' lemma (\cite[Lemma 8.1, Chapter 3]{Lions-1972}, \cite[Theorem 2.1]{Strauss-1966}, \cite[Lemma 1.4]{Temam-1984}), it yields a classical approach for proving energy equalities for evolution equations posed within a Gelfand triple. This framework is commonly used to justify the time regularity of the energy functional and the corresponding energy identity; see, for instance, \cite[Eq. (1.5), Chapter 3]{Temam-1984}. A generalization of the Lions-Magenes lemma to the setting of two Banach spaces is stated in \cite{Chepyzhov-2002} (see \cite[Exercise 8.2]{Robinson-2001} for a special case) without proof. However, a crucial ingredient in this extension is the density of the intersection \(\mathcal{V}\cap \mathcal{E},\) where $\mathcal{V}$ and $\mathcal{E}$ are Banach spaces, in the pivot Hilbert space \(\mathcal{H}\), where
	\[
	\mathcal{V}\cap \mathcal{E} \hookrightarrow \mathcal{H} \equiv \mathcal{H}^{\prime} \hookrightarrow \mathcal{V}^{\prime}+\mathcal{E}^{\prime}
	\]
	forms a Gelfand triple. This density property is essential for proving the absolute continuity of the energy functional
	\[
	[0,T]\ni t\mapsto \|\boldsymbol{y}(t)\|_{\mathcal{H}}^{2}\in[0,\infty),
	\]
	which is the cornerstone of the Lions-Magenes argument. Motivated by this observation, we provide a complete and self-contained proof of the generalized Lions-Magenes lemma, filling in the missing details (Theorem \ref{GLML}). We believe that this result is of independent interest and may find applications in the analysis of a broader class of nonlinear evolution equations.

	Following the notation and framework of Simon~\cite[Proposition~2]{Simon-2010},
	we set
$
	E=\D(\mathcal{L}),
$ where $\mathcal{L}=\widetilde{\mathcal{A}}_p$, and $
	V=\mathbb{V}\cap\mathbb{L}^{p}_{\sigma}.
$
	Since $E\subset V$, it follows immediately that
$
	V\cap E=E.
$ Moreover, our simultaneous approximation theorem establishes that \(\D(\mathcal{L})\) is dense in \(\mathbb{V}\cap\mathbb{L}^{p}_{\sigma}\). Consequently, the density requirement of \cite[Proposition~2]{Simon-2010} is fulfilled, yielding the canonical embedding
	$
	(\mathbb{V}\cap\mathbb{L}^{p}_{\sigma})'
	\hookrightarrow
	(\D(\mathcal{L}))'.
	$
	This provides the natural distributional framework for interpreting
	\begin{align*}
	\frac{\partial\boldsymbol{y}}{\partial t}\in\mathbb{V}'+\mathbb{L}^{p'}_{\sigma} 
	\cong(\mathbb{V}\cap\mathbb{L}^{p}_{\sigma})',
	\end{align*}
	for $2<p<\infty$ in $\mathcal{D}'(0,T)$, which is precisely the setting required for our generalized Lions-Magenes lemma. Consequently, \(\D(\mathcal{L})\) serves as a common dense Hilbert subspace of both \(\mathbb{V}\) and \(\mathbb{L}^{p}_{\sigma}\), thereby providing the appropriate functional framework for applying the generalized Lions-Magenes lemma. Combining this result with the density argument described above, we establish the energy equality for weak solutions of the critical and supercritical incompressible CBF equations in both two and three dimensions on general  unbounded domains (Theorem \ref{main-3}).
	

	\subsection{Organization of the paper}
The remainder of the paper is organized as follows. In Section~\ref{sec-2}, we formulate the functional framework underlying our analysis. We introduce the Stokes operator on unbounded domains, the relevant bilinear and nonlinear operators, and the Bogovskiĭ operator, and recall the properties that will be needed in the sequel. 

Section~\ref{sec-3} is devoted to the whole-space setting. We first establish a simultaneous approximation theorem for divergence-free vector fields in Sobolev and Lebesgue spaces on $\mathbb{R}^d$, $d\geq 2$ (Theorem~\ref{main-1}). As an application, we use this approximation result to establish the energy equality for Leray-Hopf weak solutions of the critical and supercritical CBF equations in dimensions two and three (Theorem~\ref{main}).

 In Section~\ref{sec-4}, we extend the approximation result to general unbounded domains of uniform $\mathrm{C}^{1,1}$-type. The construction is based on the resolvent of the Stokes operator and yields simultaneous approximation in the relevant Sobolev and Lebesgue spaces (Theorem~\ref{main-2}). We then establish, independently, a generalized Lions-Magenes lemma (Theorem~\ref{GLML}). Combining this result with the simultaneous approximation theorem, we obtain the energy equality for Leray-Hopf weak solutions of the CBF equations \eqref{CBF} on general unbounded domains (Theorem \ref{main-3}). Finally, in Appendix \ref{Ap-chain-rule}, we provide a generalized version of the Lions-Magenes chain rule (Theorem \ref{chain-rule}).

	\section{Preliminaries}\setcounter{equation}{0} \label{sec-2}
	The aim of this section is to introduce the functional framework and recall several auxiliary results, including the Poincar\'e-Steklov inequality and Bogovski\u{\i}'s theorem, that will be used throughout the paper. 

	\subsection{Functional framework}\label{fun-set}
	We begin by recalling several standard properties of sum and
	intersection spaces that are well known from interpolation theory;
	see, e.g., \cite{Triebel-1978} (see also \cite{Farwig-2005,Farwig-2009}, etc.).
	We first recall some standard facts concerning the sum and
	intersection of Banach spaces. 
	
	Let $\mathbb{X}_1$ and $\mathbb{X}_2$ be complex Banach
	spaces continuously embedded in a common topological vector space $\mathbb{Y}$.
	Assume, in addition, that the intersection
	$\mathbb{X}_1\cap \mathbb{X}_2$ is dense in each of $\mathbb{X}_1$ and $\mathbb{X}_2$. Equipped with the
	norm
	\[
	\|\boldsymbol{v}\|_{\mathbb{X}_1\cap \mathbb{X}_2}
	:=
	\max\left\{
	\|\boldsymbol{v}\|_{\mathbb{X}_1},\|\boldsymbol{v}\|_{\mathbb{X}_2}
	\right\},
	\]
	the space $\mathbb{X}_1\cap \mathbb{X}_2$ is a Banach space. The algebraic sum of $\mathbb{X}_1$ and $\mathbb{X}_2$ is defined by
	\[
	\mathbb{X}_1+\mathbb{X}_2
	:=
	\left\{
	\boldsymbol{v}\in \mathbb{Y}:\ \boldsymbol{v}=\boldsymbol{v}_1+\boldsymbol{v}_2,\ 
	\boldsymbol{v}_1\in \mathbb{X}_1,\ \boldsymbol{v}_2\in \mathbb{X}_2
	\right\}.
	\]
	It becomes a Banach space when endowed with the norm
	\[
	\|\boldsymbol{v}\|_{\mathbb{X}_1+\mathbb{X}_2}
	:=
	\inf_{\substack{\boldsymbol{v}=\boldsymbol{v}_1+\boldsymbol{v}_2\\
			\boldsymbol{v}_1\in \mathbb{X}_1,\;\boldsymbol{v}_2\in \mathbb{X}_2}}
	\left(
	\|\boldsymbol{v}_1\|_{\mathbb{X}_1}+\|\boldsymbol{v}_2\|_{\mathbb{X}_2}
	\right).
	\]
	Moreover, if both $\mathbb{X}_1$ and $\mathbb{X}_2$ are reflexive, the above infimum is
	attained. More precisely, for every $\boldsymbol{v}\in \mathbb{X}_1+\mathbb{X}_2$, there exist
	$\boldsymbol{v}_1\in \mathbb{X}_1$ and $\boldsymbol{v}_2\in \mathbb{X}_2$ such that
	\begin{align*}
	\boldsymbol{v}=\boldsymbol{v}_1+\boldsymbol{v}_2
	\ \text{ and }\ 
	\|\boldsymbol{v}\|_{\mathbb{X}_1+\mathbb{X}_2}
	=
	\|\boldsymbol{v}_1\|_{\mathbb{X}_1}+\|\boldsymbol{v}_2\|_{\mathbb{X}_2}.
	\end{align*}
	We shall also use the standard duality relations
	\begin{align*}
	(\mathbb{X}_1\cap \mathbb{X}_2)'
	\cong
	\mathbb{X}_1'+\mathbb{X}_2'
	\ \text{ and }\ 
	(\mathbb{X}_1+\mathbb{X}_2)'
	\cong
	\mathbb{X}_1'\cap \mathbb{X}_2'.
	\end{align*}
	The first identity is understood with the canonical pairing
	\[
	\langle \boldsymbol{v},\boldsymbol{f}_1+\boldsymbol{f}_2\rangle
	=
	\langle \boldsymbol{v},\boldsymbol{f}_1\rangle+\langle \boldsymbol{v},\boldsymbol{f}_2\rangle,
	\]
	where $\boldsymbol{v}\in \mathbb{X}_1\cap \mathbb{X}_2$ and
	$\boldsymbol{f}=\boldsymbol{f}_1+\boldsymbol{f}_2\in \mathbb{X}_1'+\mathbb{X}_2'$. Similarly, for
	$\boldsymbol{v}=\boldsymbol{v}_1+\boldsymbol{v}_2\in \mathbb{X}_1+\mathbb{X}_2$ and $\boldsymbol{f}\in \mathbb{X}_1'\cap \mathbb{X}_2'$, the canonical
	pairing is given by
	\[
	\langle \boldsymbol{v},\boldsymbol{f}\rangle
	=
	\langle \boldsymbol{v}_1,\boldsymbol{f}\rangle+\langle \boldsymbol{v}_2,\boldsymbol{f}\rangle.
	\]
	Consequently, the norms of the sum and intersection spaces admit the following
	dual representations
	\[
	\|\boldsymbol{v}\|_{\mathbb{X}_1+\mathbb{X}_2}
	=
	\sup_{0\neq \boldsymbol{f}\in \mathbb{X}_1'\cap \mathbb{X}_2'}
	\frac{
		\left|
		\langle \boldsymbol{v}_1,\boldsymbol{f}\rangle+\langle \boldsymbol{v}_2,\boldsymbol{f}\rangle
		\right|
	}{
		\|\boldsymbol{f}\|_{\mathbb{X}_1'\cap \mathbb{X}_2'}
	},
	\  \boldsymbol{v}=\boldsymbol{v}_1+\boldsymbol{v}_2,
	\]
	and
	\[
	\|\boldsymbol{f}\|_{\mathbb{X}_1'\cap \mathbb{X}_2'}
	=
	\sup_{0\neq \boldsymbol{v}=\boldsymbol{v}_1+\boldsymbol{v}_2\in \mathbb{X}_1+\mathbb{X}_2}
	\frac{
		\left|
		\langle \boldsymbol{v}_1,\boldsymbol{f}\rangle+\langle \boldsymbol{v}_2,\boldsymbol{f}\rangle
		\right|
	}{
		\|\boldsymbol{v}\|_{\mathbb{X}_1+\mathbb{X}_2}
	}.
	\]
	These standard properties of sum and intersection spaces can be found,
	for example, in \cite{Triebel-1978}.
	
	Let $\mathbb{X}$ be a Banach space and let $0<T\leq\infty$. For
	$1\leq s<\infty$, we denote by $\mathrm{L}^s(0,T;\mathbb{X})$ the Bochner space of
	(equivalence classes of) strongly measurable functions
	$\boldsymbol{v}:(0,T)\to \mathbb{X}$ such that
	\begin{align*}
	\|\boldsymbol{v}\|_{\mathrm{L}^s(0,T;\mathbb{X})}
	:=
	\left(
	\int_0^T \|\boldsymbol{v}(t)\|_{\mathbb{X}}^s\d t
	\right)^{1/s}<\infty.
	\end{align*}
	If, in addition, $\mathbb{X}$ is reflexive and $1<s<\infty$, then $\bigl(\mathrm{L}^s(0,T;\mathbb{X})\bigr)'\cong\mathrm{L}^{s'}(0,T;\mathbb{X}')$, where
	$s'=\frac{s}{s-1}$ is the conjugate exponent of $s$.

	\subsection{Geometric setting of the domain}
	Let us recall the precise definition of the class of domains
	$\Omega\subset\mathbb{R}^d,$ $d\geq 2$ with $\partial\Omega\neq\emptyset$ under consideration.
	
	\begin{definition}[{\cite[Definition 1.1]{Farwig-2009}}]\label{def-c11}
	We say that a domain $\Omega\subset\mathbb{R}^d$, $d\geq 2$, is of \emph{uniform $\mathrm{C}^{1,1}$-type} if the following local flattening property holds at every boundary point. There exist constants $\alpha,\beta,K>0$, independent of the point chosen, such that for each $x_0\in\partial\Omega$ one can find a Cartesian coordinate system centered at $x_0$, with coordinates
		\begin{align*}
		y=(y',y_d),\qquad y'=(y_1,\ldots,y_{d-1}),
		\end{align*}
		and a function $h\in \mathrm{C}^{1,1}$ defined on $\{y':|y'|\leq\alpha\}$, satisfying $\|h\|_{\mathrm{C}^{1,1}}\leq K$, with the following property. If
		\begin{align*}
		U_{\alpha,\beta,h}(x_0)=
		\left\{
		y=(y',y_d)\in\mathbb{R}^d:
		|y_d-h(y')|<\beta,\ |y'|<\alpha
		\right\}
		\end{align*}
		denotes the associated neighborhood of $x_0$, then the boundary and the domain admit, respectively, the local representations
		\begin{align*}
		U_{\alpha,\beta,h}(x_0)\cap\partial\Omega&=
		\left\{(y',h(y')):|y'|<\alpha\right\},\\
		\Omega\cap U_{\alpha,\beta,h}(x_0)&=
		\left\{
		(y',y_d):
		h(y')-\beta<y_d<h(y'),\ |y'|<\alpha
		\right\}.
		\end{align*}
		This class of domains (see, e.g., \cite{Galdi-2011}) is sufficiently regular to allow for the trace and Sobolev embedding theorems used throughout the paper, while still being general enough to include bounded Lipschitz domains with a $\mathrm{C}^{1,1}$ boundary, exterior domains, and perturbed half-spaces.
	\end{definition}

	We briefly recall some standard examples of domains satisfying the
	uniform $\mathrm{C}^{1,1}$-regularity assumption used throughout this paper.
	In particular, the class of uniform $\mathrm{C}^{1,1}$-domains contains both
	bounded and unbounded domains.
	
	\begin{example}
		Besides bounded domains with $\mathrm{C}^{1,1}$ boundary, the class of uniform
		$\mathrm{C}^{1,1}$-domains contains a variety of unbounded geometries relevant
		to fluid mechanics.
	\begin{enumerate}
		\item The upper half-space
	$
		\mathbb{R}^d_+	:=
		\left\{	x=(x',x_d)\in\mathbb{R}^d:x_d>0	\right\}
	$
		is a uniform $\mathrm{C}^{1,1}$-domain. Indeed, its boundary is the hyperplane
	$
		\partial\mathbb{R}^d_+
		=
		\left\{
		x_d=0
		\right\},
	$
		which can be represented locally as the graph of the function
	$
		h(x')=0.
	$
		In particular, $h\in \mathrm{C}^\infty(\mathbb{R}^{d-1})$.
		\item For $R>0$, the exterior domain
	$
		\Omega
		=
		\mathbb{R}^n\setminus\overline{B_R(0)}
	$
		is an unbounded uniform $\mathrm{C}^{1,1}$-domain. Its boundary is the sphere
	$
		\partial\Omega=\partial B_R(0),
	$
		which is of class $\mathrm{C}^\infty$ and hence of class $\mathrm{C}^{1,1}$. More generally, if \(\mathcal{O}\subset\mathbb{R}^n\) is a bounded	domain of uniform \(\mathrm{C}^{1,1}\)-type, then its exterior domain
	$
		\Omega=\mathbb{R}^n\setminus\overline{\mathcal{O}}
	$
		is an unbounded uniform \(\mathrm{C}^{1,1}\)-domain, with
	$
		\partial\Omega=\partial\mathcal{O}.
	$
		
	\item Let $h>0$ and consider the infinite layer
$
	\Omega
	=
	\mathbb{R}^{d-1}\times(0,h).
$
	Then $\Omega\subset\mathbb{R}^d$ is an unbounded domain with nonempty
	boundary. Its boundary consists of two parallel hyperplanes:
	$
	\partial\Omega
	=
	\big(\mathbb{R}^{d-1}\times\{0\}\big)
	\cup
	\big(\mathbb{R}^{d-1}\times\{h\}\big).
$
	Since both boundary components are flat, $\partial\Omega$ is of class
	$\mathrm{C}^\infty$ and, in particular, is uniformly $\mathrm{C}^{1,1}$. Hence
	$\Omega$ is a uniform $\mathrm{C}^{1,1}$-domain.
	
		\item The infinite cylinder
	$
		\Omega
		=
		B_R^{d-1}(0)\times\mathbb{R},
		\ R>0,
	$
		is another example of an unbounded uniform $\mathrm{C}^{1,1}$-domain. Its
		boundary is smooth and can be covered by coordinate neighborhoods in
		which it is represented as the graph of a function with uniformly
		bounded $\mathrm{C}^{1,1}$-norm.
		\item More generally, let
	$
		h\in \mathrm{C}^{1,1}(\mathbb{R}^{d-1})
	$
		with a uniform $\mathrm{C}^{1,1}$ bound, and consider the epigraph
	$
		\Omega_h
		:=
		\left\{
		(x',x_d)\in\mathbb{R}^d:
		x_d>h(x')
		\right\}.
	$
		Then $\Omega_h$ is a uniform $\mathrm{C}^{1,1}$-domain whenever the defining
		function $h$ satisfies the uniform bounds required in
		Definition~\ref{def-c11}. For instance, one may take
	$
		h(x')=|x'|^2.
	$
	\end{enumerate}	
		
	\end{example}

	The class of uniform $\mathrm{C}^{1,1}$-domains therefore includes many
	standard unbounded domains with nonempty boundary, such as half-spaces,
	exterior domains, and infinite cylinders. Notice that
	$\mathbb{R}^d$ itself is not included in this definition since
$
	\partial\mathbb{R}^d=\varnothing.
$

\subsection{Function spaces}
Throughout this paper, $\Omega\subseteq\mathbb{R}^{d}$, $d\in\{2,3\}$, denotes a domain of uniform $\mathrm{C}^{1,1}$-type. We now collect the function spaces used in the sequel.

For $1<p<\infty$, the Banach space of $p$-integrable vector fields on $\Omega$ is denoted by $\mathbb{L}^p(\Omega):=\mathrm{L}^{p}(\Omega;\mathbb{R}^{d})$, equipped with the norm
\begin{align*}
	\|\boldsymbol{y}\|_{\mathbb{L}^{p}(\Omega)}=
	\left(\int_{\Omega}|\boldsymbol{y}(x)|^{p}\,\mathrm{d}x\right)^{1/p}.
\end{align*}
The endpoint case $\mathbb{L}^{\infty}(\Omega):=\mathrm{L}^{\infty}(\Omega;\mathbb{R}^{d})$ consists of the essentially bounded vector fields, with norm
\begin{align*}
	\|\boldsymbol{y}\|_{\mathbb{L}^{\infty}(\Omega)}
	=
	\operatorname*{ess\,sup}_{x\in\Omega}|\boldsymbol{y}(x)|.
\end{align*}
For $p=2$, the space $\mathbb{L}^{2}(\Omega)$ carries a natural Hilbert space structure, with inner product
\begin{align*}
	(\boldsymbol{y}_1,\boldsymbol{y}_2)
	=
	\int_{\Omega}\boldsymbol{y}_1(x)\cdot \boldsymbol{y}_2(x)\,\mathrm{d}x,
	\qquad
	\boldsymbol{y}_1,\boldsymbol{y}_2\in \mathbb{L}^{2}(\Omega).
\end{align*}

We further denote by $\mathbb{H}^1(\Omega):=\mathrm{H}^{1}(\Omega;\mathbb{R}^{d})$ the Sobolev space of vector fields in $\mathbb{L}^{2}(\Omega)$ whose weak first-order derivatives also belong to $\mathbb{L}^{2}(\Omega)$. This space becomes a Hilbert space when endowed with the inner product
\begin{align*}
	(\boldsymbol{y}_1,\boldsymbol{y}_2)_{\mathbb{H}^{1}}
	=
	(\boldsymbol{y}_1,\boldsymbol{y}_2)
	+
	(\nabla\boldsymbol{y}_1,\nabla\boldsymbol{y}_2),
\end{align*}
where the gradient inner product is defined componentwise by
\begin{equation*}
	(\nabla\boldsymbol{y}_1,\nabla\boldsymbol{y}_2)
	=
	\sum_{i=1}^{d}
	\int_{\Omega}
	\frac{\partial \boldsymbol{y}_1}{\partial x_{i}}\cdot
	\frac{\partial \boldsymbol{y}_2}{\partial x_{i}}
	\,\mathrm{d}x,
	\qquad
	\boldsymbol{y}_1,\boldsymbol{y}_2\in \mathbb{H}^{1}(\Omega).
\end{equation*}

We denote by
$\mathbb{C}_0^{\infty}(\Omega):=\mathrm{C}_{0}^{\infty}(\Omega;\mathbb{R}^{d}),
$
the space of all infinitely differentiable \(\mathbb{R}^{d}\)-valued functions with compact support in \(\Omega\).	Let 
\begin{align*}
\mathbb{C}_{0,\sigma}^{\infty}(\Omega) := \left\{
\boldsymbol{y}\in \mathbb{C}_0^{\infty}(\Omega):\nabla\cdot \boldsymbol{y}=0\right\},
\end{align*}
denote	  the space of smooth, compactly supported divergence-free vector fields. We define
\begin{align*}
\mathbb{H}
=\mathbb{L}^2_{\sigma}=
\overline{\mathbb{C}_{0,\sigma}^{\infty}(\Omega)}^{\mathbb{L}^{2}(\Omega)},
\ 
\mathbb{V}
=\mathbb{H}_{0,\sigma}^1=
\overline{\mathbb{C}_{0,\sigma}^{\infty}(\Omega)}^{\mathbb{H}^{1}(\Omega)}, \ \mathbb{L}^p_{\sigma}:= \overline{\mathbb{C}_{0,\sigma}^{\infty}(\Omega)}^{\mathbb{L}^{p}(\Omega)},
\end{align*}
for $p\in(2,\infty)$.	The space $\mathbb{H}$ inherits the inner product and norm from $\mathbb{L}^{2}(\Omega)$, namely,
\begin{align*}
(\boldsymbol{y}_1,\boldsymbol{y}_2)_{\mathbb{H}}
=
(\boldsymbol{y}_1,\boldsymbol{y}_2),
\ 
\|\boldsymbol{y}_1\|_{\mathbb{H}}
=
\|\boldsymbol{y}_1\|_{\mathbb{L}^{2}(\Omega)}.
\end{align*}
For simplicity, we denote $(\cdot,\cdot)_{\mathbb{H}}$ as $(\cdot,\cdot)$. Likewise, \(\mathbb{V}\) is endowed with the inner product
\begin{equation*}
	(\boldsymbol{y}_1,\boldsymbol{y}_2)_{\mathbb{V}}
	=
	(\boldsymbol{y}_1,\boldsymbol{y}_2) 
	+
	(\nabla\boldsymbol{y}_1,\nabla\boldsymbol{y}_2),
\end{equation*}
which induces the norm
\begin{equation*}
	\|\boldsymbol{y}_1\|_{\mathbb{V}}
	=
	\left(
	\|\boldsymbol{y}_1\|_{\mathbb{H}}^{2}
	+
	\|\nabla \boldsymbol{y}_1\|_{\mathbb{H}}^{2}
	\right)^{1/2}.
\end{equation*}
For \(p\in(2,\infty)\),   the space \(\mathbb{L}_{\sigma}^{p}\) is endowed with the norm inherited from \(\mathbb{L}^{p}(\Omega)\), namely,
\begin{align*}
\|\boldsymbol{y}\|_{\mathbb{L}_{\sigma}^{p}}
:=
\|\boldsymbol{y}\|_{\mathbb{L}^{p}(\Omega)},
\ 
\boldsymbol{y}\in \mathbb{L}_{\sigma}^{p}.
\end{align*}
We denote by $\langle\cdot,\cdot\rangle$ the duality pairing between
\(\mathbb{V}\) and \(\mathbb{V}'\), as well as between
\(\mathbb{L}^p_\sigma \) and
\(\mathbb{L}^{p'}_\sigma \), where
$
p'=\frac{p}{p-1},\ p\in(1,\infty).
$

For a subdomain $\mathfrak{D}\subset\mathbb{R}^{d},$ $d\in\{2,3\}$ and $1<q<\infty$, let
\begin{align*}
\mathrm{L}_0^q(\mathfrak{D})
:=
\left\{
f\in \mathrm{L}^q(\mathfrak{D})
:\;
\int_{\mathfrak{D}} f(x)\d x=0
\right\}
\end{align*}
denote the space of zero-mean \(\mathrm{L}^q\)-functions. We denote by
\begin{align*}
\W_0^{1,q}(\mathfrak{D})
:=
\overline{\mathrm{C}_0^\infty(\mathfrak{D})}^{\W^{1,q}(\mathfrak{D})}
\end{align*}
the Sobolev space of functions with vanishing trace on $\partial\mathfrak{D}$,  and define 
$
\mathbb{W}_0^{1,q}(\mathfrak{D})
:=
\W_0^{1,q}(\mathfrak{D};\mathbb{R}^d)
$
for the Sobolev space of \(\mathbb{R}^d\)-valued functions whose components belong to $\W_0^{1,q}(\mathfrak{D})$, endowed with its standard norm.

\subsection{The Stokes operator}
The Stokes operator constitutes a central object in the mathematical
analysis of incompressible fluid equations, in particular in the study
of the Navier-Stokes system. In this section, we consider the
Stokes operator on unbounded domains $\Omega\subset\mathbb{R}^n$ of
uniform $\mathrm{C}^{1,1}$-type. On general unbounded domains, the classical
$\mathbb{L}^q$-framework presents certain difficulties. In particular, for
$q\neq2$, the Helmholtz projection need not be bounded on
$\mathbb{L}^q(\Omega)$ unless suitable assumptions are imposed on the geometry
of $\Omega$ at infinity (see \cite{VNME} for counter examples). Consequently, a direct $\mathbb{L}^q$-theory for the
Helmholtz decomposition and the Stokes operator is not available for
arbitrary unbounded domains.

To overcome this difficulty, Farwig, Kozono, and Sohr~\cite{Farwig-2007}
introduced modified function spaces which incorporate the natural
$\mathrm{L}^2$-structure. For $q\in(1,\infty)$, define
\begin{align*}
\widetilde{\mathbb{L}}^q_\sigma 
:=
\begin{cases}
	\mathbb{L}^q_\sigma \cap \mathbb{L}^2_\sigma,
	& q\in[2,\infty),\\[2mm]
	\mathbb{L}^q_\sigma +\mathbb{L}^2_\sigma ,
	& q\in(1,2).
\end{cases}
\end{align*}
The corresponding space of gradient fields is constructed in an
analogous manner. More precisely, let
\begin{align*}
\mathbb{G}^q(\Omega)
:=
\left\{
\nabla\pi\in \mathbb{L}^q(\Omega):
\pi\in \mathrm{L}^q_{\mathrm{loc}}(\Omega)
\right\}.
\end{align*}
We then introduce
\begin{align*}
\widetilde{\mathbb{G}}^q(\Omega)
:=
\begin{cases}
	\mathbb{G}^q(\Omega)\cap \mathbb{G}^2(\Omega), & q\in[2,\infty),\\[2mm]
	\mathbb{G}^q(\Omega)+\mathbb{G}^2(\Omega), & q\in(1,2).
\end{cases}
\end{align*}
For unbounded domains of uniform $\mathrm{C}^{1,1}$-type, the Helmholtz
decomposition in these modified spaces is available for every
$q\in(1,\infty)$; namely,
$
\widetilde{\mathbb{L}}^q(\Omega)
=
\widetilde{\mathbb{L}}^q_\sigma
+
\widetilde{\mathbb{G}}^q(\Omega).
$
Moreover, the associated Helmholtz projection
$
\widetilde{\mathbb P}_q:
\widetilde{\mathbb{L}}^q(\Omega)
\to 
\widetilde{\mathbb{L}}^q_\sigma
$
is bounded, and $\mathbb{C}^\infty_{0,\sigma}(\Omega)$ is dense in
$\widetilde{\mathbb{L}}^q_\sigma$. The duality properties
$
\bigl(\widetilde{\mathbb{L}}^q_\sigma\bigr)'
=
\widetilde{\mathbb{L}}^{q'}_\sigma,
\ 
(\widetilde{\mathbb P}_q)'
=
\widetilde{\mathbb P}_{q'}
$
also hold, where $q'=\frac{q}{q-1}$ denotes the conjugate exponent of $q$.

We next recall the spaces associated with the Dirichlet Laplacian.
Let
\begin{align*}
\D^q(\Omega)
:=
\W^{2,q}(\Omega)\cap \W^{1,q}_0(\Omega)
\end{align*}
denote the domain of the Dirichlet Laplace operator $\Delta_q$ in
$\mathrm{L}^q(\Omega)$. In accordance with the preceding construction, we
define, for $q\in(1,\infty)$,
\[
\widetilde \D^q(\Omega)
:=
\begin{cases}
	\D^q(\Omega)\cap \D^2(\Omega), & q\geq2,\\[2mm]
	\D^q(\Omega)+\D^2(\Omega), & q<2,
\end{cases}
\]
and
\[
\widetilde{\W}^{1,q}_0(\Omega)
:=
\begin{cases}
	\W^{1,q}_0(\Omega)\cap \W^{1,2}_0(\Omega), & q\geq2,\\[2mm]
	\W^{1,q}_0(\Omega)+\W^{1,2}_0(\Omega), & q<2.
\end{cases}
\]

The Stokes operator on  $\widetilde{\mathbb{L}}^q_\sigma$ is then defined by (\cite{Farwig-2009})
\begin{align*}
\widetilde{\mathcal{A}}_q
:=
-\widetilde{\mathbb P}_q\Delta_q, \ \D(\widetilde{\mathcal{A}}_q)
:=
\widetilde{\mathbb{D}}^q(\Omega)
\cap
\widetilde{\mathbb{L}}^q_\sigma,
\  1<q<\infty,
\end{align*}
where $\widetilde{\mathbb{D}}^q(\Omega)=	\widetilde \D^q(\Omega;\mathbb{R}^d)$. 
It is known from \cite[Theorem 1.3]{Farwig-2009} that $-\widetilde{\mathcal{A}}_q$ generates an analytic semigroup on
$\widetilde{\mathbb{L}}^q_\sigma$ and that $\widetilde{\mathcal{A}}_q$ enjoys
maximal $\mathrm{L}^s$-regularity in these spaces for every
$s\in(1,\infty)$; see~\cite[Theorem 1.4]{Farwig-2009}. These properties
provide the functional-analytic framework required for the subsequent
analysis. 


	\begin{definition}\label{def-sect}
	Let $\omega\in[0,\pi)$. We define the closed sector of angle $\omega$ in the complex plane by
	$$
	\Sigma_\omega
	:=
	\{0\neq z\in\mathbb{C}:|\arg z|\leq\omega\}\cup\{0\}=
	\left\{
	re^{i\varphi}: r\geq0,\ |\varphi|\leq\omega
	\right\}.
	$$
	A linear operator $\mathcal{T}$ on a
	Banach space $\mathbb{X}$, with dense domain and range, is said to be \emph{sectorial
		of type $\omega$} if its spectrum contained in $\Sigma_{\omega}$, that is, 
	$$
	\sigma(\mathcal{T})\subseteq\Sigma_\omega,
	$$ where 
$\sigma(\mathcal{T}):=\left\{\lambda\in\mathbb{C}:\lambda\mathrm{I}-\mathcal{T}\ \text{ is not invertible on }\ \mathbb{X}\right\}.$	
Moreover, for every $\theta\in(\omega,\pi)$, the resolvent operators satisfy the uniform bound
\begin{align*}
	\sup\limits_{\lambda\notin\Sigma_{\theta}}\|\lambda(\lambda\I-\mathcal{T})^{-1}\|_{\mathcal{L}(\mathbb{X})}<\infty.
\end{align*}
\end{definition}

We cite the following result on the Stokes operator in domains of uniform $\mathrm{C}^{1,1}$-type (\cite[Theorem 1.3]{Farwig-2009}, \cite[Theorem 2.2]{Kunstmann-2008}).
\begin{theorem}[{\cite[Theorem 1.3]{Farwig-2009}}]
Let $\Omega\subset\mathbb{R}^d$ be a domain of uniform $\mathrm{C}^{1,1}$-type. For $q\in(1,\infty)$ and $\varepsilon>0$, the Stokes operator $\widetilde{\mathcal{A}}_q$ generates a analytic semigroup
$\bigl(e^{-t\widetilde{\mathcal{A}}_q}\bigr)_{t\geq0}$
on $\widetilde{\mathbb L}^q_\sigma$, satisfying
\begin{align*}
	\|e^{-t\widetilde{\mathcal{A}}_q}
	\boldsymbol{f}\|_{\widetilde{\mathbb{L}}^q_{\sigma}(\Omega)}
	\leq
	M e^{\varepsilon t}\|\boldsymbol{f}\|_{\widetilde{\mathbb{L}}^q_{\sigma}(\Omega)},
	\qquad
	\boldsymbol{f}\in\widetilde{\mathbb{L}}^q_\sigma(\Omega), \ t>0,
\end{align*}
where $M=M(\varepsilon,q,\alpha,\beta,K)$, and $\alpha,\beta,K$ are the constants from Definition~\ref{def-c11}. Moreover, the shifted operator $\varepsilon+\widetilde{\mathcal{A}}_q$ is sectorial of type $0$, and the duality relation
$(\widetilde{\mathcal{A}}_q)'=\widetilde{\mathcal{A}}_{q'}$
holds, where $q'$ denotes the conjugate exponent of $q$.
\end{theorem}

Let $\mathbb{H}=\mathbb{L}^2_\sigma$ and
$\mathbb{V}=\mathbb{H}^1_{0,\sigma}$. We denote by $\mathcal{A}$ the
Stokes operator on $\mathbb{H}$, which we identify with the
$\mathbb{L}^2$-realization $\widetilde{\mathcal{A}}_2$ of the Stokes operator. It is
defined by
\begin{align*}
	\mathcal{A}\boldsymbol{u}=-\mathbb{P}\Delta \boldsymbol{u},\ \boldsymbol{u}\in \mathrm{D}(\mathcal{A})=\mathbb{H}^2(\Omega)\cap\mathbb{V},
\end{align*}
where $\mathbb{P}$ denotes the Helmholtz--Leray projection onto
$\mathbb{H}$. Moreover,
$\mathrm{D}(\mathcal{A}^{1/2})=\mathbb{V},$
and $\mathcal{A}$ is a nonnegative self-adjoint operator on $\mathbb{H}$.

\subsection{Bilinear operator}\label{opeB}
For $\boldsymbol{y}_1,\boldsymbol{y}_2\in\mathbb{V}$, we define a
\emph{trilinear map} $b(\cdot,\cdot,\cdot):\mathbb{V}\times\mathbb{V}\times\mathbb{V}\to\mathbb{R}$ by
\begin{align*}
b(\boldsymbol{y}_1,\boldsymbol{y}_2,\boldsymbol{y}_3)
=\int_{\mathcal{O}}(\boldsymbol{y}_1(x)\cdot\nabla)\boldsymbol{y}_2(x)\cdot\boldsymbol{y}_3(x)\,\mathrm{d}x
=\sum_{i,j=1}^d\int_{\mathcal{O}} y_{1,i}(x)\frac{\partial y_{2,j}(x)}{\partial x_i}y_{3,j}(x)\,\mathrm{d}x .
\end{align*}
From this trilinear form one builds a bilinear operator
$\mathcal{B}:\mathbb{V}\times\mathbb{V}\to\mathbb{V}'$
by
\begin{align*}
	\langle \mathcal{B}(\boldsymbol{y}_1,\boldsymbol{y}_2),\boldsymbol{y}_3\rangle
	=
	b(\boldsymbol{y}_1,\boldsymbol{y}_2,\boldsymbol{y}_3), \ \text{ for every } \
	\boldsymbol{y}_3\in\mathbb{V}.
\end{align*}
In particular, we write
$
\mathcal{B}(\boldsymbol{y})
:=\mathcal{B}(\boldsymbol{y},\boldsymbol{y})\in\mathbb{V}'.
$
A standard integration-by-parts argument, together with the divergence-free condition and homogeneous (or periodic) boundary conditions, yields the following structural properties:
\begin{equation}\label{b0}
	\left\{
	\begin{aligned}
		b(\boldsymbol{y},\boldsymbol{z},\boldsymbol{z}) &= 0,\ \text{ for all }\ \boldsymbol{y},\boldsymbol{z} \in\mathbb{V},\\
		b(\boldsymbol{y},\boldsymbol{z},\boldsymbol{w}) &=  -b(\boldsymbol{y},\boldsymbol{w},\boldsymbol{z}),\ \text{ for all }\ \boldsymbol{y},\boldsymbol{z},\boldsymbol{w}\in \mathbb{V}.
	\end{aligned}
	\right.
\end{equation}

\subsection{Nonlinear operator}\label{opeC}
Next, we introduce the operator $\mathcal{C}_r(\boldsymbol{y}):=\mathbb{P}(|\boldsymbol{y}|^{r-1}\boldsymbol{y})$ for $r\geq1$. Throughout the rest of the paper we suppress the subscript and simply write $\mathcal{C}$ in place of $\mathcal{C}_r$. A direct calculation shows that
$
\langle\mathcal{C}(\boldsymbol{y}),\boldsymbol{y}\rangle =\|\boldsymbol{y}\|_{\mathbb{L}^{r+1}_{\sigma}}^{r+1}.
$
Moreover, following the argument of \cite[Section 2.4]{Gautam-2025}, the operator $\mathcal{C}$ satisfies the following monotonicity-type estimate:
\begin{align}\label{2.23}
	\langle\mathcal{C}(\boldsymbol{y})-\mathcal{C}(\boldsymbol{z}),
	\boldsymbol{y}-\boldsymbol{z}\rangle
	&\geq
	\frac{1}{2}\big\||\boldsymbol{y}|^{\frac{r-1}{2}}(\boldsymbol{y}-\boldsymbol{z})\big\|_{\mathbb{H}}^2
	+\frac{1}{2}\big\||\boldsymbol{z}|^{\frac{r-1}{2}}(\boldsymbol{y}-
	\boldsymbol{z})\big\|_{\mathbb{H}}^2 
	\nonumber\\&\geq
	\frac{1}{2^{r-1}}\|\boldsymbol{y}-
	\boldsymbol{z}\|_{\mathbb{L}_{\sigma}^{r+1}}^{r+1}\geq0,
\end{align}
for every $r\geq 1$.

\subsection{Abstract formulation and weak solution to the CBF system \eqref{CBF}}
Applying the Helmholtz-Hodge orthogonal projection $\mathbb{P}$ to \eqref{CBF}, we arrive at the abstract formulation
\begin{equation}\label{abstract-CBF}
	\left\{
	\begin{aligned}
		\frac{\mathrm{d}\boldsymbol{y}(t)}{\mathrm{d}t}+\mu \mathcal{A}\boldsymbol{y}(t) +\mathcal{B}(\boldsymbol{y}(t)) +\alpha \boldsymbol{y}(t)+ \beta\mathcal{C}(\boldsymbol{y}(t))&=\boldsymbol{g}(t),\\
		\boldsymbol{y}(0)&=\boldsymbol{y}_0\in\mathbb{H},
	\end{aligned}
	\right.
\end{equation}
for a.e.\ $t\in[0,T]$, where $\boldsymbol{g}:=\mathbb{P}\boldsymbol{f}\in \mathrm{L}^2(0,T;\mathbb{V}')$.

To formulate the notion of weak solution, we first introduce the relevant space of test functions. Set
$
\mathcal{D}_{\sigma}([0,T)\times\Omega):=\mathrm{C}_0^{\infty}([0,T);\mathbb{C}_0^{\infty}(\Omega)),
$
and define the space of divergence-free space-time test functions
\begin{align}\label{dvtest}
	\mathcal{V}_T:=\left\{\boldsymbol{\psi}\in \mathcal{D}_{\sigma}([0,T)\times\Omega):\nabla\cdot\boldsymbol{\psi}(\cdot,t)=0\right\}.
\end{align}
Since every $\boldsymbol{\psi}\in\mathcal{V}_T$ has compact support in $[0,T)$, it follows in particular that $\boldsymbol{\psi}(x,T)=0$.

With this notation in hand, we now state the definition of a \emph{weak solution} to \eqref{abstract-CBF}, for a fixed parameter $r\in[1,\infty)$. To this end, we introduce the exponent
\begin{align}\label{eqn-pd}
	p_d=\left\{\begin{array}{cc} 2&\text{ for }\ d=2, \ r\in[1,\infty), \\ \dfrac{4}{3}&\text{ for } \ d=3, \ r\in[1,3],\\
		2&\text{ for } \ d=3, \ r\in[3,\infty), \end{array}\right.
\end{align}
which will govern the integrability of the time derivative in the weak formulation below.

\begin{definition}\label{weakd}
	For $r\in[1,\infty)$,	a function  $$\boldsymbol{y}\in \mathrm{L}^{\infty}(0,T;\H)\cap \mathrm{L}^2(0,T;\mathbb{V})\cap \mathrm{L}^{r+1}(0,T; \mathbb{L}_{\sigma}^{r+1})),$$  with $\frac{\d\boldsymbol{y}}{\d t}\in \mathrm{L}^{p_d}(0,T;\mathbb{V}')+ \mathrm{L}^{\frac{r+1}{r}}(0,T;\mathbb{L}_{\sigma}^{\frac{r+1}{r}}),$ is called a \emph{weak solution} to the system (\ref{abstract-CBF}), if for every $\boldsymbol{g}\in \mathrm{L}^2(0,T;\mathbb{V}')$ and  $\boldsymbol{y}_0\in\H$ it satisfies
	\begin{align}\label{3.13}
		&-\int_{t_0}^{t_1}{(\boldsymbol{y}(s),\partial_t\boldsymbol{\psi}(s))}\d s+ \mu\int_{t_0}^{t_1}(\nabla\boldsymbol{y}(s),\nabla\boldsymbol{\psi}(s))\d s+\int_{t_0}^{t_1}\langle(\boldsymbol{y}(s)\cdot\nabla)\boldsymbol{y}(s),\boldsymbol{\psi}(s)\rangle\d s\nonumber\\&\quad+\alpha\int_{t_0}^{t_1}(\boldsymbol{y}(s),\boldsymbol{\psi}(s))\d s+\beta\int_{t_0}^{t_1}\langle\boldsymbol{y}(s)|\boldsymbol{y}(s)|^{r-1},\boldsymbol{\psi}(s)\rangle\d s\nonumber\\&= -{(\boldsymbol{y}(t_1),\boldsymbol{\psi}(t_1))}+{(\boldsymbol{y}(t_0),\boldsymbol{\psi}(t_0))}+\int_{t_0}^{t_1}\langle \boldsymbol{g}(s),\boldsymbol{\psi}(s)\rangle\d s,
	\end{align}
	for all test functions $\boldsymbol{\psi}\in\mathcal{V}_T$, almost all initial times $t_0\in[0,T)$, including zero, and almost every $t_1\in(t_0,T)$. A function $\boldsymbol{y}$ is called a \emph{global weak solution} if it is a weak solution for all $T>0$.
\end{definition}
\begin{definition}
	We say that $\boldsymbol{u}$ is a \emph{Leray-Hopf weak solution} of the CBF equations \eqref{abstract-CBF}, with initial data $\boldsymbol{u}_0\in\mathbb{H}$, if it is a weak solution in the sense of Definition \ref{weakd} that additionally satisfies the following properties:
	\begin{enumerate}
		\item  \emph{strong energy inequality}:
		\begin{align}\label{energy-inequality}
			&	\|\boldsymbol{y}(t_1)\|_{\H}^2+2\mu\int_{t_0}^{t_1}\|\nabla\boldsymbol{y}(s)\|_{\mathbb{H}}^2\d s+2\alpha\int_{t_0}^{t_1}\|\boldsymbol{y}(s)\|_{\mathbb{H}}^{2}\d s+2\beta\int_{t_0}^{t_1}\|\boldsymbol{y}(s)\|_{\mathbb{L}_{\sigma}^{r+1}}^{r+1}\d s\nonumber\\&\leq\|\boldsymbol{y}(t_0)\|_{\H}^2+2\int_{t_0}^{t_1}\langle\boldsymbol{g}(s),\boldsymbol{y}(s)\rangle\d s,
		\end{align}
		for almost every $t_0\in[0, T)$, including zero, and all $t_1\in(t_0, T)$.
		\item $\lim\limits_{t\downarrow 0}\|\boldsymbol{y}(t)-\boldsymbol{y}_0\|_{\mathbb{H}}=0$.
	\end{enumerate}
	
\end{definition}

\begin{remark}\label{rem-main}
We know that $\boldsymbol{y}\in \mathrm{L}^{\infty}(0,T;\mathbb{H})\subset \mathrm{L}^{\frac{r+1}{r}}(0,T;\mathbb{V}'+\mathbb{L}^{\frac{r+1}{r}}_{\sigma})$ and $\frac{\mathrm{d} \boldsymbol{y}}{\mathrm{d}t}\in  \mathrm{L}^{\frac{r+1}{r}}(0,T;\mathbb{V}'+\mathbb{L}^{\frac{r+1}{r}}_{\sigma})$; together these imply $\boldsymbol{y}\in \mathrm{W}^{1,\frac{r+1}{r}}(0,T;\mathbb{V}'+\mathbb{L}^{\frac{r+1}{r}}_{\sigma})\subset \mathrm{C}([0,T];\mathbb{V}'+\mathbb{L}^{\frac{r+1}{r}}_{\sigma})$, and the embedding $\mathbb{V}'+\mathbb{L}^{\frac{r+1}{r}}_{\sigma}\hookrightarrow \mathbb{H}$ is continuous. An application of Strauss' lemma (\cite[Lemma 8.1, Chapter 3]{Lions-1972}, \cite[Theorem 2.1]{Strauss-1966}, \cite[Lemma 1.4]{Temam-1984}, \cite[Proposition 1.7.1, Chapter 1, p.~61]{Cherier -2012}) then gives $\boldsymbol{y}\in\mathrm{C}_w([0,T];\mathbb{H})$, that is, $\boldsymbol{y}$ is weakly continuous in time with values in $\mathbb{H}$. More precisely, this means that for every fixed $\boldsymbol{\phi}\in\mathbb{H}$, the scalar function $[0,T]\ni t\mapsto (\boldsymbol{y}(t),\boldsymbol{\phi})\in\mathbb{R}$ is continuous on $[0,T]$. Consequently, the first two terms on the right-hand side of \eqref{3.13} are well defined.

An analogous argument to that in \cite[Lemma 2.4]{Galdi-2011} shows that if $\boldsymbol{y}$ is a weak solution of the system \eqref{abstract-CBF}, then \eqref{3.13} can be replaced by
\begin{align}\label{3.13-1}
	&{(\boldsymbol{y}(t_1),\boldsymbol{\psi})} + \mu\int_{t_0}^{t_1}(\nabla\boldsymbol{y}(s),\nabla\boldsymbol{\psi})\,\mathrm{d}s+\int_{t_0}^{t_1}\langle(\boldsymbol{y}(s)\cdot\nabla)\boldsymbol{y}(s),\boldsymbol{\psi}\rangle\,\mathrm{d}s+\alpha\int_{t_0}^{t_1}
	(\boldsymbol{y}(s),\boldsymbol{\psi})\,\mathrm{d}s \nonumber\\&\quad+\beta\int_{t_0}^{t_1}\langle\boldsymbol{y}(s)|\boldsymbol{y}(s)|^{r-1},\boldsymbol{\psi}\rangle\,\mathrm{d}s = {(\boldsymbol{y}(t_0),\boldsymbol{\psi})}+\int_{t_0}^{t_1}\langle \boldsymbol{g}(s),\boldsymbol{\psi}\rangle\,\mathrm{d}s,
\end{align}
for all $t_0\in[0,T)$, including zero, and all $t_1\in(t_0, T)$, and $\boldsymbol{\psi}\in\mathbb{C}_{0,\sigma}^{\infty}(\Omega)$. Moreover, if $\mathbb{C}_{0,\sigma}^{\infty}(\Omega)$ is dense in
$\mathbb{V}\cap \mathbb{L}_{\sigma}^{r+1}$, the above equality extends to
every $\boldsymbol{\psi}\in \mathbb{V}\cap \mathbb{L}_{\sigma}^{r+1}$.
\end{remark}

	\subsection{Some basic results on Bogovski\u{\i} operator}
	
	The following auxiliary results play an important role in our analysis and are included here for the reader's convenience.

	\begin{lemma}[Poincar\'e-Steklov inequality, {\cite[Corollary 9.19]{Brezis-2011}, \cite[Lemma 3.27]{Ern-2021}}]\label{lem:poincare}
		Let $\mathfrak{D}\subset\mathbb{R}^d$ be a bounded Lipschitz domain and let
		\(1\le q<\infty\). Then there exists a constant
		$C=C(\mathfrak{D},q)>0$ such that
		\begin{align}\label{eqn-poin}
			\|\boldsymbol{v}\|_{\mathbb{L}^q(\mathfrak{D})}
			\leq
			C \operatorname{diam}(\mathfrak{D})\,
			\|\nabla \boldsymbol{v}\|_{\mathbb{L}^q(\mathfrak{D})}
		\end{align}
		for every $\boldsymbol{v}\in \mathbb{W}_0^{1,q}(\mathfrak{D}),$ where $\operatorname{diam}(\mathfrak{D})$ denotes the diameter of $\mathfrak{D}$, defined by
		$
		\operatorname{diam}(\mathfrak{D}):=
		\sup\{|x-y|:\ x,y\in \mathfrak{D}\}.
		$
	\end{lemma}
	
	A proof of the following result is available in \cite[Theorem 1]{Bogovskii-1979}, \cite[Lemma III.3.1]{Galdi-2011} and \cite[Lemma 2.1.1, Chapter II]{SohrNSE}.
	
	\begin{lemma}[Bogovski\u{\i}'s theorem]\label{lem:bogovskii}
		Let $\mathfrak{D}\subset\mathbb{R}^d$ be a bounded Lipschitz domain and let
		$1<q<\infty$. Then there exists a bounded linear operator
		\[
		\mathcal{B}_{\mathfrak{D}}:\mathrm{L}_0^q(\mathfrak{D})\to \mathbb{W}_0^{1,q}(\mathfrak{D})
		\]
		such that
		\[
		\nabla\cdot(\mathcal{B}_{\mathfrak{D}}f)=f
		\ \text{ a.e. in }\ \mathfrak{D},
		\]
		for every $f\in\mathrm{L}_0^q(\mathfrak{D})$. Moreover, there exists a constant
		\(C=C(\mathfrak{D},q)>0\) such that
		\begin{align}\label{eqn-Bogo}
			\|\nabla(\mathcal{B}_{\mathfrak{D}}f)\|_{\mathbb{L}^q(\mathfrak{D})}
			\leq
			C\|f\|_{\mathrm{L}^q(\mathfrak{D})},
		\end{align}
		and consequently, by the Poincar\'e inequality,
		\begin{align*}
		\|\mathcal{B}_{\mathfrak{D}}f\|_{\mathbb{W}^{1,q}(\mathfrak{D})}
		\le
		C\,\|f\|_{\mathrm{L}^q(\mathfrak{D})}.
		\end{align*}
	\end{lemma}
	
	\begin{remark}
		It is worth mentioning that the existence part of Lemma~\ref{lem:bogovskii}, namely the estimate prior to the application of the Poincar\'e inequality, remains valid for arbitrary domains \(\Omega\subseteq\mathbb{R}^d\); see, for instance, \cite[Exercise~III.3.1]{Galdi-2011} or \cite[Theorem~8.12-1]{Ciarlet-2025}. However, the constant appearing in \eqref{eqn-Bogo} generally depends on the geometry of the domain \(\Omega\).
	\end{remark}

	\begin{lemma}[Dilation of the Bogovski\u{\i}\ operator]\label{lem:scaling}
		Let $\mathfrak{D}\subset\R^d$ be a fixed bounded Lipschitz domain and let $\mathcal B_{\mathfrak{D}}$ be
		its Bogovski\u{\i}\ operator (Lemma~\ref{lem:bogovskii}). For $n>0$ set
		$\mathfrak{D}_n:=n\mathfrak{D}=\{nx:x\in\mathfrak{D}\}$. For $f\in \mathrm{L}^q_0(\mathfrak{D}_n)$, we define 
		\[
		g(y):=f(ny), \ y\in \mathfrak{D}  \ \text{ and } \ 
		(\mathcal B_{\mathfrak{D}_n}f)(x):=n\,(\mathcal B_\mathfrak{D} g)(x/n),\ \ x\in \mathfrak{D}_n.
		\]
		Then $\mathcal B_{\mathfrak{D}_n}f\in\mathbb{W}_0^{1,q}(\mathfrak{D}_n)^d$, $\dive(\mathcal B_{\mathfrak{D}_n}f)=f$, and
		there is a constant $C=C(\mathfrak{D},q)$, \emph{independent of $n$}, such that
		\begin{align}\label{eqn-bogo}
		\|\nabla(\mathcal B_{\mathfrak{D}_n}f)\|_{\mathrm{L}^q(\mathfrak{D}_n)}\le C\|f\|_{\mathrm{L}^q(\mathfrak{D}_n)}
		\ \text{ and }\ 
		\|\mathcal B_{\mathfrak{D}_n}f\|_{\mathrm{L}^q(\mathfrak{D}_n)}\le Cn\|f\|_{\mathrm{L}^q(\mathfrak{D}_n)} .
		\end{align}
	\end{lemma}
	
	\begin{proof}
		First, we verify that $g\in\mathrm{L}_0^q(\mathfrak{D})$. Indeed, by the change of variables $x=ny$,
		\begin{align*}
		\int_{\mathfrak{D}} g(y)\,\d y=\int_{\mathfrak{D}} f(ny)\,\d y
		=n^{-d}\int_{\mathfrak{D}_n}f(x)\,\d x=0,
		\end{align*}
		since $f\in\mathrm{L}_0^q(\mathfrak{D}_n)$. Hence, $g$ has zero mean over $\mathfrak{D}$, and therefore
		$g\in\mathrm{L}_0^q(\mathfrak{D})$. 
		Let us denote  $\boldsymbol{v}:=\mathcal B_{\mathfrak{D}} g\in \mathbb{W}_0^{1,q}(\mathfrak{D})$ and $\boldsymbol{w}(x):=n\boldsymbol{v}(x/n)$. By the
		chain rule, for $x=ny\in\mathfrak{D}_n$, we have 
		\begin{align*}
		\partial_{x_i}w_j(x)&=n\cdot \partial_{y_i}v_j(y)\cdot\frac1n
		=\partial_{y_i}v_j(y)
		\\  \text{ and }
		\dive_x\boldsymbol{w}(x)&=\dive_y \boldsymbol{v}(y)=g(y)=f(ny)=f(x).
		\end{align*}
		So $\dive\boldsymbol{w}=f$ on $\mathfrak{D}_n$, and $\boldsymbol{w}$ has zero trace on $\partial\mathfrak{D}_n$, since $\boldsymbol{v}$
		has zero trace on $\partial\mathfrak{D}$, that is, $\boldsymbol{w}\in \mathbb{W}_0^{1,q}(\mathfrak{D}_n)$.
		
		For the estimate of the gradient, we perform the change of variables
		$x=ny$, with Jacobian $\d x=n^d\,\d y$, to obtain
		\begin{align*}
		\|\nabla\boldsymbol{w}\|_{\mathbb{L}^q(\mathfrak{D}_n)}^q=\int_{\mathfrak{D}_n}|\nabla_x\boldsymbol{w}(x)|^q\d x
		=\int_\mathfrak{D} |\nabla_y \boldsymbol{v}(y)|^q n^d\d y = n^d\|\nabla\boldsymbol{v}\|_{\mathbb{L}^q(\mathfrak{D})}^q,
		\end{align*}
		so that $\|\nabla\boldsymbol{w}\|_{\mathrm{L}^q(\mathfrak{D}_n)}=n^{d/q}\|\nabla \boldsymbol{v}\|_{\mathbb{L}^q(\mathfrak{D})}$. Similarly, we have 
		\[
		\|f\|_{\mathrm{L}^q(\mathfrak{D}_n)}^q=\int_{\mathfrak{D}_n}|f(x)|^q dx=\int_D|g(y)|^q n^d dy
		= n^d\|g\|_{\mathrm{L}^q(\mathfrak{D})}^q,
		\]
		so that $\|f\|_{\mathrm{L}^q(D_n)}=n^{d/q}\|g\|_{\mathrm{L}^q(\mathfrak{D})}$. Combining with the Bogovski\u{\i}\
		estimate \eqref{eqn-Bogo}, that is, $\|\nabla \boldsymbol{v}\|_{\mathbb{L}^q(\mathfrak{D})}\le C(\mathfrak{D},q)\|g\|_{\mathrm{L}^q(\mathfrak{D})}$, we arrive at 
		\begin{align*}
		\|\nabla\boldsymbol{w}\|_{\mathbb{L}^q(D_n)}=n^{d/q}\|\nabla \boldsymbol{v}\|_{\mathbb{L}^q(\mathfrak{D})}
		\leq C(\mathfrak{D},q)\,n^{d/q}\|g\|_{\mathrm{L}^q(\mathfrak{D})}=C(\mathfrak{D},q)\|f\|_{\mathrm{L}^q(\mathfrak{D}_n)},
		\end{align*}
		which is the first estimate, with the same constant as on the fixed domain
		$\mathfrak{D}$, hence independent of $n$.
		
		For the second estimate, we compute similarly
		\begin{align*}
		\|\boldsymbol{w}\|_{\mathbb{L}^q(\mathfrak{D}_n)}^q=\int_{\mathfrak{D}_n}|\boldsymbol{w}(x)|^q\d x
		=\int_{\mathfrak{D}} |n\boldsymbol{v}(y)|^q n^d\d y = n^{q+d}\|\boldsymbol{v}\|_{\mathbb{L}^q(\mathfrak{D})}^q,
		\end{align*}
		so that $\|\boldsymbol{w}\|_{\mathbb{L}^q(\mathfrak{D}_n)}=n^{1+d/q}\|\boldsymbol{v}\|_{\mathbb{L}^q(\mathfrak{D})}$. Applying Lemma \ref{lem:poincare} on the fixed domain \(\mathfrak{D}\) (cf. \eqref{eqn-poin}), we obtain
		\[
		\|\boldsymbol{v}\|_{\mathbb{L}^q(\mathfrak{D})}
		\le
		C_d \operatorname{diam}(\mathfrak{D})
		\|\nabla\boldsymbol{v}\|_{\mathbb{L}^q(\mathfrak{D})}
		\le
		C(\mathfrak{D},q)\,\|g\|_{\mathrm{L}^q(\mathfrak{D})}.
		\]
		Using the scaling relation \(\boldsymbol{w}(x)=n\boldsymbol{v}(x/n)\), it follows that
		\begin{align*}
		\|\boldsymbol{w}\|_{\mathbb{L}^q(\mathfrak{D}_n)}
		\le
		C(\mathfrak{D},q)n^{1+d/q}\|g\|_{\mathrm{L}^q(\mathfrak{D})}
		=
		C(\mathfrak{D},q)n\|f\|_{\mathrm{L}^q(\mathfrak{D}_n)},
		\end{align*}
		which completes the proof.
	\end{proof}

	\section{Simultaneous approximation and energy equality for CBF equations in $\mathbb{R}^d$}\setcounter{equation}{0}\label{sec-3}

In this section, we first establish a simultaneous approximation result
for divergence-free vector fields in Sobolev and Lebesgue spaces on
$\mathbb{R}^d$. We then apply this approximation to prove that every 
Leray-Hopf weak solution of the two- and three-dimensional critical
and supercritical CBF equations \eqref{abstract-CBF} satisfy the energy equality.
	
	\subsection{Simultaneous approximation on $\mathbb{R}^d$}
	Let us first provide the simultaneous approximation result on $\mathbb{R}^d$.

	\begin{theorem}\label{main-1}
		Let $\Omega=\mathbb{R}^d$ and $\boldsymbol{y}\in\mathbb{V}\cap\mathbb{L}^p_{\sigma}$. Then, there exists a sequence $\{\boldsymbol{y}_n\}_{n\geq1}$ in $\mathbb{C}_{0,\sigma}^{\infty}(\Omega)$ such that 
		\begin{align}\label{eqn-conv}
			\|\boldsymbol{y}_n-\boldsymbol{y}\|_{\mathbb{V}}+\|\boldsymbol{y}_n-\boldsymbol{y}\|_{\mathbb{L}^p_{\sigma}}\to 0\ \text{ as }\ n\to\infty. 
		\end{align}
		Equivalently, $\mathbb{C}_{0,\sigma}^{\infty}(\Omega)$ is dense in $\mathbb{V}\cap\mathbb{L}^p_{\sigma}$ for
		the norm $\|\cdot\|_{\mathbb{V}}+\|\cdot\|_{\mathbb{L}^p_{\sigma}}$.
	\end{theorem}
	\begin{proof}  We divide the proof into the following steps:
		\vskip 0.1 cm
		\noindent \textbf{Step 1:} \emph{A smooth approximation in $\mathbb{H}^1(\R^d)\cap \mathbb{L}^p(\R^d)$.} Let us define $\eta:\R\to\R$ by
		\begin{align*}
		\eta(t):=\begin{cases} e^{-1/t}, & t>0,\\[2pt] 0, & t\le 0. \end{cases}
		\end{align*}
		It is a standard fact that $\eta\in \mathrm{C}^\infty(\R)$. Define
		\begin{align*}
		\psi(t):=\frac{\eta(2-t)}{\eta(2-t)+\eta(t-1)}, \ \ t\in\R .
		\end{align*}
		Then $\psi\in\mathrm{C}^\infty(\R)$ and satisfies
		\begin{align*}
		0\leq\psi\leq1, \ 
		\psi(t)=\begin{cases} 1, & t\leq1,\\[2pt] 0, & t\geq2. \end{cases}
		\end{align*}
Now define the radial cut-off function
		\begin{align*}
		\chi(x):=\psi(|x|), \  x\in\R^d .
		\end{align*}
	For $n>0$, let $B_n:=B_n(0)
	=\left\{x\in\mathbb{R}^d:\lvert x\rvert<n\right\}$ be an open ball in $\mathbb{R}^d$.	Although the function $x\mapsto|x|$ is not smooth at the origin, this causes no difficulty. Indeed, since $\psi\equiv1$ on $(-\infty,1]$, $\chi$ is identically $1$ on the open ball $B_1$. Thus, $\chi$ is constant, and hence smooth, in a neighborhood of the origin. On $\R^d\setminus\{0\}$, both $x\mapsto|x|$ and $\psi$ are smooth. Consequently, $\chi\in \mathrm{C}^\infty(\R^d)$, and  
		\begin{align*}
		0\le\chi\le1, \  \chi\equiv1 \text{ on } B_1, \  \chi\equiv0 \text{ outside } \overline{B}_2.
		\end{align*}
		For each $n\in\mathbb{N},$ set $\chi_n(x):=\psi(|x|/n)$, $x\in\R^d$. Then
		\begin{align*}
		\chi_n\in\mathrm{C}_0^\infty(B_{2n}), \ 0\le\chi_n\le1,\ \chi_n\equiv1 \ \text{ on } \ B_n.
		\end{align*} 
		and
		\begin{align*}
		|\nabla\chi_n(x)| \le \frac{C}{n}, \  \text{ where }\ C=\|\psi'\|_{\mathrm{L}^\infty}<\infty.
		\end{align*}
		For $\mathbb{H}^1(\R^d)\cap \mathbb{L}^p(\R^d)$, define
		\begin{align*}
		\boldsymbol{v}_n:=\chi_n \boldsymbol{y}.
		\end{align*}
		Since $\boldsymbol{y}\in \mathbb{H}^1(\R^d)$ and $\chi_n\in \mathrm{C}_0^\infty(\mathbb{R}^d)$, the product rule gives
		$\boldsymbol{v}_n\in \mathbb{H}^1(\R^d)$, with $\supp(\boldsymbol{v}_n)\subset B_{2n}$. Since $\boldsymbol{y}\in \mathbb{L}^p(\R^d)$, it is immediate that $\boldsymbol{v}_n\in  \mathbb{L}^p(\R^d)$.  
		
		Let us now show that $\boldsymbol{v}_n\to \boldsymbol{y}$ in $\mathbb{H}^1(\R^d)\cap \mathbb{L}^p(\R^d)$ as $n\to\infty$.  Since $\boldsymbol{y} \in \mathbb{L}^p(\R^d)$, we have
		\begin{align*}
		\|\boldsymbol{y}-\boldsymbol{v}_n\|_{\mathbb{L}^p(\mathbb{R}^d)}^p
		=\int_{\R^d}|1-\chi_n|^p|\boldsymbol{y}(x)|^p\d x .
		\end{align*}
		For every fixed $x\in\R^d$, we have $1-\chi_n(x)\to 0$ as $n\to\infty$ pointwise, because
		$\chi_n\equiv1$ on $B_n\uparrow\R^d$). Moreover, 
		\begin{align*}
	0\le |1-\chi_n|^p|\boldsymbol{y}|^p\le |\boldsymbol{y}|^p\in
	\mathbb{L}^1(\R^d) \ \text{ and }
	|\boldsymbol{y}|^p\in\mathrm{L}^1(\R^d).
		\end{align*} 
		 Therefore, the Lebesgue dominated convergence theorem yields $\|\boldsymbol{y}-\boldsymbol{v}_n\|_{\mathbb{L}^p(\mathbb{R}^d)}\to0$ as $n\to\infty$.
		Alternatively, since $1-\chi_n\equiv0$ on $B_n$, we directly obtain 
		\begin{align*}
		\|\boldsymbol{y}-\boldsymbol{v}_n\|_{\mathbb{L}^p(\mathbb{R}^d)}^p \le \int_{\R^d\setminus B_n}|\boldsymbol{y}(x)|^p\d x \to  0,
		\end{align*}
		as $n\to\infty$, being the tail of a convergent integral. By the product rule, we obtain 
		\begin{align*}
		\nabla(\boldsymbol{v}_n-\boldsymbol{y})=(\chi_n-1)\nabla \boldsymbol{y} + \boldsymbol{y} \nabla\chi_n .
		\end{align*}
		The first term satisfies $|(\chi_n-1)\nabla \boldsymbol{y}|\le|\nabla \boldsymbol{y}|\in \mathbb{L}^2(\R^d)$ and converges to $0$ pointwise, so by the Lebesgue dominated convergence theorem,
		$\|(\chi_n-1)\nabla \boldsymbol{y}\|_{\mathbb{L}^2(\mathbb{R}^d)}\to 0$. For the second term, since $\boldsymbol{y}\in
		\mathbb{H}^1(\R^d)=\mathbb{W}^{1,2}(\R^d)$, by definition of this space, we have 
		$\boldsymbol{y}\in \mathbb{L}^2(\R^d)$ and 
		\begin{align}\label{estconvann}
		\|\boldsymbol{y}\nabla\chi_n\|_{\mathbb{L}^2(\mathbb{R}^d)}\le \|\nabla\chi_n\|_{\mathbb{L}^\infty(\mathbb{R}^d)}\|\boldsymbol{y}\|_{\mathbb{L}^2(\mathbb{R}^d)}
		\le \frac{C}{n}\|\boldsymbol{y}\|_{\mathbb{L}^2(\mathbb{R}^d)}\longrightarrow 0.
		\end{align}
		Also $\|\boldsymbol{v}_n-\boldsymbol{y}\|_{\mathbb{L}^2(\mathbb{R}^d)}\to0$ by the same dominated-convergence argument as
		for $\mathbb{L}^p$ (with $p$ replaced by $2$). Combining the three limits gives
		\begin{align}\label{eqn-conv-3}
			\boldsymbol{v}_n\to \boldsymbol{y}\ \text{ in }\ \mathbb{H}^1(\R^d)\cap \mathbb{L}^p(\R^d)\ \text{ as }\ n\to\infty.
		\end{align}
		
		Since $\nabla\chi_n$ is supported in $A_n:=B_{2n}\setminus B_n$, the estimate
		in \eqref{estconvann} can be sharpened to
		\begin{align*}
		\|\boldsymbol{y}\nabla\chi_n\|_{\mathbb{L}^q(\mathbb{R}^d)}\le \frac{C}{n}\|\boldsymbol{y}\|_{\mathbb{L}^q(A_n)}, \ 
		\text{ for any } \ q\in(1,\infty)\ \text{ for which }\ \boldsymbol{y}\in \mathbb{L}^q(\mathbb{R}^d),
		\end{align*}
		and, because $\boldsymbol{y}\in\mathbb{L}^q(A_n)$ is the tail of a fixed $\mathbb{L}^q$ function, this tends
		to $0$ even  {faster} than $1/n$ would suggest, by dominated convergence.
		We present this localized version since it is needed to control the divergence on the annulus.
		
		\vskip 0.1 cm
		\noindent \textbf{Step 2:} \emph{Correction of the divergence via the Bogovski\u{\i}\ operator.} 	Since $\dive \boldsymbol{y}=0$, we find 
		\begin{align*}
		\dive(\boldsymbol{v}_n)=\dive(\chi_n \boldsymbol{y})=\nabla\chi_n\cdot \boldsymbol{y}+\chi_n\,\dive \boldsymbol{y}
		=\boldsymbol{y}\cdot\nabla\chi_n=:g_n .
		\end{align*}
		Clearly $\supp(g_n)\subset A_n$, since $\nabla\chi_n$ vanishes outside $A_n$. It follows that $g_n\in\mathrm{L}_0^q(A_n)$ for every $q\in(1,\infty)$ such that
		\(\boldsymbol{y}\in\mathbb{L}^q(\mathbb{R}^d)\), and in particular for
		\(q=2\) and \(q=p\). Moreover,
		\begin{align}\label{eqn-conv-1}
			\|g_n\|_{\mathrm{L}^q(A_n)}
			=
			\big\|\boldsymbol{y}\cdot\nabla\chi_n\big\|_{\mathrm{L}^q(A_n)}
			\leq
			\|\nabla\chi_n\|_{\mathrm{L}^\infty(\mathbb{R}^d)}
			\|\boldsymbol{y}\|_{\mathbb{L}^q(A_n)}
			\leq
			\frac{C}{n}\,
			\|\boldsymbol{y}\|_{\mathbb{L}^q(A_n)}
			\to  0
			\ \text{ as } \ n\to\infty.
		\end{align}
		It remains to verify that $g_n$ has zero mean over $A_n$. Since
		$\boldsymbol{y}\in \mathbb{H}^{1}(\mathbb{R}^{d})$ and 
		$\chi_n\in\mathrm{C}_0^\infty(\mathbb{R}^d)$, integration by parts yields
		\begin{align*}
		\int_{A_n} g_n\d x=
		\int_{\mathbb{R}^{d}}\boldsymbol{y}\cdot\nabla\chi_n\d x=
		-\int_{\mathbb{R}^{d}}(\nabla\cdot\boldsymbol{y})\,\chi_n\d x=0,
		\end{align*}
		where we have used the fact that \(\nabla\cdot\boldsymbol{y}=0\) in the distributional sense. Consequently,
		$\int_{A_n} g_n\d x=0,	$ and hence $g_n\in\mathrm{L}_0^q(A_n)$.

		Applying Lemma~\ref{lem:scaling} with $\mathfrak{D}=B_2\setminus \overline{B}_1$, we get $\mathfrak{D}_n=A_n$ and a family of operators
		\begin{align*}
		\mathcal B_n:=\mathcal B_{A_n}:\mathrm{L}^q_0(A_n)\to \mathbb{W}_0^{1,q}(A_n),
		\ 
		\dive(\mathcal B_n f)=f,\ 
		\|\nabla(\mathcal B_n f)\|_{\mathbb{L}^q(A_n)}\le C\|f\|_{\mathrm{L}^q(A_n)},
		\end{align*}
		with $C$ independent of $n$ for each fixed $q$. Let us set 
		\begin{align*}
		\boldsymbol{w}_n:=\mathcal B_n(g_n)\in \mathbb{W}_0^{1,p}(A_n).
		\end{align*}
		Since $g_n\in \mathrm{L}^q_0(A_n)$ simultaneously for $q=2$ and $q=p$, and $\mathcal B_n$ is given by one operator valid for
		every $q$ (last part of Lemma~\ref{lem:bogovskii}), $\boldsymbol{w}_n$ satisfies
		both
		\begin{align*}
		\|\nabla\boldsymbol{w}_n\|_{\mathbb{L}^2(A_n)}\le C\|g_n\|_{\mathrm{L}^2(A_n)}
		\ \text{ and }\ 
		\|\nabla\boldsymbol{w}_n\|_{\mathbb{L}^p(A_n)}\le C\|g_n\|_{\mathrm{L}^p(A_n)}.
		\end{align*}
		Since $p>2$ and the annulus $A_n$ has finite measure, the continuous embedding
		$\W_0^{1,p}(A_n)\hookrightarrow\W_0^{1,2}(A_n)=\mathrm{H}_0^1(A_n)$
		holds. Therefore, combining \eqref{eqn-conv-1}, which yields
		$\|g_n\|_{\mathrm{L}^q(A_n)}\to0$ for $q=2,p$, with the estimates of
		Lemma~\ref{lem:scaling} and the bound
		\(\operatorname{diam}(A_n)\le 4n\) (cf.~Lemma~\ref{lem:poincare}), we deduce  for $q=2,p$ that 
		\[
		\|\boldsymbol{w}_n\|_{\mathbb{L}^q(A_n)} \le C n \|g_n\|_{\mathrm{L}^q(A_n)} \le Cn\cdot
		\frac{C}{n}\|\boldsymbol{y}\|_{\mathbb{L}^q(A_n)} = C\|\boldsymbol{y}\|_{\mathbb{L}^q(A_n)}\to  0,
		\]
		since $\boldsymbol{y}\in \mathbb{L}^q(\R^d)$ (recall $\boldsymbol{y}\in \mathbb{H}^1(\mathbb{R}^d)\cap \mathbb{L}^p(\mathbb{R}^d)\subset \mathbb{L}^2(\mathbb{R}^d)\cap \mathbb{L}^p(\mathbb{R}^d)$). Hence
		extending $\boldsymbol{w}_n$ by zero to $\R^d\setminus A_n$, we have $\boldsymbol{w}_n\in
		\mathbb{H}^1(\R^d)\cap \mathbb{L}^p(\R^d)$, $\supp(\boldsymbol{w}_n)\subset A_n\subset B_{2n}$,
		$\dive\boldsymbol{w}_n=g_n$, and 
		\begin{align}\label{eqn-conv-2}
			\boldsymbol{w}_n\to\boldsymbol{0}\ \text{ in }\ \mathbb{H}^1(\R^d)\cap \mathbb{L}^p(\R^d) \ \text{ as }\ n\to\infty.
		\end{align}

		\vskip 0.1 cm
		\noindent \textbf{Step 3:} \emph{Divergence-free, compactly supported approximation.} 	Let us define
		$
		\boldsymbol{z}_n:=\boldsymbol{v}_n-\boldsymbol{w}_n .
		$
		Then $\supp(\boldsymbol{z}_n)\subset B_{2n}$, since $\boldsymbol{v}_n$ and $\boldsymbol{w}_n$ are supported in
		$B_{2n}$, and
		\[
		\dive \boldsymbol{z}_n=\dive \boldsymbol{v}_n-\dive\boldsymbol{w}_n = g_n-g_n=0.
		\]
		Using \eqref{eqn-conv-3} and \eqref{eqn-conv-2}, we find 
		\[
		\| \boldsymbol{z}_n- \boldsymbol{y}\|_{\mathbb{H}^1(\R^d)}+\| \boldsymbol{z}_n- \boldsymbol{y}\|_{\mathbb{L}^p(\R^d)}
		\le \| \boldsymbol{v}_n- \boldsymbol{y}\|_{\mathbb{H}^1(\R^d)}+\|\boldsymbol{v}_n-\boldsymbol{y}\|_{\mathbb{L}^p(\R^d)}+\|\boldsymbol{w}_n\|_{\mathbb{H}^1(\R^d)}+\|\boldsymbol{w}_n\|_{\mathbb{L}^p(\R^d)}
		\to 0,
		\]
		as $n\to \infty$.	Thus $\boldsymbol{z}_n$ is a compactly supported, divergence-free $\mathbb{H}^1(\R^d)\cap \mathbb{L}^p(\R^d)$	approximation of $\boldsymbol{y}$; it only remains to smooth it out.
		
		\vskip 0.1 cm
		\noindent \textbf{Step 4:} \emph{Mollification and conclusion.} Let $\rho_\varepsilon$ be a standard (Friedrichs) mollifier, $\rho_\varepsilon
		\ge0$, $\supp(\rho_\varepsilon)\subset \overline{B}_\varepsilon$, $\int_{\mathbb{R}^d}\rho_\varepsilon(x)\d x=1$. That is, we define 
		\[
		\rho(x):=\begin{cases} C\exp\!\Big(\dfrac{-1}{1-|x|^2}\Big), & |x|<1,\\[6pt]
			0, & |x|\ge1, \end{cases}
		\]
		where $C>0$ is the unique constant for which $\int_{\R^d}\rho(x)\d x=1$. Note that $\rho\in
		\mathrm{C}_c^\infty(\R^d)$, with
		$\rho\ge0, \  \supp(\rho)\subset\overline{B}_1, \  \int_{\R^d}\rho(x)\d x=1 .$
		For $\varepsilon>0$, we set
		\begin{align*}
		\rho_\varepsilon(x):=\frac{1}{\varepsilon^d} \rho\left(\frac{x}{\varepsilon}\right).
		\end{align*}
		Then $\rho_\varepsilon\in\mathrm{C}_c^\infty(\R^d)$, $\rho_\varepsilon\ge0$,
		$\supp(\rho_\varepsilon)\subset\overline{B}_\varepsilon$, and
		$\int_{\R^d}\rho_\varepsilon\d x=1$. For each fixed $n$, $ \boldsymbol{z}_n\in \mathbb{H}^1(\R^d)\cap \mathbb{L}^p(\R^d)$ is compactly supported, and
		it is a standard property of mollifiers that
		\begin{align*}
		\rho_\varepsilon*  \boldsymbol{z}_n \to  \boldsymbol{z}_n \ \text{ in }\  \mathbb{H}^1(\R^d)\cap
		\mathbb{L}^p(\R^d) \  \text{ as }\ \varepsilon\to0^+,
		\end{align*}
		for $1\leq p<\infty$.	Hence, for each $n,$ we may choose $\varepsilon_n\in(0,\tfrac1n)$ small enough
		such 	that
		\begin{align*}
		\|\rho_{\varepsilon_n}* \boldsymbol{z}_n- \boldsymbol{z}_n\|_{\mathbb{H}^1(\R^d)}
		+\|\rho_{\varepsilon_n}* \boldsymbol{z}_n- \boldsymbol{z}_n\|_{\mathbb{L}^p(\R^d)} < \frac1n .
		\end{align*}
		Let us define 
		\begin{align*}
		\boldsymbol{y}_n:=\rho_{\varepsilon_n}*\boldsymbol{z}_n.
		\end{align*}
		By the properties of mollification and the preceding construction, the sequence
		\(\{\boldsymbol{y}_n\}_{n\geq1}\) satisfies the following:
		\begin{itemize}
			\item \(\boldsymbol{y}_n\in\mathrm{C}^\infty(\mathbb{R}^d)\), and
			$
			\supp(\boldsymbol{y}_n)
			\subset
			\supp(\boldsymbol{z}_n)+B_{\varepsilon_n}
			\subset
			B_{2n+1},
			$
			so that each \(\boldsymbol{y}_n\) has compact support.
			
			\item Since convolution commutes with differentiation,
			$\nabla\cdot\boldsymbol{y}_n=
			\rho_{\varepsilon_n}*\nabla\cdot\boldsymbol{z}_n=0,
			$
			so that
			$\boldsymbol{y}_n\in\mathrm{C}_{0,\sigma}^{\infty}(\mathbb{R}^d)$ for each $n$.
			
			\item Finally, by the triangle inequality,
			\begin{align*}
				&\|\boldsymbol{y}_n-\boldsymbol{y}\|_{\mathbb{H}^1(\mathbb{R}^d)}
				+
				\|\boldsymbol{y}_n-\boldsymbol{y}\|_{\mathbb{L}^p(\mathbb{R}^d)}
				\\
				&\leq
				\|\boldsymbol{y}_n-\boldsymbol{z}_n\|_{\mathbb{H}^1(\mathbb{R}^d)}
				+
				\|\boldsymbol{y}_n-\boldsymbol{z}_n\|_{\mathbb{L}^p(\mathbb{R}^d)}
				+
				\|\boldsymbol{z}_n-\boldsymbol{y}\|_{\mathbb{H}^1(\mathbb{R}^d)}
				+
				\|\boldsymbol{z}_n-\boldsymbol{y}\|_{\mathbb{L}^p(\mathbb{R}^d)}
				\\
				&<
				\frac1n
				+
				\|\boldsymbol{z}_n-\boldsymbol{y}\|_{\mathbb{H}^1(\mathbb{R}^d)}
				+
				\|\boldsymbol{z}_n-\boldsymbol{y}\|_{\mathbb{L}^p(\mathbb{R}^d)},
			\end{align*}
			which converges to zero as \(n\to\infty\).
		\end{itemize}
		Therefore,
		we obtain \eqref{eqn-conv},
		thereby completing the proof of Theorem~\ref{main-1}.
	\end{proof}
	
	\begin{remark}
We emphasize that the construction developed in Theorem~\ref{main-1}
relies essentially on Lemma~\ref{lem:scaling}, in particular on the
estimate \eqref{eqn-bogo}. Establishing an analogous estimate on
general unbounded domains may require additional assumptions on the
geometry of the domain. Moreover, the construction of
divergence-free mollification in the presence of a nonempty boundary
involves further technical difficulties and requires more delicate
estimates. We therefore leave the construction of such approximations on
general unbounded domains for future work.

Nevertheless, in Section~\ref{sec-4}, we overcome these difficulties
for unbounded domains of uniform $\mathrm{C}^{1,1}$-type by employing resolvent
estimates corresponding to the Stokes operator. In particular, we establish a simultaneous approximation
result for divergence-free vector fields in Lebesgue and Sobolev
spaces, which provides the approximation framework needed in our
analysis.
	\end{remark}

	\subsection{Energy equality on $\mathbb{R}^d$}
	We now establish a density result, based on the arguments in \cite[Lemma 2.6]{Galdi-2000} and \cite[Lemma 5.3]{Hajduk-2017}. This result plays a fundamental role in the proof of Theorem~\ref{main}.

	Let $\eta(t),$ $t\in\mathbb{R}$ be an even, strictly positive, smooth function with compact support contained in $(-1,1)$ and normalized such that
	\[
	\int_{-\infty}^{\infty}\eta(s)\,\mathrm{d}s=1.
	\]
	For each $h>0$, we define the rescaled mollifier by
	\begin{align}\label{mollifier}
	\eta^h(s):=\frac{1}{h}\eta\left(\frac{s}{h}\right).
	\end{align}
	Since $\eta$ is even and has unit integral, it follows that
	\[
	\int_0^h \eta^h(s)\,\mathrm{d}s=\frac{1}{2}.
	\]
	
	For any function $\boldsymbol{v}\in \mathrm{L}^p(0,T;\mathbb{X})$, where $\mathbb{X}$ is a Banach space and $p\in[1,\infty)$, the temporal mollification of $\boldsymbol{v}$ is defined, for $h\in(0,T)$, by
	\[
	\boldsymbol{v}^h(s)
	:=\int_0^T \boldsymbol{v}(\tau)\eta^h(s-\tau)\,\mathrm{d}\tau,
	\  s\in[0,T].
	\]
	This regularization provides a smooth approximation of $\boldsymbol{v}$ with respect to the time variable. We have the following properties of this mollification:
	\begin{lemma}[{\cite[Lemma 2.5, pp.~16]{Galdi-2000}}]
	Let  $\boldsymbol{v}\in \mathrm{L}^p(0,T;\mathbb{X})$, $p\in[1,\infty)$, where $\mathbb{X}$ is a Banach space. Then 
	$
	\boldsymbol{v}^h\in \mathrm{C}^k([0,T);\mathbb{X})
	\ \text{ for every } \ k\geq 0.
	$
	Furthermore, the mollification converges strongly to the original function as $h$ tends to zero:
	\begin{align}\label{eqn-conv-16}
		\lim_{h\to0}
		\|\boldsymbol{v}^h-\boldsymbol{v}\|_{\mathrm{L}^p(0,T;\mathbb{X})}
		=0.
	\end{align}
	The mollification procedure also preserves strong convergence of sequences. More precisely, if
	$
	\boldsymbol{v}_n\to\boldsymbol{v}
	\ \text{ in }\ \mathrm{L}^p(0,T;\mathbb{X}),
	$
	then, for every fixed $h>0$,
	\begin{align}\label{eqn-conv-11}
		\lim_{n\to\infty}
		\|\boldsymbol{v}_n^h-\boldsymbol{v}^h\|_{\mathrm{L}^p(0,T;\mathbb{X})}
		=0.
	\end{align}
	\end{lemma}
	Hence, temporal mollification not only regularizes functions in time but also preserves strong convergence in the corresponding Bochner space.

	\begin{lemma}\label{lem-reg}
		Let $\Omega=\mathbb{R}^d$ and $p\in[2,\infty)$ be fixed. Then  $\mathcal{D}_{\sigma}([0,T)\times\Omega)$ is dense in $\mathrm{L}^2(0,T;\mathbb{V})\cap \mathrm{L}^{p}(0,T;\mathbb{L}^{p}_{\sigma})$. 
	\end{lemma}
	\begin{proof}
	Let
$
	\boldsymbol{w}\in \mathrm{L}^2(0,T;\mathbb{V})
	\cap \mathrm{L}^{p}(0,T;\mathbb{L}^{p}_{\sigma}).
$
	We regularize $\boldsymbol{w}$ in time using the mollification procedure introduced above and in space using Theorem~\ref{main-1}. Denoting the resulting space-time regularization by $\boldsymbol{w}_n^h$, we obtain
	\[
	\boldsymbol{w}_n^h(x,t)
	:=\left(\boldsymbol{w}^h(x,t)\right)_n\  \text{ for }\ (x,t)\ \text{ in }\ \Omega\times[0,T),
	\]
	where $\boldsymbol{w}^h$ denotes the temporal mollification of $\boldsymbol{w}$ and the subscript $n$ indicates the spatial regularization provided by Theorem~\ref{main-1}. It is clear that $	\boldsymbol{w}_n^h\in \mathcal{D}_{\sigma}([0,T)\times\Omega)$. Applying  Theorem \ref{main-1}, we infer  for all $t\in[0,T)$ that 
	\begin{align}\label{eqn-conv-12}
		\lim_{n\to\infty}\left\|	\boldsymbol{w}_n^h(t)-	\boldsymbol{w}^h(t)\right\|_{\mathbb{L}^{p}(\Omega)}^{p}=0
	\end{align}
	and 
	\begin{align}\label{eqn-conv-13}
		\lim_{n\to\infty}\left\|	\boldsymbol{w}_n^h(t)-	\boldsymbol{w}^h(t)\right\|_{\mathbb{V}}^{2}= 	\lim_{n\to\infty}\left\|	\boldsymbol{w}_n^h(t)-	\boldsymbol{w}^h(t)\right\|_{\mathbb{L}^2(\Omega)}^{2}
		+	\lim_{n\to\infty}\left\|\nabla(\boldsymbol{w}_n^h(t)-	\boldsymbol{w}^h(t))\right\|_{\mathbb{L}^2(\Omega)}^{2}=0. 
	\end{align}
	The convergence given in \eqref{eqn-conv-16} implies, for a given $\varepsilon>0$, we can choose $h>0$ sufficiently small so that 
	\begin{align}\label{eqn-conv-14}
		\int_0^T\|\boldsymbol{w}^h(t)-\boldsymbol{w}(t)\|_{\mathbb{L}^{p}(\Omega)}^{p}\d t<\varepsilon\ \text{ and }\  	\int_0^T\|\boldsymbol{w}^h(t)-\boldsymbol{w}(t)\|_{\mathbb{V}}^{2}\d t<\varepsilon. 
	\end{align}
	On the other hand, combining \eqref{eqn-conv-12} and \eqref{eqn-conv-13} with the Lebesgue Dominated Convergence Theorem yields, for each fixed $h\in(0,T)$,
\begin{align}\label{eqn-conv-15}
	\lim_{n\to\infty}\int_0^T\left\|	\boldsymbol{w}_n^h(t)-	\boldsymbol{w}^h(t)\right\|_{\mathbb{L}^{p}(\Omega)}^{p}\d t=0\ \text{ and }\  	\lim_{n\to\infty}\int_0^T 	\left\|	\boldsymbol{w}_n^h(t)-	\boldsymbol{w}^h(t)\right\|_{\mathbb{V}}^{2}\d t=0. 
\end{align}
Indeed, the required uniform bounds follow from the spatial regularization estimates. More precisely, for every $n\in\mathbb{N}$ and $t\in[0,T)$, we have
\[
\|\boldsymbol{w}_n^h(t)\|_{\mathbb{L}^{p}(\Omega)}
\leq C\|\boldsymbol{w}^h(t)\|_{\mathbb{L}^{p}(\Omega)}
\ \text{ and }\
\|\boldsymbol{w}_n^h(t)\|_{\mathbb{V}}
\leq C\|\boldsymbol{w}^h(t)\|_{\mathbb{V}},
\]
where $C>0$ is independent of $n$. Since
$
\boldsymbol{w}^h
\in \mathrm{L}^{p}(0,T;\mathbb{L}^{p}(\Omega))
\cap \mathrm{L}^2(0,T;\mathbb{V}),
$
these estimates provide the necessary uniform domination. Consequently, the desired convergence follows by combining \eqref{eqn-conv-14} and \eqref{eqn-conv-15} with the triangle inequality. This completes the proof of the lemma.
	\end{proof}

	As an application of Theorem~\ref{main-1}, we now prove the energy equality for the critical and supercritical  CBF equations posed on $\Omega=\mathbb{R}^d$, $d\in\{2,3\}$.

	\begin{theorem}\label{main}
	Let $\Omega=\mathbb{R}^d$, $d\in\{2,3\}$  and $r\in[1,\infty).$ Let $\boldsymbol{u}_0\in \mathbb{H}$ and $\boldsymbol{f}\in \mathrm{L}^{2}(0,T;\mathbb{V}')$  be given.  Then  there exists a Leray-Hopf weak solution to the system (\ref{abstract-CBF})  
		\begin{align} \label{eqn-regularity}
			\boldsymbol{u}\in \mathrm{L}^{\infty}(0,T;\mathbb{H})\cap \mathrm{L}^2(0,T;\mathbb{V})\cap \mathrm{L}^{r+1}(0,T; \mathbb{L}_{\sigma}^{r+1}), \ \frac{\d\boldsymbol{u}}{\d t} \in \mathrm{L}^{p_d}(0,T;\mathbb{V}')+ \mathrm{L}^{\frac{r+1}{r}}(0,T;\mathbb{L}_{\sigma}^{\frac{r+1}{r}}),
		\end{align}
		where $p_d$ is defined in \eqref{eqn-pd}.

		Moreover, for  $d=2$ with $r\in[1,\infty)$ and $d=3$ with $r\in[3,\infty)$, every Leray-Hopf weak solution satisfies the following energy equality:
			\begin{align}\label{energy-equality}
			&	\|\boldsymbol{y}(t_1)\|_{\H}^2+2\mu\int_{0}^{t_1}\|\nabla\boldsymbol{y}(s)\|_{\mathbb{H}}^2\d s+2\alpha\int_{0}^{t_1}\|\boldsymbol{y}(s)\|_{\mathbb{H}}^{2}\d s+2\beta\int_{0}^{t_1}\|\boldsymbol{y}(s)\|_{\mathbb{L}_{\sigma}^{r+1}}^{r+1}\d s\nonumber\\&=\|\boldsymbol{y}_0\|_{\H}^2+2\int_{0}^{t_1}\langle\boldsymbol{g}(s),\boldsymbol{y}(s)\rangle\d s,
		\end{align}
		for all $t_1\in(0,T)$. 
		
		Also, for  $d=2$ with $r\in[1,\infty)$ and $d=3$ with $r\in[3,\infty)$ ($4\beta\mu\geq 1$ for $d=r=3$), the Leray-Hopf weak solution is unique. 
	\end{theorem}
	
	\begin{proof}
		We divide the proof into the following steps: 
		
		\vskip 0.1 cm
		\noindent 
		\textbf{Step 1:} \emph{Existence.}
		The authors in \cite{Antontsev-2010} employed a Galerkin approximation method together with compactness arguments to establish the existence of Leray–Hopf weak solutions to \eqref{abstract-CBF} on bounded domains. In \cite{Gautam-2025}, the authors combined a Galerkin approximation scheme with the local monotonicity of the linear and nonlinear operators and the Minty–Browder technique to establish existence results for the critical and supercritical CBF equations on bounded domains. Similar approaches can be employed to obtain existence results on general unbounded domains, including $\mathbb{R}^d$ (\cite[Theorem 1]{Cai-2008}, \cite[Theorem 3.7]{Kinra-2024}). 
		We only prove the energy equality and uniqueness results.

			\vskip 0.1 cm
		\noindent 
		\textbf{Step 2:} \emph{Energy equality.}
		Our proof of the energy equality is based on the arguments developed in
		\cite{Fefferman-2022,Galdi-2011,Gautam-2025,Hajduk-2017} and related
		works. Energy equality for the three-dimensional critical CBF
		equations has previously been established on the torus in
		\cite[Theorem 1.4]{Hajduk-2017} and on bounded domains in
		\cite[Theorem 5.4]{Fefferman-2022}; see also
		\cite[Theorem 3.5]{Gautam-2025} for the supercritical case. Nevertheless,
		we present the proof in our setting for completeness.
		
	Let $\boldsymbol{y}(\cdot)$ be a Leray-Hopf weak solution of the CBF equations \eqref{abstract-CBF}, for $d=2$ with $r\in[1,\infty)$, and for $d=3$ with $r\in[3,\infty)$. Note that $\boldsymbol{y}(t)\in\mathbb{V}\cap\mathbb{L}_{\sigma}^{r+1}$ for a.e.\ $t\in[0,T]$. By Theorem \ref{main-1}, $\boldsymbol{y}(t)$ can be approximated by functions possessing the desired regularity. We denote such an approximating sequence by $\boldsymbol{y}_n(t)$. According to Theorem~\ref{main-1}, this approximation satisfies, for a.e.\ $t\in[0,T]$:
	\begin{equation}\label{4.50}
	\left\{
	\begin{aligned}
		(1)~& \nabla\cdot\boldsymbol{y}_n(t)=0 \  \text{ and } \boldsymbol{y}_n(t)\big|_{\partial\Omega}=0 \text{ for all } t \in [0,T], \\
		(2)~& \lim\limits_{n\to\infty}\|\boldsymbol{y}_n(t)-\boldsymbol{y}(t)\|_{\mathbb{H}^1_0(\Omega)}=0 \text{ with } \|\boldsymbol{y}_n(t)\|_{\mathbb{H}^1} \leq C  \|\boldsymbol{y}(t)\|_{\mathbb{H}^1}
		\text{ for all } t \in [0,T], \\
		(3)~& \lim\limits_{n\to\infty}\|\boldsymbol{y}_n(t)-\boldsymbol{y}(t)\|_{\mathbb{L}^p(\Omega)}=0 \text{ with } \|\boldsymbol{y}_n(t)\|_{\mathbb{L}^p} \leq C  \|\boldsymbol{y}(t)\|_{\L^p} \text{ for } p \in (1,\infty),  \\ &\text{ for a.e. } t \in [0,T].
	\end{aligned}
	\right.
\end{equation}
Since $\boldsymbol{y}\in\mathrm{L}^{r+1}(0,T;\mathbb{L}_{\sigma}^{r+1})$ and $\|\boldsymbol{y}_n(t)-\boldsymbol{y}(t)\|_{\mathbb{L}_{\sigma}^{r+1}}\to 0$ for all $t\in[0,T]$, an application of the Lebesgue dominated convergence theorem (with dominating function $(1+C)\|\boldsymbol{y}(t)\|_{\mathbb{L}_{\sigma}^{r+1}}$) yields
\begin{align}\label{335}
	\|\boldsymbol{y}_n-\boldsymbol{y}\|_{\mathrm{L}^{r+1}(0,T;\mathbb{L}_{\sigma}^{r+1})}\to 0, \ \text{ as }\ n\to\infty.
\end{align}
Moreover, since $\boldsymbol{y}\in\mathrm{L}^{2}(0,T;\mathbb{V})$, an identical argument gives
\begin{align}\label{335a}
	\|\boldsymbol{y}_n-\boldsymbol{y}\|_{\mathrm{L}^{2}(0,T;\mathbb{V})}\to 0, \ \text{ as }\ n\to\infty.
\end{align}

In view of the regularity result established in Lemma~\ref{lem-reg}, for any fixed $t_1\in(0,T)$ we may select a sequence of suitable test functions
\begin{align*}
	\boldsymbol{y}_n^h(t):=\int_0^{t_1}\eta_h(t-s)\boldsymbol{y}_n(s)\,\mathrm{d}s,
\end{align*}
where the parameter $h$ satisfies $0<h<T-t_1$ and $h<t_1$, and $\eta_h$ denotes the even mollifier introduced in \eqref{mollifier}. Since $\boldsymbol{y}_n^h\in\mathcal{D}_{\sigma}([0,T)\times\Omega)$, we may take $\boldsymbol{y}_n^h$ as a test function in \eqref{3.13}, which yields
\begin{align}\label{3.23}
	-&\int_0^{t_1}\bigg(\boldsymbol{y}(s),\frac{\mathrm{d}\boldsymbol{y}_n^h(s)}{\mathrm{d}s}\bigg)\,\mathrm{d}s+ \int_0^{t_1}\langle\mu\mathcal{A}\boldsymbol{y}(s)+\mathcal{B}(\boldsymbol{y}(s))+\alpha\boldsymbol{y}(s)+\beta\mathcal{C}(\boldsymbol{y}(s))-\boldsymbol{f}(s),\boldsymbol{y}_n^h(s)\rangle\,\mathrm{d}s\nonumber\\&=-(\boldsymbol{y}(t_1),\boldsymbol{y}_n^h(t_1)) +(\boldsymbol{y}(0),\boldsymbol{y}_n^h(0)).
\end{align}

Let us first consider 
\begin{align*}
&\left|\int_0^{t_1}\bigg(\boldsymbol{y}(s),\frac{\d\boldsymbol{y}_n^h(s)}{\d s}\bigg)\d s-\int_0^{t_1}\bigg(\boldsymbol{y}(s),\frac{\d\boldsymbol{y}^h(s)}{\d s}\bigg)\d s\right|
\nonumber\\&\leq \int_0^{t_1}\|\boldsymbol{y}(s)\|_{\mathbb{H}}\left\|\frac{\d\boldsymbol{y}_n^h(s)}{\d s}-\frac{\d\boldsymbol{y}^h(s)}{\d s}\right\|_{\mathbb{H}}\d s\leq \|\boldsymbol{y}\|_{\mathrm{L}^{\infty}(0,T;\mathbb{H})}
\left\|\frac{\d\boldsymbol{y}_n^h}{\d s}-
\frac{\d\boldsymbol{y}^h}{\d s}\right\|_{\mathrm{L}^1(0,T;\mathbb{H})}.
\end{align*}
Using Minkowski's integral inequality, Fubini's theorem and the convergence given in \eqref{335a}, we find 
\begin{align*}
	\left\|\frac{\d\boldsymbol{y}_n^h}{\d s}-\frac{\d\boldsymbol{y}^h}{\d s} \right\|_{\mathrm{L}^1(0,T;\mathbb{H})}&=\int_0^{t_1}\left\|\int_0^{t_1}|\dot{\eta}^h(t-s)|\left(\boldsymbol{y}_n^h(s)-\boldsymbol{y}^h(s)\right)\d s\right\|_{\mathbb{H}}\d t\nonumber\\&\leq  \int_0^{t_1}\int_0^{t_1}|\dot{\eta}^h(t-s)|\left\| \boldsymbol{y}_n^h(s)-\boldsymbol{y}^h(s) \right\|_{\mathbb{H}}\d s\d t
	\nonumber\\&=\int_0^{t_1}\left\| \boldsymbol{y}_n^h(s)-\boldsymbol{y}^h(s) \right\|_{\mathbb{H}}\left(\int_0^{t_1}|\dot{\eta}^h(t-s)|\d t\right)\d s
	\nonumber\\&\leq C\|\boldsymbol{y}_n^h-\boldsymbol{y}^h\|_{\mathrm{L}^1(0,T;\mathbb{H})}\to 0\ \text{ as }\ n\to\infty, 
\end{align*}
where $\dot{\eta}^h$ derivative of $\eta^h$ with respect to $t$ and $C=\int_{-\infty}^{\infty}|\dot{\eta}^h(t)|\d t<\infty$. Therefore, we deduce 
\begin{align*}
	-\int_0^{t_1}\bigg(\boldsymbol{y}(s),\frac{\d\boldsymbol{y}_n^h(s)}{\d t}\bigg)\d s\to 	-\int_0^{t_1}\bigg(\boldsymbol{y}(s),\frac{\d\boldsymbol{y}^h(s)}{\d t}\bigg)\d s\ \text{ as } \ n\to\infty. 
\end{align*}
Let us now take the second term in the right hand side of the equality \eqref{3.23} and estimate it as 
\begin{align}\label{eqn-strong-1}
&	\left|\int_0^{t_1}\langle\mu\mathcal{A}\boldsymbol{y}(s)+\mathcal{B}(\boldsymbol{y}(s))+\alpha\boldsymbol{y}(s)+\beta\mathcal{C}(\boldsymbol{y}(s))-\boldsymbol{f}(s),\boldsymbol{y}_n^h(s)-\boldsymbol{y}^h(s)\rangle\d s\right|
\nonumber\\&\leq \mu\int_0^{t_1}\|\nabla\boldsymbol{y}(s)\|_{\mathbb{H}}\|\nabla(\boldsymbol{y}_n^h(s)-\boldsymbol{y}^h(s))\|_{\mathbb{H}}\d s+\int_0^{t_1}\|\boldsymbol{y}(s)\|_{\mathbb{L}^4}^2 \|\nabla(\boldsymbol{y}_n^h(s)-\boldsymbol{y}^h(s))\|_{\mathbb{H}}\d s
\nonumber\\&\quad+\alpha\int_0^{t_1}\|\boldsymbol{y}(s)\|_{\mathbb{H}}\|\boldsymbol{y}_n^h(s)-\boldsymbol{y}^h(s)\|_{\mathbb{H}}\d s+\beta\int_0^t\|\boldsymbol{y}(s)\|_{\mathbb{L}^{r+1}_{\sigma}}^{r}\|\boldsymbol{y}_n^h(s)-\boldsymbol{y}^h(s)\|_{\mathbb{L}^{r+1}_{\sigma}}\d s\nonumber\\&\quad+\int_0^{t_1}\|\boldsymbol{f}(s)\|_{\mathbb{V}'}\|\boldsymbol{y}_n^h(s)-\boldsymbol{y}^h(s)\|_{\mathbb{V}}\d s
\nonumber\\&\leq\left\{(\mu+\alpha)\|\boldsymbol{y}\|_{\mathrm{L}^2(0,T;\mathbb{H})}+\|\boldsymbol{y}\|_{\mathbb{L}^4(0,T;\mathbb{L}^4(\Omega))}^2+\beta\|\boldsymbol{y}\|_{\mathrm{L}^{r+1}(0,T;\mathbb{L}^{r+1}_{\sigma})}^r+\|\boldsymbol{f}\|_{\mathrm{L}^2(0,T;\mathbb{V}')}\right\}\nonumber\\&\quad\times\left\{\|\boldsymbol{y}_n^h-\boldsymbol{y}^h\|_{\mathrm{L}^2(0,T;\mathbb{V})}+\|\boldsymbol{y}_n^h-\boldsymbol{y}^h\|_{\mathrm{L}^{r+1}(0,T;\mathbb{L}^{r+1}_{\sigma})}\right\}
\nonumber\\&\to 0\ \text{ as } h\to 0,
\end{align}
by using \eqref{335} and \eqref{335a}. One can get the convergence of $(\boldsymbol{y}(t_1),\boldsymbol{y}_n^h(t_1))\to (\boldsymbol{y}(t_1),\boldsymbol{y}^h(t_1))$ as $n\to\infty$ in the following way: 
\begin{align*}
	|(\boldsymbol{y}(t_1),\boldsymbol{y}_n^h(t_1)-\boldsymbol{y}^h(t_1))|&\leq\|\boldsymbol{y}(t_1)\|_{\mathbb{H}}\|\boldsymbol{y}_n^h(t_1)-\boldsymbol{y}^h(t_1)\|_{\mathbb{H}}\nonumber\\&\leq \|\boldsymbol{y}\|_{\mathrm{L}^{\infty}(0,T;\mathbb{H})}\int_0^{t_1}\eta^h(t_1-s)\|\boldsymbol{y}_n(s)-\boldsymbol{y}(s)\|_{\mathbb{H}}\d s\nonumber\\&\to 0\ \text{ as }\ n\to\infty, 
\end{align*}
by an application of the Lebesgue dominated convergence theorem, since $\int_0^{t_1}\eta^h(t_1-s)\d s=\int_0^{t_1}\eta^h(r)\d r\leq \int_0^{h}\eta^h(r)\d r=\frac{1}{2},$ $\boldsymbol{y}\in \mathrm{L}^{\infty}(0,T;\mathbb{H})$ and $\boldsymbol{y}_n(s)\to \boldsymbol{y}(s)$ in $\mathbb{H}$ for a.e. $s\in(0,T)$. Similarly, one can show the convergence of $(\boldsymbol{y}(0),\boldsymbol{y}_n^h(0))\to (\boldsymbol{y}(0),\boldsymbol{y}^h(0))$ as $n\to\infty$. Therefore, passing $n\to\infty$ in \eqref{3.23}, we find 
\begin{align}\label{eqn-new-s}
	-&\int_0^{t_1}\bigg(\boldsymbol{y}(s),\frac{\d\boldsymbol{y}_n^h(s)}{\d s}\bigg)\d s+ \int_0^{t_1}\langle\mu\mathcal{A}\boldsymbol{y}(s)+\mathcal{B}(\boldsymbol{y}(s))+\alpha\boldsymbol{y}(s)+\beta\mathcal{C}(\boldsymbol{y}(s))-\boldsymbol{f}(s),\boldsymbol{y}^h(s)\rangle\d s\nonumber\\&=-(\boldsymbol{y}(t_1),\boldsymbol{y}^h(t_1)) +(\boldsymbol{y}(0),\boldsymbol{y}^h(0)).
\end{align}
We now proceed to pass to the limit as $h\to0$ in \eqref{eqn-new-s}. Since the function $\eta^h$ is even in $(-h, h)$, its derivative is odd, that is,  $\dot{\eta}^h(r)=-\dot{\eta}^h(-r)$. Therefore, we immediately obtain (see \cite[pp. 7154]{Hajduk-2017}) 
\begin{align}
	\int_0^{t_1}\bigg(\boldsymbol{y}(s),\frac{\d\boldsymbol{y}^h(s)}{\d s}\bigg)\d s&= \int_0^{t_1}\int_0^{t_1}\dot{\eta}^h(s-\tau)(\boldsymbol{y}(s),\boldsymbol{y}(\tau))\d \tau\d s\nonumber\\&=- \int_0^{t_1}\int_0^{t_1}\dot{\eta}^h(\tau-s)(\boldsymbol{y}(s),\boldsymbol{y}(\tau))\d \tau\d s\nonumber\\&=- \int_0^{t_1}\int_0^{t_1}\dot{\eta}^h(\tau-s)(\boldsymbol{y}(\tau),\boldsymbol{y}(s))\d \tau\d s \nonumber\\&=- \int_0^{t_1}\int_0^{t_1}\dot{\eta}^h(s-\tau)(\boldsymbol{y}(s),\boldsymbol{y}(\tau))\d s\d \tau=0. 
\end{align}
Replacing $\boldsymbol{y}_n^h$ with $\boldsymbol{y}^h$ and $\boldsymbol{y}^h$ with $\boldsymbol{y}$ in \eqref{eqn-strong-1}, we find as $h\to 0$ 
	\begin{align}\label{eqn-strong-2}
	&	\int_0^{t_1}\langle\mu\mathcal{A}\boldsymbol{y}(s)+\mathcal{B}(\boldsymbol{y}(s))+\alpha\boldsymbol{y}(s)+\beta\mathcal{C}(\boldsymbol{y}(s))-\boldsymbol{f}(s),\boldsymbol{y}^h(s)\rangle\d s
	\nonumber\\&\to \int_0^{t_1}\langle\mu\mathcal{A}\boldsymbol{y}(s)+\mathcal{B}(\boldsymbol{y}(s))+\alpha\boldsymbol{y}(s)+\beta\mathcal{C}(\boldsymbol{y}(s))-\boldsymbol{f}(s),\boldsymbol{y}(s)\rangle\d s  
	\nonumber\\&= \mu\int_0^{t_1}\|\nabla\boldsymbol{y}(s)\|_{\mathbb{H}}^2\d s+\alpha\int_0^{t_1}\|\boldsymbol{y}(s)\|_{\mathbb{H}}^2\d s+\beta\int_0^{t_1}\|\boldsymbol{y}(s)\|_{\mathbb{L}^{r+1}_{\sigma}}^{r+1}\d s-\int_0^{t_1}\langle\boldsymbol{f}(s),\boldsymbol{y}(s)\rangle\d s,
	\end{align}
since $\langle\mathcal{B}(\boldsymbol{y}),\boldsymbol{y}\rangle=0$.  
The above convergences imply 
\begin{align}\label{utuxwc}
&\mu\int_0^{t_1}\|\nabla\boldsymbol{y}(s)\|_{\mathbb{H}}^2\d s+\alpha\int_0^{t_1}\|\boldsymbol{y}(s)\|_{\mathbb{H}}^2\d s+\beta\int_0^{t_1}\|\boldsymbol{y}(s)\|_{\mathbb{L}^{r+1}_{\sigma}}^{r+1}\d s-\int_0^{t_1}\langle\boldsymbol{f}(s),\boldsymbol{y}(s)\rangle\d s\nonumber\\&=-\lim_{h\to 0}(\boldsymbol{y}(t_1),\boldsymbol{y}^h(t_1))+\lim_{h\to 0}(\boldsymbol{y}(0),\boldsymbol{y}^h(0)).
\end{align}
Since $\boldsymbol{y}$ is $\mathrm{L}^2$-weakly continuous in time and $\int_0^h\eta_h(s)\,\mathrm{d}s=\frac{1}{2}$, we obtain
\begin{align}\label{ut1wc}
	(\boldsymbol{y}(t_1),\boldsymbol{y}_n^h(t_1)) &=\int_0^{t_1}\eta_h(s)(\boldsymbol{y}(t_1),\boldsymbol{y}(t_1-s))\,\mathrm{d}s\nonumber\\
	&=\frac{1}{2}\|\boldsymbol{y}(t_1)\|_{\mathbb{H}}^2+\int_0^h\eta_h(s)\big(\boldsymbol{y}(t_1),\boldsymbol{y}(t_1-s)-\boldsymbol{y}(t_1)\big)\,\mathrm{d}s\nonumber\\
	&\to \frac{1}{2}\|\boldsymbol{y}(t_1)\|_{\mathbb{H}}^2,
\end{align}
as $h\to 0$. An analogous computation gives
\begin{align}\label{ut0wc}
	(\boldsymbol{y}(0),\boldsymbol{y}_n^h(0))\to\frac{1}{2}\|\boldsymbol{y}(0)\|_{\mathbb{H}}^2 \quad \text{as } h\to 0.
\end{align}
Substituting the limits \eqref{ut1wc}-\eqref{ut0wc} into \eqref{utuxwc}, we arrive at the energy equality
\begin{align*}
	&\frac{1}{2}\|\boldsymbol{y}(t_1)\|_{\mathbb{H}}^2+\mu\int_0^{t_1}\|\nabla\boldsymbol{y}(s)\|_{\mathbb{H}}^2\,\mathrm{d}s+\alpha\int_0^{t_1}\|\boldsymbol{y}(s)\|_{\mathbb{H}}^2\,\mathrm{d}s+\beta\int_0^{t_1}\|\boldsymbol{y}(s)\|_{\mathbb{L}^{r+1}_{\sigma}}^{r+1}\,\mathrm{d}s\nonumber\\
	&=\frac{1}{2}\|\boldsymbol{y}(0)\|_{\mathbb{H}}^2+\int_0^{t_1}\langle\boldsymbol{f}(s),\boldsymbol{y}(s)\rangle\,\mathrm{d}s,
\end{align*}
valid for every $t_1\in(0,T)$. This shows that every Leray-Hopf weak solution of \eqref{abstract-CBF} satisfies the energy equality \eqref{energy-equality}.

Let us now show that every weak solution of the CBF equations \eqref{abstract-CBF} possesses time-continuous trajectories in the $\mathbb{L}^2$-space, that is,
\begin{align}\label{eqn-conv-17}
	\|\boldsymbol{y}(t)-\boldsymbol{y}(t_0)\|_{\mathbb{H}}\to 0\ \text{ as }\ t\to t_0,
\end{align}
for all $t\in[0,T)$. From Remark~\ref{rem-main}, we infer
\begin{align}\label{eqn-weak-con}
	\boldsymbol{y}(t)\xrightarrow{w}	\boldsymbol{y}(t_0)\ \text{ as }\ t\to t_0,
\end{align}
for all $t\in[0,T)$. To upgrade this to convergence of the norms, it suffices to invoke the energy inequality \eqref{energy-inequality}. Indeed, \eqref{energy-inequality} immediately gives
\begin{align*}
	\limsup_{t\to t_0}\|\boldsymbol{y}(t)\|_{\mathbb{H}}^2\leq\|\boldsymbol{y}(t_0)\|_{\mathbb{H}}^2,
\end{align*}
while the weak continuity \eqref{eqn-weak-con} yields
\begin{align*}
	\|\boldsymbol{y}(t_0)\|_{\mathbb{H}}^2\leq 	\liminf_{t\to t_0}\|\boldsymbol{y}(t)\|_{\mathbb{H}}^2.
\end{align*}
Combining these two estimates, we conclude that every weak solution of \eqref{abstract-CBF} satisfies
\begin{align}\label{eqn-weak-strong}
	\boldsymbol{y}(t)\xrightarrow{w}	\boldsymbol{y}(t_0)\ \text{ and }\ \|\boldsymbol{y}(t)\|_{\mathbb{H}}\to \|\boldsymbol{y}(t_0)\|_{\mathbb{H}}\  \text{ as }\ t\to t_0.
\end{align}
The desired conclusion \eqref{eqn-conv-17} now follows directly from \eqref{eqn-weak-strong}: weak convergence together with convergence of the corresponding norms implies strong convergence in a Hilbert space, by the Radon-Riesz theorem (see \cite[Proposition 3.32, p. 78]{Brezis-2011}).
			\vskip 0.1 cm
		\noindent 
		\textbf{Step 3:} \emph{Uniqueness.}
	We next establish the uniqueness of Leray-Hopf weak solutions in both the critical and supercritical regimes. We consider  the following three cases. 
		\vskip 0.1cm
		\noindent 
		\textbf{Case 1:}  \emph{$d\in\{2,3\}$ and  $r>3$.} 
		Let us first consider the case $d\in\{2,3\}$ and $r>3$. Let $\boldsymbol{y}_1(\cdot)$ and $\boldsymbol{y}_2(\cdot)$ be two weak solutions of the system \eqref{abstract-CBF} satisfying \eqref{eqn-regularity}. Then $\boldsymbol{y}_1(\cdot)-\boldsymbol{y}_2(\cdot)$ also possesses the same regularity, and, arguing as in Step 2, one can show that the difference $\boldsymbol{y}_1(\cdot)-\boldsymbol{y}_2(\cdot)$ satisfies the following energy equality:
		\begin{align}\label{346}
			&\|\boldsymbol{y}_1(t)-\boldsymbol{y}_2(t)\|_{\mathbb{H}}^2+2\mu \int_0^t\|\nabla(\boldsymbol{y}_1(s)-\boldsymbol{y}_2(s))\|_{\mathbb{H}}^2\d s+2\alpha \int_0^t\| \boldsymbol{y}_1(s)-\boldsymbol{y}_2(s)\|_{\mathbb{H}}^2\d s\nonumber\\&=\|\boldsymbol{y}_1(0)-\boldsymbol{y}_2(0)\|_{\mathbb{H}}^2-2\int_0^t\langle\B(\boldsymbol{y}_1(s))-\mathrm{B}(\boldsymbol{y}_2(s)),\boldsymbol{y}_1(s)-\boldsymbol{y}_2(s)\rangle\d s\nonumber\\&\quad -2\beta\int_0^t\langle\mathcal{C}(\boldsymbol{y}_1(s))-\mathcal{C}_2(\boldsymbol{y}_2(s)),\boldsymbol{y}_1(s)-\boldsymbol{y}_2(s)\rangle\d s.
		\end{align}
		Note that $\langle\B(\boldsymbol{y}_1)-\B(\boldsymbol{y}_2),\boldsymbol{y}_1-\boldsymbol{y}_2\rangle=\langle\B(\boldsymbol{y}_1-\boldsymbol{y}_2,\boldsymbol{y}_2),\boldsymbol{y}_1-\boldsymbol{y}_2\rangle=-\langle\B(\boldsymbol{y}_1-\boldsymbol{y}_2,\boldsymbol{y}_1-\boldsymbol{y}_2),\boldsymbol{y}_2\rangle,$ since $\langle\B(\boldsymbol{y}_1,\boldsymbol{y}_1-\boldsymbol{y}_2),\boldsymbol{y}_1-\boldsymbol{y}_2\rangle=0$. 
		We estimate $	|\langle\B(\boldsymbol{y}_1-\boldsymbol{y}_2,\boldsymbol{y}_1-\boldsymbol{y}_2),\boldsymbol{y}_2\rangle|$ using H\"older's and Young's inequalities as 
		\begin{align}\label{2p28}
			|\langle\B(\boldsymbol{y}_1-\boldsymbol{y}_2,\boldsymbol{y}_1-\boldsymbol{y}_2),\boldsymbol{y}_2\rangle|&\leq\|\nabla(\boldsymbol{y}_1-\boldsymbol{y}_2)\|_{\mathbb{H}}\|\boldsymbol{y}_2(\boldsymbol{y}_1-\boldsymbol{y}_2)\|_{\mathbb{H}}\nonumber\\&\leq\frac{\mu }{2}\|\nabla(\boldsymbol{y}_1-\boldsymbol{y}_2)\|_{\mathbb{H}}^2+\frac{1}{2\mu }\|\boldsymbol{y}_2(\boldsymbol{y}_1-\boldsymbol{y}_2)\|_{\mathbb{H}}^2.
		\end{align}
		We take the term $\|\boldsymbol{y}_2(\boldsymbol{y}_1-\boldsymbol{y}_2)\|_{\mathbb{H}}^2$ from \eqref{2p28} and once again use H\"older's and Young's inequalities to estimate it as (see \cite{Gautam-2025,Hajduk-2017} also)
		\begin{align}\label{2.29}
			&\int_{\mathcal{O}}|\boldsymbol{y}_2(x)|^2|\boldsymbol{y}_1(x)-\boldsymbol{y}_2(x)|^2\d x\nonumber\\&=\int_{\mathcal{O}}|\boldsymbol{y}_2(x)|^2|\boldsymbol{y}_1(x)-\boldsymbol{y}_2(x)|^{\frac{4}{r-1}}|\boldsymbol{y}_1(x)-\boldsymbol{y}_2(x)|^{\frac{2(r-3)}{r-1}}\d x\nonumber\\&\leq\left(\int_{\mathcal{O}}|\boldsymbol{y}_2(x)|^{r-1}|\boldsymbol{y}_1(x)-\boldsymbol{y}_2(x)|^2\d x\right)^{\frac{2}{r-1}}\left(\int_{\mathcal{O}}|\boldsymbol{y}_1(x)-\boldsymbol{y}_2(x)|^2\d x\right)^{\frac{r-3}{r-1}}\nonumber\\&\leq{\beta\mu }\left(\int_{\mathcal{O}}|\boldsymbol{y}_2(x)|^{r-1}|\boldsymbol{y}_1(x)-\boldsymbol{y}_2(x)|^2\d x\right) +\varrho \left(\int_{\mathcal{O}}|\boldsymbol{y}_1(x)-\boldsymbol{y}_2(x)|^2\d x\right),
		\end{align}
		for $r>3$, 	where $\varrho=\frac{r-3}{2\mu(r-1)}\left(\frac{2}{\beta\mu (r-1)}\right)^{\frac{2}{r-3}}$. Using \eqref{2.29} in \eqref{2p28}, we deduce 
		\begin{align}\label{2.30}
			&|\langle\B(\boldsymbol{y}_1-\boldsymbol{y}_2,\boldsymbol{y}_1-\boldsymbol{y}_2),\boldsymbol{y}_2\rangle|\nonumber\\&\leq\frac{\mu }{2}\|\boldsymbol{y}_1-\boldsymbol{y}_2\|_{\mathbb{V}}^2+\frac{\beta}{2}\||\boldsymbol{y}_2|^{\frac{r-1}{2}}(\boldsymbol{y}_1-\boldsymbol{y}_2)\|_{\mathbb{H}}^2+\varrho\|\boldsymbol{y}_1-\boldsymbol{y}_2\|_{\mathbb{H}}^2.
		\end{align}
	 A calculation similar to \eqref{2.23} gives 
		\begin{align}\label{261}
			&\beta\langle\mathcal{C}(\boldsymbol{y}_1)-\mathcal{C}(\boldsymbol{y}_2),\boldsymbol{y}_1-\boldsymbol{y}_2\rangle\geq \frac{\beta}{2}\||\boldsymbol{y}_2|^{\frac{r-1}{2}}(\boldsymbol{y}_1-\boldsymbol{y}_2)\|_{\mathbb{H}}^2.
		\end{align}
		Using (\ref{2.30}) and (\ref{261}) in \eqref{346}, we obtain 
		\begin{align}\label{3.18}
			&\|\boldsymbol{y}_1(t)-\boldsymbol{y}_2(t)\|_{\mathbb{H}}^2+\mu \int_0^t\|\boldsymbol{y}_1(s)-\boldsymbol{y}_2(s)\|_{\mathbb{V}}^2\d s\nonumber\\&\leq\|\boldsymbol{y}_1(0)-\boldsymbol{y}_2(0)\|_{\mathbb{H}}^2+2\varrho\int_0^t\|\boldsymbol{y}_1(s)-\boldsymbol{y}_2(s)\|_{\mathbb{H}}^2\d s.
		\end{align}
		Applying Gr\"onwall's inequality in (\ref{3.18}), we arrive at 
		\begin{align}\label{269}
			\|\boldsymbol{y}_1(t)-\boldsymbol{y}_2(t)\|_{\mathbb{H}}^2&\leq \|\boldsymbol{y}_1(0)-\boldsymbol{y}_2(0)\|_{\mathbb{H}}^2e^{2\varrho T},
		\end{align}
		and hence the uniqueness follows by taking $\boldsymbol{y}_1(0)=\boldsymbol{y}_2(0)$ in \eqref{269}.
		
			\vskip 0.1cm
		\noindent 
		\textbf{Case 2:}  \emph{$d=2$ and  $r\in[1,3]$.} 
		For $d=2$ and $r\in[1,3]$, since $\boldsymbol{u}\in \mathrm{L}^{\infty}(0,T;\mathbb{H})\cap \mathrm{L}^2(0,T;\mathbb{V})$, an application of the Ladyzhenskaya inequality yields $\boldsymbol{u}\in \mathrm{L}^4(0,T; \mathbb{L}_{\sigma}^4)$. Combining this with \cite[Theorem 4.1, pp. 23]{Galdi-2019}, we conclude that $\boldsymbol{u}$ satisfies the energy equality \eqref{energy-equality}; consequently, $\boldsymbol{u}\in \mathrm{C}([0,T];\mathbb{H})$ as well. Since $\boldsymbol{y}\in \mathrm{L}^2(0,T;\mathbb{V})$ and $\partial_t\boldsymbol{y}\in \mathrm{L}^2(0,T;\mathbb{V}'),$ one can even apply the classical Lions-Magenes lemma ((\cite[Theorem 3.1, Chapter 1]{Lions-1972}, \cite[Lemma 1.2, Chapter 3]{Temam-1984})) to obtain this result. 
		
		In order to prove the uniqueness of Leray-Hopf weak solutions, we estimate $\langle\B(\boldsymbol{y}_1-\boldsymbol{y}_2,\boldsymbol{y}_2),\boldsymbol{y}_1-\boldsymbol{y}_2\rangle$ from \eqref{346} as
		\begin{align}\label{3711}
			|\langle\B(\boldsymbol{y}_1-\boldsymbol{y}_2,\boldsymbol{y}_2),\boldsymbol{y}_1-\boldsymbol{y}_2\rangle|&=|\langle\B(\boldsymbol{y}_1-\boldsymbol{y}_2,\boldsymbol{y}_1-\boldsymbol{y}_2),\boldsymbol{y}_2\rangle|\nonumber\\&\leq\|\boldsymbol{y}_2\|_{\mathbb{L}_{\sigma}^4}\|\nabla(\boldsymbol{y}_1-\boldsymbol{y}_2)\|_{\mathbb{H}}\|\boldsymbol{y}_1-\boldsymbol{y}_2\|_{\mathbb{L}_{\sigma}^4}\nonumber\\&\leq 2^{1/4}\|\boldsymbol{y}_2\|_{\mathbb{L}_{\sigma}^4}\|\nabla(\boldsymbol{y}_1-\boldsymbol{y}_2)\|_{\mathbb{H}}^{3/2}\|\boldsymbol{y}_1-\boldsymbol{y}_2\|_{\mathbb{H}}^{1/2}\nonumber\\&\leq\frac{\mu}{2}\|\nabla(\boldsymbol{y}_1-\boldsymbol{y}_2)\|_{\mathbb{H}}^2+\frac{27}{32\mu^3}\|\boldsymbol{y}_2\|_{\mathbb{L}_{\sigma}^4}^4\|\boldsymbol{y}_1-\boldsymbol{y}_2\|_{\mathbb{H}}^2. 
		\end{align}
		Using \eqref{3711} in \eqref{346}, we arrive at 
		\begin{align}\label{372}
			&\|\boldsymbol{y}_1(t)-\boldsymbol{y}_2(t)\|_{\mathbb{H}}^2+\mu \int_0^t\|\nabla(\boldsymbol{y}_1(s)-\boldsymbol{y}_2(s))\|_{\mathbb{H}}^2\d s+\frac{\beta}{2^{r-1}}\int_0^t\|\boldsymbol{y}_1(s)-\boldsymbol{y}_2(s)\|_{\mathbb{L}_{\sigma}^{r+1}}^{r+1}\d s \nonumber\\&\leq\|\boldsymbol{y}_1(0)-\boldsymbol{y}_2(0)\|_{\mathbb{H}}^2+\frac{27}{16\mu^3}\int_0^t\|\boldsymbol{y}_2(s)\|_{\mathbb{L}_{\sigma}^4}^4\|\boldsymbol{y}_1(s)-\boldsymbol{y}_2(s)\|_{\mathbb{H}}^2\d s.
		\end{align}
		An application of Gr\"onwall's inequality in \eqref{372} yields 
		\begin{align}
			\|\boldsymbol{y}_1(t)-\boldsymbol{y}_2(t)\|_{\mathbb{H}}^2\leq \|\boldsymbol{y}_1(0)-\boldsymbol{y}_2(0)\|_{\mathbb{H}}^2\exp\left\{\frac{27}{16\mu^3}\int_0^T\|\boldsymbol{y}_2(t)\|_{\mathbb{L}_{\sigma}^4}^4\d t\right\}, 
		\end{align}
		and the uniqueness $\boldsymbol{y}_1(t)=\boldsymbol{y}_2(t)$ for all $t\in[0,T]$ in $\mathbb{H}$ follows since  $\boldsymbol{y}\in \mathrm{L}^4(0,T; \mathbb{L}_{\sigma}^4)$. 
		
			\vskip 0.1cm
		\noindent 
		\textbf{Case 3:}  \emph{$d=r=3$.} 
		For the case $d=r=3$, we estimate  $\langle\B(\boldsymbol{y}_1)-\B(\boldsymbol{y}_2),\boldsymbol{y}_1-\boldsymbol{y}_2\rangle$ as  (\cite[pp. 154]{Kim-2021})
		\begin{align}
			\langle\B(\boldsymbol{y}_1)-\B(\boldsymbol{y}_2),\boldsymbol{y}_1-\boldsymbol{y}_2\rangle& =-\langle\B(\boldsymbol{y}_1-\boldsymbol{y}_2,\boldsymbol{y}_1),\boldsymbol{y}_1-\boldsymbol{y}_2\rangle\nonumber\\& =\langle\B(\boldsymbol{y}_1-\boldsymbol{y}_2,\boldsymbol{y}_1-\boldsymbol{y}_2),\boldsymbol{y}_1\rangle \nonumber\\& =\langle\B(\boldsymbol{y}_1-\boldsymbol{y}_2,\boldsymbol{y}_1-\boldsymbol{y}_2),\boldsymbol{y}_2\rangle \nonumber\\& =\frac{1}{2} \langle\B(\boldsymbol{y}_1-\boldsymbol{y}_2,\boldsymbol{y}_1-\boldsymbol{y}_2),\boldsymbol{y}_1+\boldsymbol{y}_2\rangle. 
		\end{align}
		Using H\"older's and Young's inequalities, we find 
		\begin{align}\label{261-0}
			|\langle\B(\boldsymbol{y}_1)-\B(\boldsymbol{y}_2),\boldsymbol{y}_1
			-\boldsymbol{y}_2\rangle|
			&\leq\frac{1}{2} \left|\langle\B(\boldsymbol{y}_1-\boldsymbol{y}_2,\boldsymbol{y}_1-\boldsymbol{y}_2),\boldsymbol{y}_1+\boldsymbol{y}_2\rangle\right|
			\nonumber\\&\leq
			\frac12\|\nabla(\boldsymbol{y}_1-\boldsymbol{y}_2)\|_{\mathbb{H}}\|(\boldsymbol{y}_1-\boldsymbol{y}_2)(\boldsymbol{y}_1+\boldsymbol{y}_2)\|_{\mathbb{H}}\nonumber\\&\leq \theta\mu\|\nabla(\boldsymbol{y}_1-\boldsymbol{y}_2)\|_{\mathbb{H}}^2+\frac{1}{16\theta\mu} \|(\boldsymbol{y}_1-\boldsymbol{y}_2)|\boldsymbol{y}_1+\boldsymbol{y}_2|\|_{\mathbb{H}}^2 
			\nonumber\\&\leq  \theta\mu\|\nabla(\boldsymbol{y}_1-\boldsymbol{y}_2)\|_{\mathbb{H}}^2+\frac{1}{8\theta\mu} \||\boldsymbol{y}_1|(\boldsymbol{y}_1-\boldsymbol{y}_2)\|_{\mathbb{H}}^2 \nonumber\\&\quad+\frac{1}{8\theta\mu} \||\boldsymbol{y}_2|(\boldsymbol{y}_1-\boldsymbol{y}_2)\|_{\mathbb{H}}^2,
		\end{align}
	for some $0<\theta\leq 1$.	We infer from   \eqref{2.23} that
		\begin{align}\label{261-1}
			&\beta\langle\mathcal{C}(\boldsymbol{y}_1)-\mathcal{C}(\boldsymbol{y}_2),\boldsymbol{y}_1-\boldsymbol{y}_2\rangle\geq \frac{\beta}{2}\||\boldsymbol{y}_1|(\boldsymbol{y}_1-\boldsymbol{y}_2)\|_{\mathbb{H}}^2 + \frac{\beta}{2}\||\boldsymbol{y}_2|(\boldsymbol{y}_1-\boldsymbol{y}_2)\|_{\mathbb{H}}^2.
		\end{align}
		Using \eqref{261-0} and \eqref{261-1} in \eqref{346}, we arrive at 
		\begin{align}\label{372-1}
			&\|\boldsymbol{y}_1(t)-\boldsymbol{y}_2(t)\|_{\mathbb{H}}^2+2(1-\theta)\mu \int_0^t\|\nabla(\boldsymbol{y}_1(s)-\boldsymbol{y}_2(s))\|_{\mathbb{H}}^2\d s\nonumber\\&\quad+\left(\beta-\frac{1}{4\theta\mu}\right)\left[\int_0^t\||\boldsymbol{y}_1(s)|(\boldsymbol{y}_1(s)-\boldsymbol{y}_2(s))\|_{\mathbb{H}}^2 \d s+\int_0^t\||\boldsymbol{y}_1(s)|(\boldsymbol{y}_1(s)-\boldsymbol{y}_2(s))\|_{\mathbb{H}}^2 \d s\right]\nonumber\\&\leq\|\boldsymbol{y}_1(0)-\boldsymbol{y}_2(0)\|_{\mathbb{H}}^2,
		\end{align}
		for all $t\in[0,T]$. Therefore, for $4\beta\mu\geq 1$, the uniqueness follows for $d=r=3$. 
	\end{proof}

		\section{Simultaneous approximation and energy equality for CBF equations in unbounded domains}\setcounter{equation}{0}\label{sec-4} 
		In this section, we extend the approximation result from $\mathbb{R}^d$ to general unbounded domains $\Omega\subset\mathbb{R}^d$, $d\geq 2$, with nonempty boundary and of uniform $\mathrm{C}^{1,1}$-type. Our construction is based on the resolvent of the Stokes operator and yields simultaneous approximation in the relevant Sobolev and Lebesgue spaces. We next establish an independent version of the generalized Lions-Magenes lemma. Combining this result with the simultaneous approximation theorem, we prove the energy equality for Leray-Hopf weak solutions of the CBF equations on general unbounded domains.
		
			\subsection{General unbounded domains} 
		Our next aim is to provide the simultaneous approximation result on general unbounded domains on $\mathbb{R}^d$ by using the Stokes resolvent problem given in \cite{Farwig-2005,Farwig-2009}. 
		
		\begin{theorem}[Resolvent approximation in $\mathbb{V}\cap \mathbb{L}^p_\sigma(\Omega)$]
			\label{main-2}
			Let $\Omega\subset\mathbb{R}^d$, $d\geq 2,$ be a uniform $\mathrm{C}^{1,1}$-domain and let
			$\mathcal{L}$ denote the Stokes operator on $\mathbb{H}= \mathbb{L}^2_\sigma(\Omega)$. Let us set 
			$
			\mathbb{V}:=\mathbb{H}^1_{0,\sigma}(\Omega)=\D(\mathcal{L}^{1/2})
			$
			and for $2<p<\infty,$  define
			$
			\widetilde{\mathbb{L}}^p_\sigma(\Omega)
			:=
			\mathbb{L}^p_\sigma(\Omega)\cap \mathbb{L}^2_\sigma(\Omega).
			$
			Let $\widetilde{\mathcal{A}}_p$ denote the realization of $\mathcal{L}$ on
			$\widetilde{\mathbb L}^p_\sigma(\Omega)$. Assume that the resolvents of $\widetilde{\mathcal{A}}_p$ and $\mathcal{L}$ are consistent, that is,
			\begin{align*}
			(\mathrm{I}+\lambda \mathcal{L})^{-1}\boldsymbol{f}
			=
			(\mathrm{I}+\lambda\widetilde{\mathcal{A}}_p)^{-1}\boldsymbol{f},
			\ 
			\boldsymbol{f}\in\widetilde{\mathbb{L}}^p_\sigma(\Omega),\  \lambda>0.
			\end{align*}
			For each $n\in\mathbb{N}$, define
			\begin{align*}
			\mathcal{P}_n:=(\mathrm{I}+n^{-1}\mathcal{L})^{-1}.
			\end{align*}
			Then, for every
			$
			\boldsymbol{y}\in \mathbb{V}\cap \mathbb{L}^p_\sigma(\Omega),
			$ we have 
			\begin{align*}
			\mathcal{P}_n \boldsymbol{y} \in \D(\mathcal{L})
			\ \text{ and } \ 
			\sup_{n\ge1}
			\|\mathcal{P}_n\|_{\mathfrak{L}(\widetilde{\mathbb{L}}^p_\sigma(\Omega))}
			<\infty.
			\end{align*}
			Moreover
			\begin{align}\label{eqn-bound-0}
				\|\mathcal{P}_n\boldsymbol{y}-\boldsymbol{y}\|_{\mathbb{V}}
				+
				\|\mathcal{P}_n\boldsymbol{y}-\boldsymbol{y}\|_{\mathbb{L}^p(\Omega)}
				\to 0\ \text{ as }\ n\to\infty. 
			\end{align}
		\end{theorem}
		
		\begin{proof}
			We divide the proof into two steps.
			
			\vskip 0.1cm 
			\noindent
			\textbf{Step 1:} \emph{Uniform boundedness and convergence in
				$\widetilde{\mathbb{L}}^p_\sigma(\Omega)$.}
			From \cite[Theorem 1.4]{Farwig-2009} (see also \cite[Theorem 2.2]{Kunstmann-2008}), we infer that the operator $\varepsilon+\widetilde{\mathcal{A}}_p$ is sectorial of type $0,$ for every $\varepsilon>0$. Therefore, we immediately have 
			$$\mathcal{T}_p:=\frac{1}{2}+\widetilde{\mathcal{A}}_p$$
			is sectorial of type \(0\) and 
			\begin{align*}
			\sigma(\mathcal{T}_p)\subseteq\Sigma_0=[0,\infty),
			\end{align*}
			and, for every $\theta\in(0,\pi)$, there exists $C_{p,\theta}>0$ such that
			\begin{align*}
			\sup_{\lambda\notin\Sigma_\theta}
			\left\|
			\lambda(\lambda-\mathcal{T}_p)^{-1}
			\right\|_{\mathfrak{L}(\widetilde{\mathbb{L}}^p_\sigma(\Omega))}
			\le C_{p,\theta}.
			\end{align*}
			We fix $\theta=\pi/2$. Since
			$\Sigma_{\pi/2}=
			\left\{
			re^{i\varphi}:r\geq0,\ |\varphi|\leq\frac{\pi}{2}
			\right\}
			$ (see Definition \ref{def-sect}), 
			every negative real number lies outside $\Sigma_{\pi/2}$. Indeed, for
			$\mu>0$, $-\mu=\mu e^{i\pi},$ so that $|\arg(-\mu)|=\pi>\frac{\pi}{2}.$
			Hence, by the sectorial resolvent estimate, we have 
			\begin{align*}
			\left\|
			(-\mu)(-\mu-\mathcal{T}_p)^{-1}
			\right\|_{\mathcal L(\widetilde{\mathbb{L}}^p_\sigma(\Omega))}
			\leq C_{p,\pi/2},
			\  \mu>0.
			\end{align*}
			Moreover,
			$-\mu-\mathcal{T}_p=-(\mu+\mathcal{T}_p),$
			and therefore $(-\mu-\mathcal{T}_p)^{-1}=-(\mu+\mathcal{T}_p)^{-1}.$
			Consequently, we have 
			$(-\mu)(-\mu-\mathcal{T}_p)^{-1}=\mu(\mu+\mathcal{T}_p)^{-1}.$
			It follows that
			\begin{align}\label{eqn-bound-1}
				\sup_{\mu>0}
				\left\|
				\mu(\mu+\mathcal{T}_p)^{-1}
				\right\|_{\mathfrak{L}(\widetilde{\mathbb{L}}^p_\sigma(\Omega))}
				\leq C_p,
			\end{align}
			where 
			$C_p=C_{p,\pi/2}.$  For $n\ge1$, let us set
			$\mu_n:=n-\frac12.$
			Then, we have  
			\begin{align*}
			\mu_n\ge\frac12\ \text{ and	}\ n+\widetilde{\mathcal{A}}_p=\mu_n+\mathcal{T}_p.
			\end{align*}
			Since
			$\frac{n}{n-\frac12}\le2,	\  n\geq1 $, one gets 
			\begin{align*}
				\left\|
				n(n+\widetilde{\mathcal{A}}_p)^{-1}
				\right\|_{\mathfrak{L}(\widetilde{\mathbb{L}}^p_\sigma(\Omega))}
				&=
				\frac{n}{\mu_n}
				\left\|
				\mu_n(\mu_n+\mathcal{T}_p)^{-1}
				\right\|_{\mathfrak{L}(\widetilde{\mathbb{L}}^p_\sigma(\Omega))}
				\le 2C_p,
			\end{align*}
			where we have used \eqref{eqn-bound-1}. Therefore, it is immediate that 
			\begin{align}\label{eqn-bound-2}
				\sup_{n\ge1}
				\left\|
				n(n+\widetilde{\mathcal{A}}_p)^{-1}
				\right\|_{\mathfrak{L}(\widetilde{\mathbb{L}}^p_\sigma(\Omega))}
				\le2C_p.
			\end{align}
			By the assumed consistency of the $\mathbb{L}^2_\sigma$-and
			$\widetilde{\mathbb{L}}^p_\sigma$-realizations, for every
			$\boldsymbol{f}\in\widetilde{\mathbb{L}}^p_\sigma(\Omega)$ and every $\lambda>0$, we have (\cite[Theorem 1.3]{Farwig-2009})
			\begin{align*}
			(\mathrm{I}+\lambda\mathcal{L})^{-1}\boldsymbol{f}=
			(\mathrm{I}+\lambda\widetilde{\mathcal{A}}_p)^{-1}\boldsymbol{f}.
			\end{align*}
			In particular, with $\lambda=n^{-1}$, we define 
			\begin{align}\label{eqn-bound-3}
				\mathcal{P}_n\boldsymbol{f}
				=
				(\mathrm{I}+n^{-1}\mathcal{L})^{-1}\boldsymbol{f}
				=
				(\mathrm{I}+n^{-1}\widetilde{\mathcal{A}}_p)^{-1}\boldsymbol{f}
				=
				n(n+\widetilde{\mathcal{A}}_p)^{-1}\boldsymbol{f}.
			\end{align}
			We infer from \eqref{eqn-bound-2} that 
			\begin{align}\label{eqn-bound-4}
				\sup_{n\ge1}
				\|\mathcal{P}_n\|_{\mathfrak{L}(\widetilde{\mathbb{L}}^p_\sigma(\Omega))}
				\le2C_p.
			\end{align}
			Let us now show  that as $n\to\infty$ 
			\begin{align}\label{eqn-bound-5}
				\mathcal{P}_n\boldsymbol{f}\to \boldsymbol{f} 
				\ \text{ in }\ \widetilde{\mathbb{L}}^p_\sigma(\Omega),
			\end{align}
			for every $\boldsymbol{f}\in\widetilde{\mathbb{L}}^p_\sigma(\Omega)$. First let $\boldsymbol{f}\in\mathrm{D}(\widetilde{\mathcal{A}}_p)$. Then, we have 
			\begin{align*}
				\mathcal{P}_n\boldsymbol{f}-\boldsymbol{f}
				=
				\left[n(n+\widetilde{\mathcal{A}}_p)^{-1}-\mathrm{I}\right]\boldsymbol{f}=
				-(n+\widetilde{\mathcal{A}}_p)^{-1}\widetilde{\mathcal{A}}_p \boldsymbol{f}=
				-\frac1n\mathcal{P}_n\widetilde{\mathcal{A}}_p \boldsymbol{f}.
			\end{align*}
			Hence, by using \eqref{eqn-bound-4}, we obtain 
			\begin{align}\label{eqn-bound-6}
				\|\mathcal{P}_n\boldsymbol{f}-\boldsymbol{f}\|_{\widetilde{\mathbb{L}}^p}
				\leq
				\frac{2C_p}{n}\|\widetilde{\mathcal{A}}_p \boldsymbol{f}\|_{\widetilde{\mathbb{L}}^p}
				\to 0\  \text{ as } \ n\to\infty,
			\end{align}
			for every $\boldsymbol{f}\in\mathrm{D}(\widetilde{\mathcal{A}}_p)$. 
			Since $-\widetilde{\mathcal{A}}_p$ generates a $\mathrm{C}_0$-semigroup on
			$\widetilde{\mathbb{L}}^p_\sigma(\Omega)$, its generator $\widetilde{\mathcal{A}}_p$ is
			densely defined (\cite[Theorem 1.3]{Farwig-2009}). Thus $\mathrm{D}(\widetilde{\mathcal{A}}_p)$ is dense in
			$\widetilde{\mathbb{L}}^p_\sigma(\Omega)$. Let
			$\boldsymbol{f}\in\widetilde{\mathbb{L}}^p_\sigma(\Omega)$ and choose
			$\boldsymbol{f}_k\in\mathrm{D}(\widetilde{\mathcal{A}}_p)$ such that
			\begin{align}\label{eqn-bound-7}
				\|\boldsymbol{f}_k-\boldsymbol{f}\|_{\widetilde{\mathbb{L}}^p}\to 0\ 
				\text{ as }\ k\to\infty. 
			\end{align}
			For every $n,k$, using \eqref{eqn-bound-4}, we find 
			\begin{align}\label{eqn-bound-8}
				\|\mathcal{P}_n\boldsymbol{f}-\boldsymbol{f}\|_{\widetilde{\mathbb{L}}^p}
				&\le
				\|\mathcal{P}_n(\boldsymbol{f}-\boldsymbol{f}_k)\|_{\widetilde{\mathbb{L}}^p}
				+\|\mathcal{P}_n\boldsymbol{f}_k-\boldsymbol{f}_k\|_{\widetilde{\mathbb{L}}^p}
				+\|\boldsymbol{f}_k-\boldsymbol{f}\|_{\widetilde{\mathbb{L}}^p}\nonumber\\
				&\le
				(2C_p+1)\|\boldsymbol{f}-\boldsymbol{f}_k\|_{\widetilde{\mathbb{L}}^p}
				+\|\mathcal{P}_n\boldsymbol{f}_k-\boldsymbol{f}_k\|_{\widetilde{\mathbb{L}}^p}.
			\end{align}
			Using the convergence \eqref{eqn-bound-7},	given $\delta>0$, choose $k$ sufficiently large so that
			\begin{align*}
			(2C_p+1)\|\boldsymbol{f}-\boldsymbol{f}_k\|_{\widetilde{\mathbb{L}}^p}<\frac{\delta}{2}.
			\end{align*}
			For this fixed $k$, \eqref{eqn-bound-6} gives $N$ such that
			\begin{align*}
			\|\mathcal{P}_n\boldsymbol{f}_k-\boldsymbol{f}_k\|_{\widetilde{\mathbb{L}}^p}<\frac{\delta}{2},
			\ \text{ for all }\  n\ge N.
			\end{align*}
			Therefore, we infer from \eqref{eqn-bound-8} that 
			\begin{align*}
			\|\mathcal{P}_n\boldsymbol{f}-\boldsymbol{f}\|_{\widetilde{\mathbb{L}}^p}
			<\delta,
			\ \text{ for all }\  n\ge N,
			\end{align*}
			and the convergence \eqref{eqn-bound-5} follows. 
			
			Now let us consider 
			$
			\boldsymbol{y}\in \mathbb{V}\cap \mathbb{L}^p_\sigma(\Omega).
			$
			Since $\mathbb{V}=\D(\mathcal{L}^{1/2})\subset \mathbb{L}^2_\sigma(\Omega)$,
			we have
			$
			\boldsymbol{y}\in \mathbb{L}^2_\sigma(\Omega)\cap \mathbb{L}^p_\sigma(\Omega)
			=\widetilde{\mathbb{L}}^p_\sigma(\Omega).
			$
			Thus \eqref{eqn-bound-5} gives
			$
			\mathcal{P}_n\boldsymbol{y}\to \boldsymbol{y}
			\ \text{ in }\ \widetilde{\mathbb{L}}^p_\sigma(\Omega),
			$
			and in particular
			\begin{align}\label{eqn-bound-9}
				\mathcal{P}_n\boldsymbol{y}\to \boldsymbol{y}
				\ \text{ in } \ \mathbb{L}^p(\Omega)\ \text{ as }\ n\to\infty. 
			\end{align}
			
			\medskip
			\noindent
			\textbf{Step 2:}\emph{ Convergence in $\mathbb{V}$.}
			Since $\mathcal{L}$ is nonnegative and self-adjoint on
			$\mathbb{L}^2_\sigma(\Omega)$ and $\mathbb{V}=D(\mathcal{L}^{1/2})$, the famous result from the functional calculus called the spectral theorem (\cite[pp. 151]{Amrein-2009}, \cite[Theorem 5.7]{Schmudgen-2012}) gives
			\begin{align*}
			\mathcal{P}_n=(\mathrm{I}+n^{-1}\mathcal{L})^{-1}=\varphi_n(\mathcal{L}),
			\ \text{ where }\ 
			\varphi_n(\lambda):=\frac{1}{1+\lambda/n}.
			\end{align*}
			Let $\boldsymbol{y}\in \mathbb{V}$ and let $\mu_{\boldsymbol{y}}$ be the spectral measure of $\boldsymbol{y}$ associated with $\mathcal{L}$. Then, we have 
			\begin{equation*}
			\begin{aligned}
				\|\mathcal{L}^{1/2}(\mathcal{P}_n\boldsymbol{y}-\boldsymbol{y})\|_{\mathbb{L}^2(\Omega)}^2
				&=
				\int_{[0,\infty)}
				\lambda
				\left|
				\frac{1}{1+\lambda/n}-1
				\right|^2
				\d\mu_{\boldsymbol{y}}(\lambda).
			\end{aligned}
			\end{equation*}
			The integrand converges pointwise to zero and satisfies
			\begin{align*}
			0\leq\lambda\left|\frac{1}{1+\lambda/n}-1\right|^2\leq\lambda.
			\end{align*}
			Moreover, we have 
			\begin{align*}
			\int_{[0,\infty)}\lambda\,\d\mu_{\boldsymbol{y}}(\lambda)=
			\|\mathcal{L}^{1/2}\boldsymbol{y}\|_{\mathbb{L}^2(\Omega)}^2<\infty.
			\end{align*}
			Hence, by the Lebesgue dominated convergence theorem, we conclude 
			\begin{align}\label{eqn-bound-10}
				\mathcal{L}^{1/2}\mathcal{P}_n\boldsymbol{y}\to \mathcal{L}^{1/2}\boldsymbol{y}
				\ \text{ in }\ \mathbb{L}^2(\Omega).
			\end{align}
			In a similar way, we arrive at 
			\begin{align}\label{eqn-bound-11}
				\|\mathcal{P}_n\boldsymbol{y}-\boldsymbol{y}\|_{\mathbb{L}^2(\Omega)}^2
				=
				\int_{[0,\infty)}
				\left|
				\frac{1}{1+\lambda/n}-1
				\right|^2
				\d\mu_{\boldsymbol{y}}(\lambda)
				\to 0,
			\end{align}
			again by  the Lebesgue dominated convergence theorem, since the integrand is bounded by $1$.
			Consequently,
			\begin{align*}
			\|\mathcal{P}_n\boldsymbol{y}-\boldsymbol{y}\|_{\mathbb{L}^2(\Omega)}
			+
			\|\mathcal{L}^{1/2}(\mathcal{P}_n\boldsymbol{y}-\boldsymbol{y})\|_{\mathbb{L}^2(\Omega)}
			\to 0 \ \text{ as }\ n\to\infty. 
			\end{align*}
			Since the graph norm of $\mathcal{L}^{1/2}$ is equivalent to the $\mathbb{V}$-norm,
			we obtain
			\begin{align}\label{eqn-bound-12}
				\mathcal{P}_n\boldsymbol{y}\to  \boldsymbol{y}
				\ \ \text{ in }\ \mathbb{V}\ \text{ as }\ n\to\infty. 
			\end{align}
			Combining \eqref{eqn-bound-9} and \eqref{eqn-bound-12}, we conclude the proof of \eqref{eqn-bound-0}. 
		\end{proof}

		\begin{remark}
			One can obtain	\eqref{eqn-bound-1} by using \cite[Theorem 1.3]{Farwig-2009}, which states that for
			every $\varepsilon>0$, there exists a constant
			$M=M(\varepsilon,q,\alpha,\beta,K)$  such that
			\begin{align}\label{rem-1}
				\|e^{-t\widetilde{\mathcal{A}}_q}\boldsymbol{f}\|_{\widetilde{\mathbb{L}}^q}
				\leq
				M e^{\varepsilon t}\|\boldsymbol{f}\|_{\widetilde{\mathbb{L}}^q},
				\ 
				t>0,
			\end{align}
			where  $\alpha,\beta,K$ are the constants from Definition \ref{def-c11}.  Since $\mathcal{T}_p=\frac12+\widetilde{\mathcal{A}}_p$, the semigroup generated by $-\mathcal{T}_p$ is
			$e^{-t\mathcal{T}_p}=e^{-t/2}e^{-t\widetilde{\mathcal{A}}_p} $. Using \eqref{rem-1}, we find
			\begin{align}
				\|e^{-t\mathcal{T}_p}\|_{\mathfrak{L}(\widetilde{\mathbb{L}}^p_\sigma(\Omega))}
				= e^{-t/2}	\|e^{-t\widetilde{\mathcal{A}}_q}\|_{\mathfrak{L}(\widetilde{\mathbb{L}}^p_\sigma(\Omega))}
				\leq 	e^{-t/2}M e^{t/2}
				=
				M,
				\ t\geq0.
			\end{align}
			Hence the shifted semigroup $(e^{-t\mathcal{T}_p})_{t\geq0}$ is bounded. Consequently, for $\mu>0$ and $\boldsymbol{f}\in\widetilde{\mathbb{L}}^p_\sigma(\Omega)$, the
			Laplace transform representation of the resolvent gives
			\begin{align*}
				(\mu+\mathcal{T}_p)^{-1}\boldsymbol{f}
				=
				\int_0^\infty e^{-\mu t}e^{-t\mathcal{T}_p}\boldsymbol{f}\d t.
			\end{align*}
			Therefore, we have 
			\begin{align*}
				\left\|
				\mu(\mu+\mathcal{T}_p)^{-1}\boldsymbol{f}
				\right\|_{\widetilde{\mathbb{L}}^q}
				\leq
				\mu\int_0^\infty e^{-\mu t}
				\left\|e^{-t\mathcal{T}_p}\boldsymbol{f}\right\|_{\widetilde{\mathbb{L}}^q}\d t
				\leq
				M\mu\int_0^\infty e^{-\mu t}\left\|\boldsymbol{f}\right\|_{\widetilde{\mathbb{L}}^q}\d t=M\left\|\boldsymbol{f}\right\|_{\widetilde{\mathbb{L}}^q},
			\end{align*}
			which implies \eqref{eqn-bound-1}. 
		\end{remark}

		\begin{remark}
			One may naturally ask whether the compact operator $L$ constructed in \cite{Brzezniak-2013} can be used to obtain simultaneous approximations. We note that the results established in \cite[Lemma 2.4]{Brzezniak-2013} hold for functions in $
			\D(L)\subset \mathbb{U}\cong\D(L^{1/2})\subset \mathbb{V}_s\subset \mathbb{V}\cap \mathbb{L}^p_{\sigma}\subset \mathbb{H},$
			where
		$s>\max\left\{1,\frac{d(p-2)}{2p}\right\},$ and $\mathbb{V}_s$ denotes the closure of $\mathbb{C}_{0,\sigma}^{\infty}(\Omega)$ in $\mathbb{H}^s_0(\Omega)$. Moreover, the embedding $\mathbb{U}\hookrightarrow \mathbb{V}_s$ is compact. Although we are able to establish separately that $\D(L)$ is dense in $\mathbb{V}$ and in $\mathbb{L}^p_{\sigma}$, we have not been able to show that $\D(L)$ is dense in the intersection $\mathbb{V}\cap \mathbb{L}_{\sigma}^p$. Consequently, the compactness result for $L$ obtained in \cite{Brzezniak-2013} does not, by itself, immediately yield the simultaneous approximation required here. Similarly, the simultaneous  approximation results in \cite{Brzezniak-2019,Hornung-2018} rely on the spectral decomposition of a self-adjoint operator and are therefore restricted to bounded domains, where the relevant operator has a discrete spectrum.
		\end{remark}

\subsection{Generalized Lions-Magenes lemma}
	The aim of this section is to establish the energy equality for the
	CBF equations on general unbounded domains. Our approach combines a
	generalization of the classical Lions-Magenes lemma (\cite[Theorem 3.1, Chapter 1]{Lions-1972}, \cite[Lemma 1.2, Chapter 3]{Temam-1984}, \cite[Proposition 23.23]{Zeidler-1990}) with the 
	approximation result obtained in Theorem \ref{main-2}. Let us first establish a generalization of the classical Lions-Magenes
	lemma. Although this result was stated in \cite[Theorem 1.8, Chapter I]{Chepyzhov-2002} without assuming a Gelfand triple and without providing a proof, to the best of our knowledge, no proof of this result has appeared elsewhere in the literature. Therefore,  we provide a proof here for completeness.

Let $\mathcal{H}$ be a Hilbert space, and let $\mathcal{V}$ and
$\mathcal{E}$ be Banach spaces continuously embedded into $\mathcal{H}$.
Identifying $\mathcal{H}$ with its dual, assume that
\begin{align*}
\mathcal{V}\hookrightarrow\mathcal{H}\hookrightarrow\mathcal{V}',
\ 
\mathcal{E}\hookrightarrow\mathcal{H}\hookrightarrow\mathcal{E}'
\end{align*}
continuously. Define
\begin{align*}
\mathcal{X}:=\mathcal{V}\cap\mathcal{E},
\ 
\|\boldsymbol{y}\|_{\mathcal{X}}
:=
\|\boldsymbol{y}\|_{\mathcal{V}}
+
\|\boldsymbol{y}\|_{\mathcal{E}},
\end{align*}
and assume that \emph{$\mathcal{X}$ is dense in $\mathcal{H}$}. We identify
\begin{align*}
\mathcal{X}'
=
(\mathcal{V}\cap\mathcal{E})'
\cong
\mathcal{V}'+\mathcal{E}'
\end{align*}
in the usual intersection-sum sense (see Subsection \ref{fun-set}). Thus, if
\begin{align*}
\boldsymbol f=\boldsymbol f_1+\boldsymbol f_2,
\ 
\boldsymbol f_1\in\mathcal{V}',
\ 
\boldsymbol f_2\in\mathcal{E}',
\end{align*}
then
\begin{align*}
\langle\boldsymbol f,\boldsymbol{y}\rangle_{\mathcal{X}',\mathcal{X}}
=
\langle\boldsymbol f_1,\boldsymbol{y}\rangle_{\mathcal{V}',\mathcal{V}}
+
\langle\boldsymbol f_2,\boldsymbol{y}\rangle_{\mathcal{E}',\mathcal{E}},
\ 
\boldsymbol v\in\mathcal{X}.
\end{align*}
Therefore, we infer 
\begin{align}\label{eqn-gelfand}
	\mathcal{X}\hookrightarrow\mathcal{H}
	\cong\mathcal{H}'
	\hookrightarrow\mathcal{X}'
\end{align}
is a Gelfand triple (or evolution triple), that is, a chain of continuous and dense embeddings. We denote by $\mathcal{D}(0,T)$ the space of all infinitely differentiable, real-valued functions with compact support in $(0,T)$. Given a Banach space $\mathbb{X}$, we denote by $\mathcal{D}'(0,T;\mathbb{X})$ the space of all continuous linear operators (distributions) from $\mathcal{D}(0,T)$ into $\mathbb{X}$. An element $\boldsymbol{y}\in\mathcal{D}'(0,T;\mathbb{X})$ is called an $\mathbb{X}$-valued distribution on $(0,T)$.

\begin{theorem}[Generalized Lions-Magenes lemma]\label{GLML}
	\label{thm:generalized-lions-magenes}
	Let $\mathcal{H}$ be a Hilbert space, and let $\mathcal{V}$ and $\mathcal{E}$ be Banach spaces satisfying the Gelfand triple \eqref{eqn-gelfand}. Let $1<p<\infty$, and set $p'=\frac{p}{p-1}$. Suppose that
	\begin{align*}
		\boldsymbol{y}\in
		\mathrm{L}^2(0,T;\mathcal{V})
		\cap
		\mathrm{L}^p(0,T;\mathcal{E})
	\end{align*}
	and
	\begin{align*}
		\boldsymbol{y}'=\boldsymbol{y}_1+\boldsymbol{y}_2
		\ \text{in }\mathcal{D}'(0,T;\mathcal{V}'+\mathcal{E}'),
	\end{align*}
	where
	\begin{align*}
		\boldsymbol{y}_1\in \mathrm{L}^2(0,T;\mathcal{V}'),
		\ 
		\boldsymbol{y}_2\in \mathrm{L}^{p'}(0,T;\mathcal{E}').
	\end{align*}
	Then $\boldsymbol{y}$ admits a unique representative, still denoted by
	$\boldsymbol{y}$, such that
	\begin{align*}
		\boldsymbol{y}\in \mathrm{C}([0,T];\mathcal{H}).
	\end{align*}
	Moreover, the map
	$t\mapsto\|\boldsymbol{y}(t)\|_{\mathcal{H}}^2$
	is absolutely continuous on $[0,T]$, and
	\begin{align*}
		\frac{\mathrm{d}}{\mathrm{d}t}\|\boldsymbol{y}(t)\|_{\mathcal{H}}^2
		=
		2\langle\boldsymbol{y}_1(t),\boldsymbol{y}(t)\rangle_{\mathcal{V}',\mathcal{V}}
		+
		2\langle\boldsymbol{y}_2(t),\boldsymbol{y}(t)\rangle_{\mathcal{E}',\mathcal{E}}
	\end{align*}
	for a.e.\ $t\in[0,T]$. Consequently, for every $0\leq s\leq t\leq T$, we have
	\begin{equation*}
		\begin{aligned}
			\|\boldsymbol{y}(t)\|_{\mathcal{H}}^2
			-
			\|\boldsymbol{y}(s)\|_{\mathcal{H}}^2
			=&
			2\int_s^t
			\langle\boldsymbol{y}_1(r),\boldsymbol{y}(r)\rangle_{\mathcal{V}',\mathcal{V}}
			\,\mathrm{d}r
			+
			2\int_s^t
			\langle\boldsymbol{y}_2(r),\boldsymbol{y}(r)\rangle_{\mathcal{E}',\mathcal{E}}
			\,\mathrm{d}r.
		\end{aligned}
	\end{equation*}
\end{theorem}
	
	\begin{proof}
		The proof is divided into the following number of steps:
		\vskip 2mm
		\noindent
		\textbf{Step 1:} \emph{Integrability and time regularity.}
		Since
	$
		\mathcal{X}=\mathcal{V}\cap\mathcal{E}
	$
		is continuously and densely embedded into $\mathcal{H}$, identifying
		$\mathcal{H}$ with $\mathcal{H}'$, we have the Gelfand triple
	$
		\mathcal{X}\hookrightarrow\mathcal{H}\hookrightarrow\mathcal{X}'.
	$
		Moreover,
	$
		\mathcal{X}'\cong\mathcal{V}'+\mathcal{E}'.
	$
		We first observe that
	$$
		\langle \boldsymbol{y}_1(\cdot),\boldsymbol{y}(\cdot)\rangle_{\mathcal{V}',\mathcal{V}}
		\in \mathrm{L}^1(0,T)
\	\text{ 	and }\ 
		\langle \boldsymbol{y}_2(\cdot),\boldsymbol{y}(\cdot)\rangle_{\mathcal{E}',\mathcal{E}}
		\in \mathrm{L}^1(0,T).
	$$
		Indeed, by H\"older's inequality, we have 
		\begin{align}\label{eqn-0}
			\int_0^T
			\left|
			\langle \boldsymbol{y}_1(t),\boldsymbol{y}(t)\rangle_{\mathcal{V}',\mathcal{V}}
			\right|\d t
			&\le
			\|\boldsymbol{y}_1\|_{\mathrm{L}^2(0,T;\mathcal{V}')}
			\|\boldsymbol{y}\|_{\mathrm{L}^2(0,T;\mathcal{V})}
			<\infty,
		\end{align}
		and
		\begin{align}\label{eqn-00}
			\int_0^T
			\left|
			\langle \boldsymbol{y}_2(t),\boldsymbol{y}(t)\rangle_{\mathcal{E}',\mathcal{E}}
			\right|\d t
			&\le
			\|\boldsymbol{y}_2\|_{\mathrm{L}^{p'}(0,T;\mathcal{E}')}
			\|\boldsymbol{y}\|_{\mathrm{L}^p(0,T;\mathcal{E})}
			<\infty.
		\end{align}
		We shall also need to regard $\boldsymbol{y}'$ as an $\mathrm{L}^1(0,T;\mathcal{X}')$
		function. Since the restriction maps
	$\mathcal{V}'\hookrightarrow\mathcal{X}',$ and 
	$\mathcal{E}'\hookrightarrow\mathcal{X}'$
		are continuous, and $(0,T)$ has finite measure, we have
		\begin{align*}
		\boldsymbol{y}_1\in \mathrm{L}^2(0,T;\mathcal{V}')
		\hookrightarrow \mathrm{L}^1(0,T;\mathcal{X}')
	\ \text{ and }\ 
		\boldsymbol{y}_2\in \mathrm{L}^{p'}(0,T;\mathcal{E}')
		\hookrightarrow \mathrm{L}^1(0,T;\mathcal{X}').
		\end{align*}
		Therefore, we deduce
		\begin{align*}
		\boldsymbol{y}'=\boldsymbol{y}_1+\boldsymbol{y}_2\in \mathrm{L}^1(0,T;\mathcal{X}').
		\end{align*}
		Furthermore,
	$
		\boldsymbol{y}\in \mathrm{L}^2(0,T;\mathcal{V})
		\hookrightarrow \mathrm{L}^2(0,T;\mathcal{H})
		\hookrightarrow \mathrm{L}^1(0,T;\mathcal{X}').
	$
		Consequently, we find 
		\begin{align*}
		\boldsymbol{y}\in\W^{1,1}(0,T;\mathcal{X}').
		\end{align*}
		In particular, $\boldsymbol{y}$ has an absolutely continuous
		$\mathcal{X}'$-valued representative and hence well-defined traces in
		$\mathcal{X}'$ at $t=0$ and $t=T$. This allows us to perform the
		reflection construction used below.
		\vskip 2mm
		\noindent
		\textbf{Step 2:} \emph{Extension by reflection and time regularization.}
		Choose $\sigma>0$ sufficiently small. We extend $\boldsymbol{y}$ from $(0,T)$
		to
	$I_\sigma:=(-\sigma,T+\sigma)$
		by reflection at the endpoints. Define
		\begin{align*}
		\widetilde{\boldsymbol{y}}(t)
		=
		\begin{cases}
			\boldsymbol{y}(-t), & -\sigma<t<0,\\[1mm]
			\boldsymbol{y}(t), & 0\le t\le T,\\[1mm]
			\boldsymbol{y}(2T-t), & T<t<T+\sigma.
		\end{cases}
	\end{align*}
		Likewise, for $i=1,2$, define
		\begin{align*}
		\widetilde{\boldsymbol{y}}_i(t)
		=
		\begin{cases}
			-\boldsymbol{y}_i(-t), & -\sigma<t<0,\\[1mm]
			\boldsymbol{y}_i(t), & 0<t<T,\\[1mm]
			-\boldsymbol{y}_i(2T-t), & T<t<T+\sigma.
		\end{cases}
		\end{align*}
	The absolute continuity of $\boldsymbol u$ in $\mathcal X'$ implies
$
	\widetilde{\boldsymbol u}(0-)
	=
	\boldsymbol u(0)
	=
	\widetilde{\boldsymbol u}(0+)
$
	and
$
	\widetilde{\boldsymbol u}(T-)
	=
	\boldsymbol u(T)
	=
	\widetilde{\boldsymbol u}(T+)
$
	in $\mathcal X'$. Thus $\widetilde{\boldsymbol u}$ has no jump at
	$t=0$ or $t=T$.	Reflection preserves the relevant Bochner norms. Therefore, we infer 
	\begin{align*}
		\widetilde{\boldsymbol{y}}
		\in
		\mathrm{L}^2(I_\sigma;\mathcal{V})
		\cap
		\mathrm{L}^p(I_\sigma;\mathcal{E}), \
		\widetilde{\boldsymbol{y}}_1
		\in \mathrm{L}^2(I_\sigma;\mathcal{V}'),
		\ 
		\widetilde{\boldsymbol{y}}_2
		\in \mathrm{L}^{p'}(I_\sigma;\mathcal{E}').
	\end{align*}
	We claim that
	\begin{align}\label{eqn-1}
		\widetilde{\boldsymbol{y}}'
		=
		\widetilde{\boldsymbol{y}}_1+\widetilde{\boldsymbol{y}}_2
		\ 
		\text{ in }\ \mathcal D'(I_\sigma;\mathcal{X}').
		\end{align}
		Indeed, on $(0,T)$ this is precisely the assumed identity.
		For $t<0$,
	$
		\widetilde{\boldsymbol{y}}(t)=\boldsymbol{y}(-t),
	$
		and hence, in the sense of $\mathcal{X}'$-valued distributions,
		\begin{align*}
		\widetilde{\boldsymbol{y}}'(t)
		=
		-\boldsymbol{y}'(-t)
		=
		-\boldsymbol{y}_1(-t)-\boldsymbol{y}_2(-t)
		=
		\widetilde{\boldsymbol{y}}_1(t)+\widetilde{\boldsymbol{y}}_2(t).
		\end{align*}
		Similarly, for $t>T$,
	$
		\widetilde{\boldsymbol{y}}(t)=\boldsymbol{y}(2T-t),
	$
		so that 
		\begin{align*}
		\widetilde{\boldsymbol{y}}'(t)
		=
		-\boldsymbol{y}'(2T-t)=	-\boldsymbol{y}_1(2T-t)-\boldsymbol{y}_2(2T-t)
		=
		\widetilde{\boldsymbol{y}}_1(t)+\widetilde{\boldsymbol{y}}_2(t).
		\end{align*}
		Since $\boldsymbol{y}\in\W^{1,1}(0,T;\mathcal{X}')$ has well-defined endpoint
		traces and the reflected function is continuous in $\mathcal{X}'$ at
		the reflection points, no Dirac masses occur at $0$ or $T$.
		Therefore \eqref{eqn-1} holds on the whole interval $I_\sigma$.
	\vskip 2mm
	\noindent
	\textbf{Step 3:} \emph{Energy identity for the regularized functions.}
	Let us now regularize the time. Let
		\begin{align*}
		\rho\in\mathrm{C}_c^\infty(\mathbb R),
		\ 
		\rho\ge0,
		\ 
		\rho(-r)=\rho(r),
		\ 
		\int_{\mathbb R}\rho(r)\,dr=1,
		\end{align*}
		be a standard symmetric mollifier, and put
		\begin{align*}
		\rho_\varepsilon(r)
		=
		\frac1\varepsilon
		\rho\left(\frac r\varepsilon\right).
		\end{align*}
		For $\varepsilon>0$ sufficiently small, define
		\[
		\boldsymbol{y}^\varepsilon
		=
		\rho_\varepsilon*\widetilde{\boldsymbol{y}},\ 
		\boldsymbol{y}_1^\varepsilon
		=
		\rho_\varepsilon*\widetilde{\boldsymbol{y}}_1,
			\ \text{ and }\ 
		\boldsymbol{y}_2^\varepsilon
		=
		\rho_\varepsilon*\widetilde{\boldsymbol{y}}_2.
		\]
		These functions are defined on a neighborhood of $[0,T]$. Since convolution commutes with distributional differentiation, \eqref{eqn-1}
		implies
		\begin{align}\label{eqn-2}
		(\boldsymbol{y}^\varepsilon)'
		=
		\boldsymbol{y}_1^\varepsilon+\boldsymbol{y}_2^\varepsilon
		\ \text{ in }\ \mathcal{X}'.
		\end{align}
			By the standard approximation properties of time mollification, we have 
	\begin{align}\label{eqn-3}
		\left\{
		\begin{aligned}
		&	\boldsymbol{y}^\varepsilon\to \boldsymbol{y}
			\ 
			\text{ in }\ \mathrm{L}^2(0,T;\mathcal{V}),\\
			&\boldsymbol{y}^\varepsilon\to \boldsymbol{y}
			\ 
			\text{ in }\ \mathrm{L}^p(0,T;\mathcal{E}),\\
			&
			\boldsymbol{y}_1^\varepsilon\to \boldsymbol{y}_1
			\ 
			\text{ in }\ \mathrm{L}^2(0,T;\mathcal{V}'),\\
		&	\boldsymbol{y}_2^\varepsilon\to \boldsymbol{y}_2
			\ 
			\text{ in }\mathrm{L}^{p'}(0,T;\mathcal{E}').
		\end{aligned}
		\right.
	\end{align}
		For each fixed $\varepsilon>0$ and $t\in[0,T]$, the Bochner
		integral defining $\boldsymbol{y}^\varepsilon(t)$ belongs simultaneously to
		$\mathcal{V}$ and $\mathcal{E}$. Thus, we have 
	$
		\boldsymbol{y}^\varepsilon(t)\in\mathcal{X}.
	$
		On the other hand,
	$
		\widetilde{\boldsymbol{y}}\in \mathrm{L}^2(I_\sigma;\mathcal{H}),
	$
		and therefore convolution in time gives
	$
		\boldsymbol{y}^\varepsilon\in \mathrm{C}^\infty([0,T];\mathcal{H}).
	$
		In particular,
	\begin{align}\label{eqn-4}
		(\boldsymbol{y}^\varepsilon)'(t)\in\mathcal{H}.
	\end{align}
		Now the canonical embedding
	$
		\mathcal{H}\hookrightarrow\mathcal{X}'
	$
		is characterized by
	$$
		\langle h,y\rangle_{\mathcal{X}',\mathcal{X}}
		=
		(h,y)_{\mathcal{H}},
		\ 
		h\in\mathcal{H},\  y\in\mathcal{X}.
	$$
		By \eqref{eqn-2} and \eqref{eqn-4}, the element
	$
		\boldsymbol{y}_1^\varepsilon(t)+\boldsymbol{y}_2^\varepsilon(t)\in\mathcal{X}'
	$
		is precisely the image in $\mathcal{X}'$ of the $\mathcal{H}$-element
		$(	\boldsymbol{y}^\varepsilon)'(t)$. Hence, we deduce for all $t\in[0,T]$ that 
		\begin{align}\label{eqn-5}
			\frac12\frac{\d}{\d t}
			\|	\boldsymbol{y}^\varepsilon(t)\|_{\mathcal{H}}^2
			&=
			\big((	\boldsymbol{y}^\varepsilon)'(t),	\boldsymbol{y}^\varepsilon(t)\big)_{\mathcal{H}}
		\nonumber	\\
			&=
			\left\langle
			\boldsymbol{y}_1^\varepsilon(t)+\boldsymbol{y}_2^\varepsilon(t),
				\boldsymbol{y}^\varepsilon(t)
			\right\rangle_{\mathcal{X}',\mathcal{X}}
		\nonumber	\\
			&=
			\left\langle
			\boldsymbol{y}_1^\varepsilon(t),	\boldsymbol{y}^\varepsilon(t)
			\right\rangle_{\mathcal{V}',\mathcal{V}}+
			\left\langle
			\boldsymbol{y}_2^\varepsilon(t),	\boldsymbol{y}^\varepsilon(t)
			\right\rangle_{\mathcal{E}',\mathcal{E}}.
		\end{align}
		This is the regularized energy identity.
	\vskip 2mm
	\noindent
	\textbf{Step 4:} \emph{Passage to an $\mathcal{H}$-continuous representative.}
	Let us now prove the Cauchy property of the regularizations in
			$\mathrm{C}([0,T];\mathcal{H})$.
		Choose a sequence
	$
		\varepsilon_n\downarrow0.
	$
		By \eqref{eqn-3} and the continuous embedding
		$\mathcal{V}\hookrightarrow\mathcal{H}$, we have 
		\begin{align}\label{eqn-6}
		\boldsymbol{y}^{\varepsilon_n}\to \boldsymbol{y}
		\ 
		\text{ in }\ \mathrm{L}^2(0,T;\mathcal{H}).
		\end{align}
		Passing to a subsequence, which we do not relabel, we may assume that
		\begin{align}\label{eqn-7}
		\boldsymbol{y}^{\varepsilon_n}(s)\to \boldsymbol{y}(s)
		\ \text{ in }\ \mathcal{H}
		\end{align}
		for almost every $s\in(0,T)$.
	Fix one such point
	$
		s_0\in(0,T).
	$
		For $n,m\in\mathbb N$, we define
		\[
		\boldsymbol{w}_{n,m}
		:=
		\boldsymbol{y}^{\varepsilon_n}-\boldsymbol{y}^{\varepsilon_m}.
		\]
		Then
		$
		\boldsymbol{w}_{n,m}\in \mathrm{C}^\infty([0,T];\mathcal{H})
	$
		and
	$
		\boldsymbol{w}_{n,m}(t)\in\mathcal{X}.
	$
		Moreover,
		\[
		\boldsymbol{w}_{n,m}'
		=
		\boldsymbol{y}_1^{\varepsilon_n}-\boldsymbol{y}_1^{\varepsilon_m}
		+
		\boldsymbol{y}_2^{\varepsilon_n}-\boldsymbol{y}_2^{\varepsilon_m}.
		\]
		Applying the same Hilbert-space chain-rule argument used in \eqref{eqn-5} to
		$	\boldsymbol{w}_{n,m}$, we obtain for all $t\in[0,T]$ that 
		\begin{align}
			\frac12\frac{\d}{\d t}
			\|	\boldsymbol{w}_{n,m}(t)\|_{\mathcal{H}}^2
			&=
			\left\langle
			\boldsymbol{y}_1^{\varepsilon_n}(t)-\boldsymbol{y}_1^{\varepsilon_m}(t),
				\boldsymbol{w}_{n,m}(t)
			\right\rangle_{\mathcal{V}',\mathcal{V}}
+
			\left\langle
			\boldsymbol{y}_2^{\varepsilon_n}(t)-\boldsymbol{y}_2^{\varepsilon_m}(t),
				\boldsymbol{w}_{n,m}(t)
			\right\rangle_{\mathcal{E}',\mathcal{E}}.
		\end{align}
	For $t\ge s_0$,	integrating from $s_0$ to $t$ gives
		\begin{align}
			\|	\boldsymbol{w}_{n,m}(t)\|_{\mathcal{H}}^2
			&=
			\|	\boldsymbol{w}_{n,m}(s_0)\|_{\mathcal{H}}^2
	+
			2\int_{s_0}^t
			\left\langle
			\boldsymbol{y}_1^{\varepsilon_n}(r)-\boldsymbol{y}_1^{\varepsilon_m}(r),
				\boldsymbol{w}_{n,m}(t)
			\right\rangle_{\mathcal{V}',\mathcal{V}}\d r
	\nonumber\\&\quad	+
			2\int_{s_0}^t
			\left\langle
			\boldsymbol{y}_2^{\varepsilon_n}(r)-\boldsymbol{y}_2^{\varepsilon_m}(r),
				\boldsymbol{w}_{n,m}(r)
			\right\rangle_{\mathcal{E}',\mathcal{E}}\d r.
		\end{align}
		For $t<s_0$, the same identity is integrated from $t$ to $s_0$.
		Consequently, for every $t\in[0,T]$, we find 
		\begin{align*}
			\|	\boldsymbol{w}_{n,m}(t)\|_{\mathcal{H}}^2
			&\le
			\|	\boldsymbol{w}_{n,m}(s_0)\|_{\mathcal{H}}^2
			+
			2\int_0^T
			\left|
			\left\langle
			\boldsymbol{y}_1^{\varepsilon_n}(r)-\boldsymbol{y}_1^{\varepsilon_m}(r),
				\boldsymbol{w}_{n,m}(r)
			\right\rangle_{\mathcal{V}',\mathcal{V}}
			\right|\d r
			\\
			&\quad+
			2\int_0^T
			\left|
			\left\langle
			\boldsymbol{y}_2^{\varepsilon_n}(r)-\boldsymbol{y}_2^{\varepsilon_m}(r),
				\boldsymbol{w}_{n,m}(r)
			\right\rangle_{\mathcal{E}',\mathcal{E}}
			\right|\d r.
		\end{align*}
		Taking the supremum over $t\in[0,T]$ and applying H\"older's
		inequality, we obtain
		\begin{align}\label{eqn-8}
			\sup_{0\le t\le T}
			\|\boldsymbol{w}_{n,m}(t)\|_{\mathcal{H}}^2
			&\le
			\|\boldsymbol{w}_{n,m}(s_0)\|_{\mathcal{H}}^2
			+
			2
			\|\boldsymbol{y}_1^{\varepsilon_n}-\boldsymbol{y}_1^{\varepsilon_m}\|_
			{\mathrm{L}^2(0,T;\mathcal{V}')}
			\|\boldsymbol{y}^{\varepsilon_n}-\boldsymbol{y}^{\varepsilon_m}\|_
			{\mathrm{L}^2(0,T;\mathcal{V})}
		\nonumber	\\
			&\quad+
			2
			\|\boldsymbol{y}_2^{\varepsilon_n}-\boldsymbol{y}_2^{\varepsilon_m}\|_
			{\mathrm{L}^{p'}(0,T;\mathcal{E}')}
			\,
			\|\boldsymbol{y}^{\varepsilon_n}-\boldsymbol{y}^{\varepsilon_m}\|_
			{\mathrm{L}^p(0,T;\mathcal{E})}.
		\end{align}
		By the choice of $s_0$, we infer from \eqref{eqn-7} that 
		\[
		\|\boldsymbol{w}_{n,m}(s_0)\|_{\mathcal{H}}\to0
		\ \text{ as }\ n,m\to\infty.
		\]
		Moreover, by using the convergences given in \eqref{eqn-3}, each of the remaining products on the
		right-hand side of \eqref{eqn-8} tends to zero. Therefore, we arrive at 
		\begin{align}\label{eqn-9}
		\lim_{n,m\to\infty}
		\sup_{0\le t\le T}
		\|\boldsymbol{y}^{\varepsilon_n}(t)-\boldsymbol{y}^{\varepsilon_m}(t)\|_{\mathcal{H}}
		=0.
		\end{align}
		Hence, the sequence $\{\boldsymbol{y}^{\varepsilon_n}\}_{n\geq1}$ is Cauchy in the Banach space
	$\mathrm{C}([0,T];\mathcal{H}).$
	Thus,  there exists
	$
		\boldsymbol{v}\in \mathrm{C}([0,T];\mathcal{H})
	$
		such that
		\begin{align}\label{eqn-10}
		\boldsymbol{y}^{\varepsilon_n}\to \boldsymbol{v}
		\ \text{ in }\ \mathrm{C}([0,T];\mathcal{H}).
		\tag{16}
		\end{align}
		On the other hand, by \eqref{eqn-6}, we have 
	$\boldsymbol{y}^{\varepsilon_n}\to \boldsymbol{y}
		\ \text{ in }\ \mathrm{L}^2(0,T;\mathcal{H}).$
		Since uniform convergence in $\mathcal{H}$ also implies convergence in
		$\mathrm{L}^2(0,T;\mathcal{H})$, uniqueness of the $\mathrm{L}^2$-limit yields
		\begin{align*}
		\boldsymbol{v}=\boldsymbol{y}
		\ \text{ for a.e. }\ t\in[0,T].
		\end{align*}
		Therefore $\boldsymbol{v}$ is an $\mathcal{H}$-continuous representative of $\boldsymbol{y}$.
		Replacing $\boldsymbol{y}$ by this representative, we obtain
	$
		\boldsymbol{y}\in \mathrm{C}([0,T];\mathcal{H}).$
		\vskip 2mm
		\noindent
		\textbf{Step 5:} \emph{Passage to the limit and the energy identity.}
		For each $n$, integrating \eqref{eqn-5} from $s$ to $t$, where
		$0\le s\le t\le T$, gives
		\begin{align}\label{eqn-11}
			\|\boldsymbol{y}^{\varepsilon_n}(t)\|_{\mathcal{H}}^2
			-
			\|\boldsymbol{y}^{\varepsilon_n}(s)\|_{\mathcal{H}}^2
			&=
			2\int_s^t
			\left\langle
			\boldsymbol{y}_1^{\varepsilon_n}(r),
			\boldsymbol{y}^{\varepsilon_n}(r)
			\right\rangle_{\mathcal{V}',\mathcal{V}}\d r+
			2\int_s^t
			\left\langle
			\boldsymbol{y}_2^{\varepsilon_n}(r),
			\boldsymbol{y}^{\varepsilon_n}(r)
			\right\rangle_{\mathcal{E}',\mathcal{E}}\d r.
		\end{align}
	The convergence \eqref{eqn-10} implies 
	$
		\boldsymbol{y}^{\varepsilon_n}\to \boldsymbol{y}
		\ \text{ in }\ \mathrm{C}([0,T];\mathcal{H}),
	$
		and hence, for every $s,t\in[0,T]$,
		\begin{align}
			&
			\|\boldsymbol{y}^{\varepsilon_n}(t)\|_{\mathcal{H}}^2
			-
			\|\boldsymbol{y}^{\varepsilon_n}(s)\|_{\mathcal{H}}^2
			\rightarrow
			\|\boldsymbol{y}(t)\|_{\mathcal{H}}^2
			-
			\|\boldsymbol{y}(s)\|_{\mathcal{H}}^2.
		\end{align}
			For the terms in the right hand side of \eqref{eqn-11}, we deduce 
		\begin{align}\label{eqn-12}
			&
			\left|\int_s^t
			\left[
			\left\langle
			\boldsymbol{y}_1^{\varepsilon_n}(r),\boldsymbol{y}^{\varepsilon_n}(r)
			\right\rangle_{\mathcal{V}',\mathcal{V}}
			-
			\left\langle
			\boldsymbol{y}_1(r),\boldsymbol{y}(r)
			\right\rangle_{\mathcal{V}',\mathcal{V}}
			\right]\d r\right|
			\nonumber\\&\quad+
			\left|\int_s^t
			\left[
			\left\langle
			\boldsymbol{y}_2^{\varepsilon_n}(r),\boldsymbol{y}^{\varepsilon_n}(r)
			\right\rangle_{\mathcal{E}',\mathcal{E}}
			-
			\left\langle
			\boldsymbol{y}_2(r),\boldsymbol{y}(r)
			\right\rangle_{\mathcal{E}',\mathcal{E}}
			\right]\d r\right|
		\nonumber	\\
			&\leq 
		\left|	\int_s^t
			\left\langle
			\boldsymbol{y}_1^{\varepsilon_n}(r)-\boldsymbol{y}_1(r),
			\boldsymbol{y}^{\varepsilon_n}(r)
			\right\rangle_{\mathcal{V}',\mathcal{V}}\d r\right|
		+\left|
			\int_s^t
			\left\langle
			\boldsymbol{y}_1(r),
			\boldsymbol{y}^{\varepsilon_n}(r)-\boldsymbol{y}(r)
			\right\rangle_{\mathcal{V}',\mathcal{V}}\d r\right|
			\nonumber\\&\quad +
				\left|	\int_s^t
			\left\langle
			\boldsymbol{y}_2^{\varepsilon_n}(r)-\boldsymbol{y}_2(r),
			\boldsymbol{y}^{\varepsilon_n}(r)
			\right\rangle_{\mathcal{E}',\mathcal{E}}\d r\right|
			+\left|
			\int_s^t
			\left\langle
			\boldsymbol{y}_2(r),
			\boldsymbol{y}^{\varepsilon_n}(r)-\boldsymbol{y}(r)
			\right\rangle_{\mathcal{E}',\mathcal{E}}\d r\right|
			\nonumber\\&\leq  	\|\boldsymbol{y}_1^{\varepsilon_n}-\boldsymbol{y}_1\|_{\mathrm{L}^2(0,T;\mathcal{V}')}
			\|\boldsymbol{y}^{\varepsilon_n}\|_{\mathrm{L}^2(0,T;\mathcal{V})}
		+
			\|\boldsymbol{y}_1\|_{\mathrm{L}^2(0,T;\mathcal{V}')}
			\|\boldsymbol{y}^{\varepsilon_n}-\boldsymbol{y}\|_{\mathrm{L}^2(0,T;\mathcal{V})}
			\nonumber\\&\quad+  \|\boldsymbol{y}_2^{\varepsilon_n}-\boldsymbol{y}_2\|_{\mathrm{L}^{p'}(0,T;\mathcal{E}')}
			\|\boldsymbol{y}^{\varepsilon_n}\|_{\mathrm{L}^p(0,T;\mathcal{E})}
		+
			\|\boldsymbol{y}_2\|_{\mathrm{L}^{p'}(0,T;\mathcal{E}')}
			\|\boldsymbol{y}^{\varepsilon_n}-\boldsymbol{y}\|_{\mathrm{L}^p(0,T;\mathcal{E})}
			\nonumber\\&\to 0\ \text{ as } \ n\to\infty. 
		\end{align}
		Passing to the limit $n\to\infty$ in \eqref{eqn-11} and using
		\eqref{eqn-12}  gives
		\begin{align}\label{eqn-13}
			\|\boldsymbol{y}(t)\|_{\mathcal{H}}^2-\|\boldsymbol{y}(s)\|_{\mathcal{H}}^2
			&=
			2\int_s^t
			\langle \boldsymbol{y}_1(r),\boldsymbol{y}(r)\rangle_{\mathcal{V}',\mathcal{V}}\d r
		+
			2\int_s^t
			\langle \boldsymbol{y}_2(r),\boldsymbol{y}(r)\rangle_{\mathcal{E}',\mathcal{E}}\d r
		\end{align}
		for every
	$0\le s\le t\le T.$
		Let us define
		\begin{align*}
		\mathcal{F}(\cdot)
		:=
		2\langle \boldsymbol{y}_1(\cdot),\boldsymbol{y}(\cdot)\rangle_{\mathcal{V}',\mathcal{V}}
		+
		2\langle \boldsymbol{y}_2(\cdot),\boldsymbol{y}(\cdot)\rangle_{\mathcal{E}',\mathcal{E}}.
		\end{align*}
	Using  the estimates \eqref{eqn-0} and \eqref{eqn-00}, we infer 
	$\mathcal{F}\in \mathrm{L}^1(0,T).$
		Taking $s=0$ in \eqref{eqn-13}, we obtain
		\begin{align*}
		\|\boldsymbol{y}(t)\|_{\mathcal{H}}^2
		=\|\boldsymbol{y}(0)\|_{\mathcal{H}}^2
		+\int_0^t \mathcal{F}(r)\,\d r,
		\  t\in[0,T].
		\end{align*}
		It follows that
	$t\mapsto\|\boldsymbol{y}(t)\|_{\mathcal{H}}^2$
		is absolutely continuous on $[0,T]$.
		By the fundamental theorem of calculus for absolutely continuous
		functions,
		\begin{align*}
		\frac{\d}{\d t}\|\boldsymbol{y}(t)\|_{\mathcal{H}}^2
		=\mathcal{F}(t)=
		2\langle \boldsymbol{y}_1(t),\boldsymbol{y}(t)\rangle_{\mathcal{V}',\mathcal{V}}
		+
		2\langle \boldsymbol{y}_2(t),\boldsymbol{y}(t)\rangle_{\mathcal{E}',\mathcal{E}},
		\end{align*}
		for a.e. $t\in[0,T]$. 
		
	Suppose that
	$\boldsymbol{v},\boldsymbol{w}\in \mathrm{C}([0,T];\mathcal{H})$
		are two continuous representatives of the same Bochner function
		$\boldsymbol{y}$. Then
	$\boldsymbol{v}(t)=\boldsymbol{w}(t),$
		for a.e. $t\in[0,T]$. Since the function
	$t\mapsto \boldsymbol{v}(t)-\boldsymbol{w}(t)$
		is continuous from $[0,T]$ into $\mathcal{H}$ and vanishes almost
		everywhere, it must vanish identically. Hence
	$\boldsymbol{v}(t)=\boldsymbol{w}(t)$
	 for every $t\in[0,T]$. Thus the continuous representative is unique, and the proof is
		complete.
		\end{proof}
		
Let us now prove the energy equality for CBF equations defined on general unbounded domains.
	\begin{theorem}\label{main-3}
		Let \(\Omega\subset\mathbb{R}^d\), \(d\geq 2\), be a domain with nonempty
		boundary of uniform \(\mathrm{C}^{1,1}\)-type. Let
		\(\boldsymbol{y}_0\in\mathbb{H}\) and
		\(\boldsymbol{f}\in \mathrm{L}^{2}(0,T;\mathbb{V}')\) be given. Then, for every
		\(r\in[1,\infty)\), there exists a Leray-Hopf weak solution to
		\eqref{abstract-CBF}.
		
		Moreover, if \(d=2\) and \(r\in[1,\infty)\), or if \(d=3\) and
		\(r\in[3,\infty)\), then every Leray-Hopf weak solution satisfies the
		energy equality
		\eqref{energy-equality}.
	
		Furthermore, under the same assumptions, namely, \(d=2\) with
		\(r\in[1,\infty)\) or \(d=3\) with \(r\in[3,\infty)\), the Leray-Hopf
		weak solution is unique, provided
	$
		4\beta\mu\geq 1
		\ \text{ when }\ d=r=3.
	$
	\end{theorem}
	\begin{proof}
		By taking $\mathcal{E}=\mathbb{L}^{r+1}_{\sigma}$
		and $\mathcal{V}=\mathbb{V}$ in Theorem \ref{GLML}, and using the density result established
		in Theorem~\ref{main-2}, we immediately obtain the
		 energy equality \eqref{energy-equality} for the two- and three-dimensional critical
		and supercritical CBF equations \eqref{abstract-CBF}.
		\end{proof}

		\appendix
		
		\section{Generalized Lions-Magenes chain rule}
		\label{Ap-chain-rule}
		We present a generalized version of the Lions-Magenes chain rule, extending \cite[Theorem 3.4]{Barbu1993}. This generalized formulation is particularly useful in optimal control theory, where it provides a rigorous framework for establishing the duality between linearized and adjoint variables and, consequently, plays a crucial role in deriving first-order necessary optimality conditions.
		
		\begin{theorem}[Generalized Lions-Magenes chain rule]\label{chain-rule}
		Let $\mathcal{H}$ be a Hilbert space, and let $\mathcal{V}$ and $\mathcal{E}$ be Banach spaces satisfying the Gelfand triple \eqref{eqn-gelfand}. Let $1<p<\infty$ and let $p'=\frac{p}{p-1}$. Suppose that
		\begin{align*}
			\boldsymbol{y},\boldsymbol{z}\in\mathrm{L}^2(0,T;\mathcal V)\cap\mathrm{L}^p(0,T;\mathcal E)
		\end{align*}
		and that, in $\mathcal D'(0,T;\mathcal{V}'+\mathcal{E}')$,
		\begin{align*}
			\boldsymbol{y}'=\boldsymbol{y}_1+\boldsymbol{y}_2,
			\ 
			\boldsymbol{z}'=\boldsymbol{z}_1+\boldsymbol{z}_2,
		\end{align*}
		where
		\begin{align*}
			\boldsymbol{y}_1,\boldsymbol{z}_1\in\mathrm{L}^2(0,T;\mathcal V'),
			\ 
			\boldsymbol{y}_2,\boldsymbol{z}_2\in\mathrm{L}^{p'}(0,T;\mathcal E').
		\end{align*}
		
		Then $\boldsymbol{y}$ and $\boldsymbol{z}$ admit representatives, still denoted by $\boldsymbol{y}$ and
		$\boldsymbol{z}$, such that
		$
		\boldsymbol{y},\boldsymbol{z}\in\mathrm{C}([0,T];\mathcal H).
		$
		Moreover, the map
		\begin{align*}
			t\mapsto (\boldsymbol{y}(t),\boldsymbol{z}(t))_{\mathcal H}
		\end{align*}
		is absolutely continuous on $[0,T]$, and
		\begin{align}
			\frac{\mathrm{d}}{\mathrm{d}t}(\boldsymbol{y}(t),\boldsymbol{z}(t))_{\mathcal H}
			=
			\langle \boldsymbol{y}_1(t),\boldsymbol{z}(t)\rangle_{\mathcal V',\mathcal V}
			+\langle \boldsymbol{y}_2(t),\boldsymbol{z}(t)\rangle_{\mathcal E',\mathcal E}
			+
			\langle \boldsymbol{z}_1(t),\boldsymbol{y}(t)\rangle
		\end{align}
			for a.e. $t\in[0,T]$.
		\end{theorem}
		
		\begin{proof}
			Let
			$
			\boldsymbol{y},\boldsymbol{z}\in\mathrm{L}^2(0,T;\mathcal V)\cap \mathrm{L}^p(0,T;\mathcal E)
			$
			and set
			$\psi(t):=(\boldsymbol{y}(t),\boldsymbol{z}(t))_{\mathcal H}.$
			By the generalized Lions-Magenes lemma, we may choose representatives
			$\boldsymbol{y},\boldsymbol{z}\in\mathrm{C}([0,T];\mathcal H)$
			which are also absolutely continuous as $\mathcal X'$-valued functions.
			For $\varepsilon>0$ and $0<t<T-\varepsilon$, define
			\begin{align*}
			D_\varepsilon \boldsymbol{y}(t)
			:=
			\frac{\boldsymbol{y}(t+\varepsilon)-\boldsymbol{y}(t)}{\varepsilon}
			\ \text{ and }
			D_\varepsilon \boldsymbol{z}(t)
			:=
			\frac{\boldsymbol{z}(t+\varepsilon)-\boldsymbol{z}(t)}{\varepsilon}.
			\end{align*}
			Since
			$
			\boldsymbol{y}'=\boldsymbol{y}_1+\boldsymbol{y}_2
			\ \text{ in }\ \mathcal D'(0,T;\mathcal X'),$
			we have
			\begin{align*}
			D_\varepsilon \boldsymbol{y}(t)=
			\frac1{\varepsilon}\int_t^{t+\varepsilon}\boldsymbol{y}_1(s)\d s
			+
			\frac1{\varepsilon}\int_t^{t+\varepsilon}\boldsymbol{y}_2(s)\d s \ \text{	in }\ \mathcal X',
			\end{align*}
			for a.e. $t\in[0,T]$. Similarly, we find 
			\begin{align*}
			D_\varepsilon \boldsymbol{z}(t)
			=
			\frac1{\varepsilon}\int_t^{t+\varepsilon}\boldsymbol{z}_1(s)\d s
			+
			\frac1{\varepsilon}\int_t^{t+\varepsilon}\boldsymbol{z}_2(s)\d s \ \text{	in }\ \mathcal X',
			\end{align*}
			for a.e. $t\in[0,T]$. By the standard translation property of Bochner integrals (\cite[Theoem 3.2]{Barbu1993}), we obtain 
			\begin{align*}
			\lim_{\varepsilon\downarrow0}
			\int_0^{T-\varepsilon}
			\left\|
			\frac1{\varepsilon}
			\int_t^{t+\varepsilon}\boldsymbol{y}_1(s)\d s-\boldsymbol{y}_1(t)
			\right\|_{\mathcal V'}^2\d t
			=0,
			\end{align*}
			and
			\begin{align*}
			\lim_{\varepsilon\downarrow0}
			\int_0^{T-\varepsilon}
			\left\|
			\frac1{\varepsilon}
			\int_t^{t+\varepsilon}\boldsymbol{y}_2(s)\d s-\boldsymbol{y}_2(t)
			\right\|_{\mathcal E'}^{p'}\d t
			=0.
			\end{align*}
			The analogous relations hold for $\boldsymbol{z}_1$ and $\boldsymbol{z}_2$.
			On the other hand, since
			\begin{align*}
			\boldsymbol{y}\in\mathrm{L}^2(0,T;\mathcal V)
			\cap\mathrm{L}^p(0,T;\mathcal E)
			\ \text{ and }
			\boldsymbol{z}\in\mathrm{L}^2(0,T;\mathcal V)
			\cap\mathrm{L}^p(0,T;\mathcal E),
			\end{align*}
			we also have
			\begin{align*}
			\lim_{\varepsilon\downarrow0}
			\int_0^{T-\varepsilon}
			\|\boldsymbol{y}(t+\varepsilon)-\boldsymbol{y}(t)\|_{\mathcal V}^2\,\d t=0
			\ \text{ and }\ 
			\lim_{\varepsilon\downarrow0}
			\int_0^{T-\varepsilon}
			\|\boldsymbol{y}(t+\varepsilon)-\boldsymbol{y}(t)\|_{\mathcal E}^p\,\d t=0.
			\end{align*}
			Similar results hold true for $	\boldsymbol{z}$ also. 
			Now,
			\begin{align*}
				\frac{\psi(t+\varepsilon)-\psi(t)}{\varepsilon}
				&=\left(D_\varepsilon \boldsymbol{y}(t),\boldsymbol{z}(t+\varepsilon)\right)+\left(\boldsymbol{y}(t),D_\varepsilon \boldsymbol{z}(t)\right)\nonumber\\&=
				\left\langle D_\varepsilon \boldsymbol{y}(t),\boldsymbol{z}(t+\varepsilon)\right\rangle_{\mathcal X',\mathcal X}
				+
				\left\langle \boldsymbol{y}(t),D_\varepsilon \boldsymbol{z}(t)\right\rangle_{\mathcal X',\mathcal X},
			\end{align*}
		for a.e. $t\in[0,T]$. Using the decomposition of $\mathcal X'$ into
			$\mathcal V'+\mathcal E'$, this becomes
			\begin{equation*}
			\begin{aligned}
				\frac{\psi(t+\varepsilon)-\psi(t)}{\varepsilon}
				&=
				\left\langle
				\frac1{\varepsilon}\int_t^{t+\varepsilon}\boldsymbol{y}_1(s)\d s,
				\boldsymbol{z}(t+\varepsilon)
				\right\rangle_{\mathcal V',\mathcal V}
				+
				\left\langle
				\frac1{\varepsilon}\int_t^{t+\varepsilon}\boldsymbol{y}_2(s)\d s,
				\boldsymbol{z}(t+\varepsilon)
				\right\rangle_{\mathcal E',\mathcal E}
				\\
				&\quad+
				\left\langle
				\boldsymbol{y}(t),
				\frac1{\varepsilon}\int_t^{t+\varepsilon}\boldsymbol{z}_1(s)\d s
				\right\rangle_{\mathcal V',\mathcal V}
				+
				\left\langle
				\boldsymbol{y}(t),
				\frac1{\varepsilon}\int_t^{t+\varepsilon}\boldsymbol{z}_2(s)\d s
				\right\rangle_{\mathcal E',\mathcal E}.
			\end{aligned}
			\end{equation*}
			Therefore, by the Cauchy-Schwarz inequality for the
			$\mathcal V$-$\mathcal V'$ terms and Hölder's inequality for the
			$\mathcal E$-$\mathcal E'$ terms,
			\begin{equation*}
			\begin{aligned}
				\lim_{\varepsilon\downarrow0}
				\int_0^{T-\varepsilon}
				\left|
				\frac{\psi(t+\varepsilon)-\psi(t)}{\varepsilon}
				-
				\Phi(t)
				\right|\d t
				=0,
			\end{aligned}
			\end{equation*}
			where
			\begin{equation*}
			\begin{aligned}
				\Phi(t)
				=&
				\langle \boldsymbol{y}_1(t),\boldsymbol{z}(t)\rangle_{\mathcal V',\mathcal V}
				+
				\langle \boldsymbol{y}_2(t),\boldsymbol{z}(t)\rangle_{\mathcal E',\mathcal E}
				+
				\langle \boldsymbol{z}_1(t),\boldsymbol{y}(t)\rangle_{\mathcal V',\mathcal V}
				+
				\langle \boldsymbol{z}_2(t),\boldsymbol{y}(t)\rangle_{\mathcal E',\mathcal E},
			\end{aligned}
			\end{equation*}
      for a.e. $t\in[0,T]$.
			Hence
			$\psi\in\W^{1,1}(0,T;\mathbb R)$
			and
			\begin{align*}
			\frac{\d}{\d t}(\boldsymbol{y}(t),\boldsymbol{z}(t))_{\mathcal H}
			=
			\langle \boldsymbol{y}_1(t),\boldsymbol{z}(t)\rangle_{\mathcal V',\mathcal V}
			+
			\langle \boldsymbol{y}_2(t),\boldsymbol{z}(t)\rangle_{\mathcal E',\mathcal E}
			+
			\langle \boldsymbol{z}_1(t),\boldsymbol{y}(t)\rangle_{\mathcal V',\mathcal V}
			+
			\langle \boldsymbol{z}_2(t),\boldsymbol{y}(t)\rangle_{\mathcal E',\mathcal E},
			\end{align*}
			for a.e. $t\in[0,T]$.
		\end{proof}

		\medskip\noindent
		\textbf{Acknowledgments:} 
The first author would like to thank the Council of Scientific and Industrial Research (CSIR), India, for providing financial support through a Junior Research Fellowship (JRF). The second author would like to thank the Ministry of Education (MoE), Government of India, for financial assistance. M.\,T. Mohan was supported by the National Board of Higher Mathematics (NBHM), Department of Atomic Energy, Government of India (Project No.\ 02011/13/2025/NBHM(R.P)/R\&D~II/1137). The third author would like to thank Dr.\ Kush Kinra, Department of Mathematics, Friedrich-Alexander-Universit\"at Erlangen-N\"urnberg (FAU), Germany, for valuable discussions.

		\medskip\noindent	\textbf{Declarations:} 
		
		\noindent 	\textbf{Ethical Approval:}   Not applicable 
		
		\noindent  \textbf{Competing interests: } The authors declare no competing interests. 
		
		\noindent  \textbf{Conflict of interest: }On behalf of all authors, the corresponding author states that there is no conflict of interest.
		
		\noindent 	\textbf{Authors' contributions:} All authors have contributed equally. 
		
		\noindent 	\textbf{Availability of data and materials:} Not applicable.

\end{document}